\documentclass[10pt,reqno]{amsart}
\usepackage{enumitem}
\usepackage{amscd}
\usepackage{amssymb}
\usepackage[all]{xy}
\usepackage{fge}
\newcommand{\moins}{\mathbin{\fgebackslash}}

\usepackage[textsize=tiny]{todonotes}



\RequirePackage[T1]{fontenc}
\RequirePackage{amsfonts,latexsym,amssymb}

\let\cal\mathcal

\usepackage{url}
\usepackage{bbm}
\usepackage{hyperref}

\newtheorem{theorem}[equation]{Theorem}
 \newtheorem{lemma}[equation]{Lemma}
 \newtheorem{proposition}[equation]{Proposition}
 \newtheorem{corollary}[equation]{Corollary}  
\newtheorem{conjecture}[equation]{Conjecture}

\theoremstyle{definition}
\newtheorem{definition}[equation]{Definition}

\newtheorem{remark}[equation]{Remark}

\theoremstyle{remark}
\newtheorem*{acknowledgments}{Acknowledgments}

\def\jcdot{\scriptscriptstyle\bullet}

\def\invlim{\mathop{\vtop{\ialign{##\crcr$\hfill{\lim}\hfil$\crcr
\noalign{\kern1pt\nointerlineskip}\leftarrowfill\crcr\noalign
{\kern -3pt}}}}\limits}
\def\dirlim{\mathop{\vtop{\ialign{##\crcr$\hfill{\lim}\hfil$\crcr
\noalign{\kern1pt\nointerlineskip}\rightarrowfill\crcr\noalign
{\kern -3pt}}}}\limits} 
\def\lomapr#1{\smash{\mathop{\relbar\joinrel\longrightarrow}\limits^{#1}}}
 \def\verylomapr#1{\smash{\mathop{\relbar\joinrel\relbar\joinrel\relbar\joinrel\longrightarrow}\limits^{#1}}}

\def\phi{\varphi} 
\def\epsilon{\varepsilon}

\newcommand{\ovk}{\overline{K} }

\newcommand{\dr}{\operatorname{dR} } 
  \newcommand{\syn}{\operatorname{syn} }

 \newcommand{\colim}{\operatorname{colim} }

 \newcommand{\proeet}{\operatorname{pro\acute{e}t} } 
 \newcommand{\eet}{\operatorname{\acute{e}t} }
  \newcommand{\Spa}{\operatorname{Spa} }

 \newcommand{\Hom}{{\rm{Hom}} }
  \newcommand{\uHom}{\underline{{\rm{Hom}} }}
  \newcommand{\Hhom}{{\cal Hom}} 
 \newcommand{\Ext}{\operatorname{Ext} }

\newcommand{\Gal}{\operatorname{Gal} }
\newcommand{\can}{ \operatorname{can} }
\newcommand{\synt}{ \operatorname{syn} }
 
\newcommand{\st}{\operatorname{st} }
  
 \newcommand{\coker}{\operatorname{coker} }  
    
 \newcommand{\crr}{\operatorname{cr} }
  \newcommand{\hk}{\operatorname{HK} }
   \newcommand{\LL}{\operatorname{L} }
     \newcommand{\FF}{\operatorname{FF} }

 \newcommand{\sff}{{\mathcal{F}}}

 \newcommand{\sg}{{\mathcal{G}}}

 \newcommand{\scc}{{\mathcal{C}}}

 \newcommand{\so}{{\mathcal O}}
 
 \newcommand{\se}{{\mathcal{E}}}

\newcommand{\sd}{{\mathcal{D}}}

 \newcommand{\wh}{\widehat}
   \numberwithin{equation}{section}

\def\R{{\mathrm R}}
\def\O{{\cal O}}
  \def\B{{\bf B}}
\def\Q{{\bf Q}} \def\Z{{\bf Z}}

\def\N{{\bf N}}
\def\O{{\cal O}}

\def\bdr{{\bf B}_{{\rm dR}}}

\def\epsilon{\varepsilon}

\def\EEE{${\cal E}_{\dr}$, ${\cal E}_{\rm FF}$, ${\cal E}_{\hk}$, ${\cal E}_{\proeet}$, ${\cal E}_{\syn}$}
\def\FFF{$X_{\rm FF}$, $X_{{\rm FF},S}$, $Y_{\rm FF}$, $Y_{{\rm FF},S}$, $Y_{{\rm FF},S}^I$, $Y_{{\rm FF},S}^{[u,v]}$}
\def\SYN{$\R\Gamma^{\B_{\crr}^+}_{\synt,?}$, $\R\Gamma^{\B^{[u,v]}}_{\synt,?}$, $\R\Gamma^{\FF}_{\synt}$, $\R\Gamma^{\B}_{\synt,?}$, $\R\Gamma^{[u,v]}_{\synt,?}$}
\def\HK{$\R\Gamma_{\hk}$, $\R\Gamma_{\hk}(-)\{r\}$, $\R\Gamma_{\hk}^I(-,r)$, $\R\Gamma_{\hk}^{[u,v]}(-,r)$, $\R\Gamma_{\hk}^{\B}(-,r)$, $\R\Gamma^{\mathbb B}_{\hk}(-,r)$}
\def\HKi{$\iota_{\hk}$}
\def\DR{$F^r\R\Gamma_{\dr,?}$, $F^r\R\Gamma_{\dr,?}(-/\B^+_{\dr})$, $F^i\R\Gamma_{\dr,?}(X_S/\B^{\prime})$, $\R\Gamma_{\dr,?}(-,r)$, $\R\Gamma^{{[u,v]}}_{\dr,?}(-,r)$, $\R\Gamma^{{\B}}_{\dr,?}(-,r)$, $\R\Gamma^{\mathbb B}_{\dr}(-,r)$ }
\def\ETA{${\R\Gamma}_{\proeet}$, ${\R\Gamma}_{{\proeet},c}$}
\def\ETB{${\mathbb R}_{\proeet}$, ${\mathbb R}_{{\proeet},c}$, ${\mathbb R}_{{\proeet},?}^{\rm alg}$, ${\mathbb H}^\bullet_{\proeet}$, ${\mathbb H}^\bullet_{\proeet,c}$}
\def\TVS{TVS, VS, BC}
\def\rtau{${\rm R}\tau_*$, ${\rm R}\tau'_*$, ${\rm R}\eta_*$, ${\rm R}\pi_*$}
\def\bD{${\mathbb D}$}
\def\CC{$C_{\Q_p}$, $C_K$}
\def\PERF{${\rm Perf}_C$, ${\rm Perf}_S$, ${\rm sPerf}_C$}
\def\AN{$-_{\rm an}$}
\def\BRING{$\B_{\crr}$, $\B_{\st}$, $\B_{\dr}$, $\B^I$, $\B^{I,+}$, $\B_{S}^I$, $\B^{I,+}_{S}$, $\B_S$, $\B_{S,\log}$, $\B_{S,\log}^I$, $\B_{S}^{[u,v]}$, $\B^{\FF}_{S}$, ${\mathbb B}^{[u,v]}$, ${\mathbb B}^{I}$, ${\mathbb B}$, ${\mathbb B}_{\rm dR}$, ${\mathbb B}_{\rm log}$}
\def\DER{$\sd(-_{\Box})$, $ \sd(\B^{\FF}_{S,{\rm an}})^{\phi}$, $ \sd({\B}^{\FF}_{S})^{\phi}$}

\def\TRACE{${\rm Tr}_{\rm coh}$, ${\rm Tr}_{\rm dR}$, ${\rm Tr}_{\bdr^+}$, ${\rm Tr}_X$, ${\rm Tr}_X^{[u,v]}$}
\def\KAPPA{$\kappa$, $\kappa_{\rm HT}$, $\kappa_{\rm dR}$, $\kappa^{[u,v]}$}
\def\BORD{$\partial D_C$, $\partial X$}
\def\Cbreve{$C$, $\breve{C}$}
\def\rhom{$\R\uHom$}

 \makeatletter
\let\@wraptoccontribs\wraptoccontribs
\makeatother

\numberwithin{equation}{section}
\makeindex
\begin{document}
\title[Duality  for  $p$-adic geometric pro-\'etale cohomology]
 {Duality  for $p$-adic geometric pro-\'etale cohomology}
 \author{Pierre Colmez} 
\address{CNRS, IMJ-PRG, Sorbonne Universit\'e, 4 place Jussieu, 75005 Paris, France}
\email{pierre.colmez@imj-prg.fr} 
\author{Sally Gilles}
\address{LMB, Universit\'e Marie et Louis Pasteur, 16 route de Gray, 25030 Besan\c{c}on CEDEX, France}
\email{sally.gilles@umlp.fr}
\author{Wies{\l}awa Nizio{\l}}
\address{CNRS, IMJ-PRG, Sorbonne Universit\'e, 4 place Jussieu, 75005 Paris, France}
\email{wieslawa.niziol@imj-prg.fr}
 \date{\today}
\thanks{P.C. and W.N.'s research was supported in part by the grant ANR-19-CE40-0015-02 COLOSS and by the Simons Foundation through the Simons Collaboration on Perfection in Algebra, Geometry, and Topology. S.G.'s research was partially supported by the National Science Foundation under Grant No. DMS-1926686. 
}
\maketitle
 \begin{abstract} We prove that $p$-adic geometric pro-\'etale cohomology of smooth partially proper rigid analytic varieties over $p$-adic fields seen in the category of Topological Vector Spaces satisfies a Poincar\'e duality as we have conjectured. This duality descends, via full faithfulness results of Colmez-Nizio{\l},  from a Poincar\'e duality for 
 solid quasi-coherent sheaves on the Fargues-Fontaine curve representing this cohomology. The latter duality is proved by passing, via comparison theorems, to analogous sheaves representing syntomic cohomology and then reducing to Poincar\'e duality for $\B^+_{\st}$-twisted Hyodo-Kato and filtered  $\B^+_{\dr}$-cohomologies that, in turn, reduce to Serre duality for smooth Stein varieties -- a classical result. 
\end{abstract}

\tableofcontents
\section{Introduction}
 Let $p$ be a prime and let $K$ be a discretely valued $p$-adic field
with  perfect residue field. 
Let $\ovk$ be an algebraic closure of $K$ and \index{C@\Cbreve}let $C=\wh{\ovk}$ be the $p$-adic completion of $\ovk$.  
 \subsection{Duality Conjecture}
 In \cite{CDN1,CDN5}, the authors showed that (a part of) the $p$-adic local Langlands correspondence for ${\rm GL}_2(\Q_p)$ can be realized in the $p$-adic (pro-)\'etale cohomology of the Drinfeld 
tower in dimension $1$. It involved the duals of the unitary representations of ${\rm GL}_2(\Q_p)$
provided by the $p$-adic local Langlands correspondence and not the representations themselves.
In contrast,
the classical local Langlands correspondence (not its dual)
is realized in the compactly supported $\ell$-adic cohomology rather than the usual one; 
this brought us to consider a possibility of duality for $p$-adic pro-\'etale cohomology. 
 
 The reason the authors  of loc. cit. 
worked with the usual cohomology was mainly due to their newly acquired competence
with comparison theorems for analytic varieties~\cite{CN1}. 
In retrospect, this was a lucky choice since the description of
geometric pro-\'etale cohomology with compact support is not as transparent as that
of usual cohomology. The compactly supported cohomology (even in dimension~$1$)
was not defined at that time, but we did some heuristic computations
assuming that it could be defined similarly
to de Rham cohomology with compact support~\cite{vdP}, and that (derived) 
comparison theorems would work for it as expected. 
The first results did not look very promising.

Here is the simplest of them. 
Let $D$ be the open disc of   dimension~$1$ over $K$.   Using syntomic methods,
we obtain that the only nontrivial cohomology groups are as \index{BORD@\BORD}follows:
\begin{align}\label{sing4}
& H^0_{\proeet}(D_C,\Q_p(1))\simeq \Q_p(1),\quad H^1_{\proeet}(D_C,\Q_p(1))\simeq \so(D_C)/C,\\
&  H^2_{{\proeet},c}(D_C,\Q_p(1))\simeq \Q_p\oplus \so(\partial D_C)/\so(D_C),\notag
\end{align}
where $\partial D_C$ denotes the ``boundary of $D_C$''. Since we have the isomorphisms 
(of topological $C$-vector spaces)
$$\so(D_C)/C\stackrel{\sim}{\to} \Omega^1(D_C),\quad \so(\partial D_C)/\so(D_C)\stackrel{\sim}{\to} H^1_c(D_C,\so), 
$$we see in \eqref{sing4} a Serre duality (of topological $C$-vector spaces)
$$
 \Omega^1(D_C)\simeq H^1_c(D_C,\so)^*
$$
as well as a simple $\Q_p$-duality (between $H^0_{\proeet}(D_C,\Q_p(1))$ and the $\Q_p$ appearing in 
$H^2_{{\proeet},c}(D_C,\Q_p(1))$) but they do not fit together into a simple  
Poincar\'e duality (since the degrees do not match, and the $C$-duality cannot be turned into
a $\Q_p$-duality as $[C:\Q_p]=\infty$). 

However, we realized (after quite a while) that, if we couple the above with the following 
computations\footnote{In \eqref{BC}  
the last nontrivial $\Ext$ group is generated by the fundamental exact sequence
 $$0\to {\Q}_p(1)\to {\mathbb B}^{+,\phi=p}_{\crr}\to {\mathbb G}_a\to 0,$$
where ${\mathbb B}^{+,\phi=p}_{\crr}$ is the  {\rm TVS} corresponding to  ${\B}^{+,\phi=p}_{\crr}$.}
 in the category of \index{TVS@\TVS}Banach-Colmez spaces ($\mathrm{BC}$'s  for short) and 
 assume that the results remain valid
in the bigger category of Topological Vector Spaces (${\rm TVS}$'s for short):
 \begin{align}\label{BC}
&  \uHom_{\rm TVS}({\Q}_p,\Q_p(1))\simeq {\Q}_p(1),\quad \underline{\Ext}^1_{\rm TVS}({\Q}_p,{\Q}_p(1))=0,\\
&    \uHom_{\rm TVS}({\mathbb G}_a,{\Q}_p(1))=0,\quad \underline{\Ext}^1_{\rm TVS}({\mathbb G}_a,{\Q}_p(1))\simeq {\mathbb G}_a,\notag\\
&  \underline{\Ext}^i_{\rm TVS}(M,N)=0, \quad i\geq 2,\ M,N \in \mathrm{BC},\notag
 \end{align}
as well as that ${\rm Ext}$-groups of tensor products with constant objects behave as if
everything was finite dimensional, we get abstract isomorphisms
(with $H^i_{?}:=H^i_{?}(D_C,\Q_p(1))$)
\begin{equation}\label{dua1}
 H^0_{\proeet}\simeq {\rm Hom}_{\rm TVS}(H^2_{\proeet,c},{\Q}_p(1))
\quad
 H^1_{\proeet}\simeq \Ext^1_{\rm TVS}(H^2_{\proeet,c},{\Q}_p(1)) 
\end{equation}
and an abstract exact sequence
\begin{equation}\label{dua2}
0\to \Ext^1_{\rm TVS}(H^1_{\proeet},\Q_p(1))\to
H^2_{\proeet,c}\to {\rm Hom}_{\rm TVS}(H^0_{\proeet},\Q_p(1))\to 0,
\end{equation}
which suggest a (derived) duality, both ways.
  
These examples  brought us to formulate the following conjecture (see \cite{COB}, \cite{CGN}):
 \begin{conjecture}\label{conj-main} 
Let $X$ be a smooth partially proper rigid analytic variety over $K$ of dimension~$d$. 
In the category of Topological Vector Spaces we have a natural \index{ETB@\ETB}quasi-isomorphism
 $$
 {\mathbb R}_{\proeet}(X_C,\Q_p)\simeq \R\Hhom_{\rm TVS}({\mathbb R}_{\proeet,c}(X_C,\Q_p(d))[2d],\Q_p).
 $$
 \end{conjecture}
 Here:
 \begin{enumerate}
 \item {\rm TVS}'s are $\underline{\Q}_p$-modules in the category of  topologically enriched\footnote{For the sake of the introduction we invite the reader to ignore the issue of enrichment and think of {\rm TVS}'s as just topological presheaves. See \cite{TVS} for precise definitions.} presheaves on  strictly totally disconnected spaces over $C$, 
\index{PERF@\PERF}denoted  ${\rm sPerf}_C$,  with values in solid abelian groups.
 \item  ${\mathbb R}_{\proeet}(X_C,\Q_p)$ is  \index{ETA@\ETA}a TVS defined by  $S\mapsto \R\Gamma_{\proeet}(X_S,\Q_p)$;  the topology on $\R\Gamma_{\proeet}(X_S,\Q_p)$  is canonically  inherited from the pro-\'etale site. 
 \item Pro-\'etale cohomology with compact support ${\mathbb R}_{\proeet,c}(X_C,\Q_p)$ is defined by $S\mapsto \R\Gamma_{\proeet,c}(X_S,\Q_p)$, where, 
 for an exhaustive covering $\{U_n\}$, $U_n\Subset U_{n+1}$, by quasi-compact open spaces, we \index{BORD@\BORD}set
 \begin{align*}
 \R\Gamma_{\proeet,c}(X_S,\Q_p) & :=[\R\Gamma_{\proeet}(X_S,\Q_p)\to \R\Gamma_{\proeet}(\partial X_S,\Q_p)],\\
 \R\Gamma_{\proeet}(\partial X_S,\Q_p) &:= \colim_n\R\Gamma_{\proeet}(X_S\moins U_{n,S},\Q_p),
 \end{align*}
 with the induced topology. 
  By \cite[Th.\,1.2(1), Th.\,2.22]{AGN}, we have $\R\Gamma_{\proeet,c}(X_C,\Q_p)\simeq\R\Gamma_{\eet,c,{\rm Hu}}(X_C,\Q_p)$, the Huber $p$-adic \'etale cohomology.
 \end{enumerate}

\begin{remark}
We state the duality only one way because of potential pathologies for extensions of Fr\'echet
(or Banach) spaces in the condensed world, but (\ref{dua1}) and (\ref{dua2}) suggest
 strongly that there should be
a duality both ways, at least in reasonable cases.
\end{remark}
 
 \subsection{The main result}  The main result of this paper is the following: 
 \begin{theorem} \label{main-1}Conjecture \ref{conj-main} holds.
 \end{theorem}
    Our strategy for the proof of Theorem \ref{main-1} follows the heuristic computations we have done on examples.
 The  foundational results  needed to do that mentioned in the previous section
   were proven in a series of papers by the authors and Piotr Achinger: Hyodo-Kato cohomology of rigid analytic and dagger varieties was defined  and studied in \cite{CN4}, $p$-adic comparison theorems were proven in \cite{CN4}, \cite{CN5}, \cite{SG}, compactly supported $p$-adic pro-\'etale cohomology and Hyodo-Kato dualities were studied in \cite{AGN}, the  properties of Topological Vector Spaces were studied in \cite{TVS} (that they satisfy the expected duality for BC's was derived there from a result of the same  type due to Ansch\"utz-Le Bras in the category of Vector Spaces\footnote{That is, the $\infty$-derived category of $\underline{\Q}_p$-modules on the site of perfectoid affinoids over $C$ equipped with the pro-\'etale topology.} \cite{ALB}). 
   
    To prove Theorem \ref{main-1}, we start by passing from pro-\'etale cohomology to syntomic cohomology.
 Recall that the latter is built from the Hyodo-Kato part that records the mod $p$ information, and the filtered de Rham part that records the characteristic zero information; the two parts are connected via the Hyodo-Kato morphism.
 It is part of the standard yoga of $p$-adic Hodge theory for algebraic varieties that, when   dealing with syntomic cohomology one should work as long as possible separately with the Hyodo-Kato and the de Rham parts and glue them together only at the last moment.\footnote{A technique inherited from Beilinson-Deligne cohomology (see \cite{Bei}). For how this works for syntomic cohomology in the algebraic setting  see \cite{Ban}, \cite{DN}.} In the perfectoid world this separation can be obtained geometrically by representing syntomic cohomology by a quasi-coherent sheaf on the Fargues-Fontaine curve.
 This sheaf will have (completed) stalks equal to twisted Hyodo-Kato cohomology at all points outside of $\infty$ and at $\infty$ it will be equal to the filtered $\B^+_{\dr}$-cohomology.
 Now Hyodo-Kato duality (inherited itself from de Rham duality) and filtered de Rham duality (see \cite{AGN})  yield a duality on the Fargues-Fontaine curve. Taking derived global sections of this duality, via full faithfulness results from \cite{TVS},   yields a duality on the level of {\rm TVS}'s.  
   \subsection{A Corollary} \label{before} Before reviewing our proof of Theorem \ref{main-1} we will state an implication and sketch its proof. It contains many of the essential elements of the proof of the main theorem.  
 \begin{corollary} \label{main-cor} Let $X$ be a smooth Stein variety over $K$ of dimension $d$. 
Let $i\geq 0$. There is a natural short exact sequence \index{ETB@\ETB}of {\rm TVS}'s
$$ 0\to {\mathcal Ext}^1_{\rm TVS}({\mathbb H}^{2d-i+1}_{\proeet,c}(X_C,\Q_p(d)),\Q_p)\to {\mathbb H}^i_{\proeet}(X_C,\Q_p)\to{\mathcal Hom}_{\rm TVS} ({\mathbb H}^{2d-i}_{\proeet,c}(X_C,\Q_p(d)),\Q_p)\to 0$$
 \end{corollary}
 We note that the term on the right is almost constant (see the computations \eqref{BC}). 
   Since we have  the internal $\R\Hom$ spectral sequence, it suffices to show  that
   \begin{equation}\label{vanishing-0}
   {\mathcal Ext}^a_{\rm TVS}({\mathbb H}^{b}_{\proeet, c}(X_C,\Q_p),\Q_p)=0,\quad a\geq 2. 
   \end{equation}
   Using  syntomic comparison theorems we can trivialize the {\rm TVS}-structure on pro-\'etale cohomology. That is, for $r\geq 2d$, we have a long exact sequence of {\rm TVS}'s
   $$
   \cdots\to \mathbb{DR}^{b-1}_c(X_C,r)\to {\mathbb H}^b_{\proeet,c}(X_C,\Q_p(r))\to \mathbb{HK}^b_c(X_C,r)\to \mathbb{DR}^b_c(X_C,r)\to\cdots,
   $$
   where we set
   \begin{align}\label{hk-dr}
   \mathbb{HK}^b_c(X_C,r) & := (H^b_{\hk,c}(X_C)\otimes^{\LL_{\Box}}_{\breve{C}}{\mathbb B}^+_{\st})^{\phi=p^r}\\
   \mathbb{DR}^b_c(X_C,r) & :=H^{b-d}(\cdots\to H^d_c(X,\Omega^j)\otimes^{\LL_{\Box}}_K({\mathbb B}^+_{\dr}/t^{r-j})\to H^d_c(X,\Omega^{j+1})\otimes^{\LL_{\Box}}_K({\mathbb B}^+_{\dr}/t^{r-j-1})\to \cdots)\notag
   \end{align}
  \index{C@\Cbreve}Here   $\breve{C}={\rm Frac}(W(\overline{k}))$, where $k$ is the residue field of $K$. 
We note that the {\rm TVS}-structure in \eqref{hk-dr} comes solely from the period presheaves.  
  
   To show \eqref{vanishing-0}, modulo a boundary case, it suffices to show it for  
the Hyodo-Kato and de Rham parts separately.
 For the Hyodo-Kato part, passing via a limit argument to overconvergent quasi-compact opens, we may assume that the ranks of  Hyodo-Kato groups are finite.
  But then $(H^b_{\hk,c}(X_C)\otimes^{\LL_{\Box}}_{\breve{C}}{\mathbb B}^+_{\crr})^{\phi=p^r}$ is a BC and we know that the $\Ext$-groups for those vanish in degrees at least 2.
 So far there were no functional analytic difficulties. They appear when we need to show that 
   $$
   {\mathcal Ext}^a_{\rm TVS}(H^d_c(X,\Omega^i)\otimes^{\Box}_K{\mathbb G}_a,\Q_p)=0,\quad a\geq 2.
   $$
   But now the space $H^d_c(X,\Omega^i)$ is of compact type; thus, assuming $K=\Q_p$ for simplicity, we have
   $$
   {\mathcal Ext}^a_{\rm TVS}(H^d_c(X,\Omega^i)\otimes^{\Box}_K{\mathbb G}_a,\Q_p)\simeq H^d_c(X,\Omega^i)^*\otimes^{\Box}_{\Q_p}{\mathcal Ext}^a_{\rm TVS}({\mathbb G}_a,\Q_p),
   $$
   which vanishes in the required range (by the BC computations), as wanted. 
 \subsection{Duality for $p$-adic geometric pro-\'etale cohomology on the  Fargues-Fontaine curve}
 We  first show  that  the $p$-adic geometric pro-\'etale cohomology  seen  as living on the Fargues-Fontaine curve satisfies a Poincar\'e duality. Recall that the $p$-adic geometric pro-\'etale cohomology of a smooth partially proper rigid analytic variety $X$  over $K$  can be represented 
by  a solid quasi-coherent sheaf on the Fargues-Fontaine curve, i.e.,
  the pro-\'etale cohomology can be computed \index{E@\EEE}as
 $$\R\Gamma_{\proeet}(X_C,\Q_p)\simeq \R\Gamma(X_{\FF},\se_{\proeet}(X_C, \Q_p)), $$  
for a (nuclear) solid quasi-coherent  sheaf $\se_{\proeet}(X_C,\Q_p)$ on the Fargues-Fontaine 
\index{FFF@\FFF}curve $X_{\FF}:=X_{\FF,C^{\flat}}$ defined using relative period sheaves. 
Similarly,  geometric compactly supported pro-\'etale cohomology 
 $\R\Gamma_{\proeet,c}(X_C,\Q_p)$ can be represented by a solid quasi-coherent sheaf $\se_{\proeet,c}(X_C, \Q_p)$ on $X_{\FF}$. See Section \ref{sing1}  for the definitions. 
 
  Via comparison theorems, we see that, if $?\in\{- ,c\}$,
 $$\se_{\proeet,?}(X_C, \Q_p(r))\simeq \se_{\synt,?}(X_C, \Q_p(r)),\quad  r\geq 2d,$$
   where $d$ is the dimension of $X$ and  $\se_{\synt,?}(X_C, \Q_p(r))$ 
  is  the syntomic cohomology sheaf (a solid quasi-coherent sheaf on the Fargues-Fontaine curve representing syntomic cohomology; see  Section \ref{sing2} for a definition).
 This is equivalent to proving a comparison theorem between corresponding  Frobenius equivariant sheaves on the Fargues-Fontaine punctured disc $Y_{\FF}$, which amounts to untwisting Frobenius from classical comparison theorems.
 Luckily for us, the proofs of comparison theorems in \cite{CN4} and \cite{AGN}  do actually (implicitly) prove the untwisted versions we want (see Corollary \ref{hot1K} and the accompanying definitions for the notation): 
  \begin{theorem}{\rm (${\mathbb B}^I$-comparison theorem)} 
Let $X$ be a smooth partially proper variety over $K$.  Let $r\geq 0$ and 
 let $I\subset (0,\infty)$ be a compact interval with rational endpoints\footnote{Containing the fixed intervals from Section \ref{fixed}}.  We have a natural (in $X$ and $S$) Frobenius-compatible
  quasi-isomorphism 
in $\sd(\B^I_{S^{\flat},\Box})$: 
\begin{align*}
\tau_{\leq r}\R\Gamma_{\proeet,?}(X_S,{\mathbb B}^{I})(r) & \simeq \tau_{\leq r}[\R\Gamma^{{I}}_{\hk,?}(X_S,r)\lomapr{\iota^I_{\hk}} \R\Gamma_{\dr,?}^{I}(X_S,r)].
\end{align*}
\end{theorem}

  Recall that classical syntomic cohomology is built from $(\phi,N)$-eigenspaces of $\B^+_{\st}$-twisted Hyodo-Kato cohomology and from filtered $\B^+_{\dr}$-cohomology. Representing it (in a stable range) by the sheaf $\se_{\synt,?}(X_C, \Q_p(r))$ on the Fargues-Fontaine curve separates these terms: heuristically speaking,  the (completed) stalks of $\se_{\synt,?}(X_C, \Q_p(r))$ at points outside $\infty$ are $N$-eigenspaces of $\B^+_{\st}$-twisted Hyodo-Kato cohomology and the (completed) stalk at $\infty$ is the $r$-th filtration level of 
 $\B^+_{\dr}$-cohomology. 
 
 Now, the   stalk cohomology sheaves   satisfy Poincar\'e duality: Poincar\'e duality for Hyodo-Kato cohomology reduces, via the Hyodo-Kato isomorphism, to that for de Rham cohomology and Poincar\'e duality for filtered de Rham cohomology, in turn, reduces to Serre duality for smooth Stein varieties -- a classical result (see \cite{AGN} and Section \ref{sing3} for details).
  These dualities are  inherited by the sheaves $\se_{\synt,?}(X_C, \Q_p(r))$, for $r\geq 2d$,  and then by the  sheaves $\se_{\proeet,?}(X_C, \Q_p(r))$ yielding  the second main result of this paper:
 \begin{theorem} 
{\rm (Poincar\'e duality for pro-\'etale sheaves)} \label{first1}  
We have a natural, Galois equivariant,  quasi-isomorphism in ${\rm QCoh}(X_{\FF})$
 \begin{equation}\label{first}
 \se_{\proeet}(X_C, \Q_p)\stackrel{\sim}{\to} \R\Hhom_{{\rm QCoh}(X_{\FF})}(\se_{\proeet,c}(X_C, \Q_p(d))[2d],\so).
 \end{equation}
 \end{theorem}

 The proof of the theorem does not proceed as sketched above though,
 due to the difficulties of passing to stalks in the theory of solid quasi-coherent sheaves.
 Instead we argue in a similar vein with $\phi$-modules on the $Y_{\FF}$-curve.
 In a side remark,
 we sketch an alternative proof of Theorem \ref{first1}
 that, instead of passing to the $Y_{\FF}$-curve, uses dual modifications.
 
  An analogous argument, with splitting into Hyodo-Kato and de Rham terms, yields a K\"unneth formula:
 \begin{theorem}{\rm(K\"unneth formula)} 
\label{first2} Let $X,Y$ be smooth partially proper varieties over $K$. 
 Then the  canonical  \index{KAPPA@\KAPPA}map 
 $$
\kappa:\quad   \se_{\proeet}(X_C,\Q_p)\otimes^{\LL}_{\so} \se_{\proeet}(Y_C,\Q_p) \to \se_{\proeet}((X\times_K Y)_C,\Q_p)
 $$
 is a quasi-isomorphism in ${\rm QCoh}(X_{\FF})$. 
 \end{theorem}
\begin{remark} In Theorem \ref{first1} and Theorem \ref{first2}, we can replace $C$, functorially, with any affinoid perfectoid over $C$. 
\end{remark}
\subsection{Descent to {\rm TVS}'s}
Finally, to prove Theorem \ref{main-1},  we need to  descend the duality \eqref{first} to the ``real'' world, which for us is the world of Topological Vector Spaces. 
We apply the projection \index{rtau@\rtau}functor $$\R\tau_*: {\rm QCoh}(X_{\FF})\to {\rm TVS},$$ the derived global sections functor from \cite{TVS},  to the duality on the Fargues-Fontaine curve \eqref{first}  and, since
$$
\R\tau_*\se_{\proeet,?}(X_C,\Q_p)\simeq {\mathbb R}_{\proeet,?}(X_C,\Q_p),\quad \R\tau_*\so\simeq\Q_p,
$$
we reduce to showing that the canonical map
$$
\R\tau_*\R\Hhom_{\rm QCoh(X_{\FF})}(\se_{\proeet,c}(X_C,\Q_p),\so) {\to}\R\Hhom_{\rm TVS}(\R\tau_*\se_{\proeet,c},\R\tau_*\so)
$$
is a quasi-isomorphism. 
But this  full faithfulness result can be reduced by an argument similar to the one used in the proof of Corollary \ref{main-cor} to full faithfulness for perfect complexes on the Fargues-Fontaine curve and this was proven in \cite{TVS}. 
\begin{remark} \label{rks}(1)({\em Algebraic Poincar\'e Duality}) The duality in Conjecture \ref{conj-main} has an algebraic version in the category of Vector Spaces (see Corollary \ref{VSS-duality} for the statement). It is deduced from Theorem \ref{main-1}  via full faithfulness results from \cite{TVS}. \vskip2mm

 (2)  ({\em Arithmetic Duality})  Our Conjecture \ref{conj-main} has an arithmetic version, i.e.,  for 
$p$-adic arithmetic pro-\'etale cohomology (see \cite{COB}, \cite{CGN}). The statement is  much simpler: it is a  Poincar\'e duality in the category of topological vector spaces that yields a nonderived version working in ``both directions''. 
In~\cite[Th. 1.1]{CGN}, we have proved this conjecture for dagger curves over~$K$ via relatively down to earth techniques ($p$-adic comparison theorems, Serre duality, reciprocity laws via $(\phi, \Gamma)$-modules). The general case was derived from Theorem \ref{first1} via Galois descent by Zhenghui Li in \cite{ZL}. 
 \end{remark}

\subsection{The story of this paper and related work}  
 The proof of duality presented in this paper is quite simple.
 But a lot of foundational work went into setting up the right formalism for this to be the case.
 Here is the story of how this developed. 

We started working on this project in the spring of 2020 when the computations for the open ball suggested
that there could be a Poincar\'e duality for $p$-adic geometric pro-\'etale cohomology
of rigid analytic spaces provided one could mix
the degrees of cohomology and combine $\Q_p$-duality and $C$-duality.
  This suggested that there could
be a duality in the $\mathrm{BC}$-category, by using ${\rm RHom}(-,\Q_p)$ instead of ${\rm Hom}(-,\Q_p)$; for this to
work one needed the vanishing of ${\rm Ext}^i$'s for $i\geq 2$ in the $\mathrm{BC}$-category (and in fact in
a bigger category containing the {\rm TVS}'s appearing in our comparison theorems with syntomic cohomology).
We discussed  this with Fargues and Le Bras in Oberwolfach at the first post-covid conference in
July 2020.
  The required  vanishing in the {\rm VS}-category was proved the following year  by Ansch\"utz and Le Bras~\cite{ALB}
 by a reduction to a theorem of  Breen
(we realized, more recently, that these vanishings are elementary in the $ \mathrm{BC}$-category itself).
 We needed a version of this result for {\rm TVS}'s but it was a strong indication that what we wanted could be true. 

We were at the time in the middle of writing~\cite{CN4,CN5} which contained part
of the foundational tools needed for a Poincar\'e duality 
(definition of geometric Hyodo-Kato cohomology, geometrization 
of $p$-adic comparison theorems for usual cohomology), and
we started considering their compact support avatars~\cite{AGN}. 
Concurrently, we did some extra computations for analytifications of
algebraic curves, which again pointed strongly towards the existence of a Poincar\'e duality
in the geometric and arithmetic cases (the computations were more involved in the arithmetic case,
but the groups that were appearing looked much more manageable, belonging to the usual
world of $\Q_p$-topological vector spaces).
By the time of the workshop ``Non-Archimedean Geometry and Applications'' of February 2022, in Oberwolfach, we had a conjecture~\cite{COB} and a strategy
that seemed to work well in examples, starting from our geometrized comparison theorem, and reinterpreting
syntomic cohomology as quasi-coherent $\varphi$-equivariant cohomology on the Fargues-Fontaine $Y$-curve.
We presented these results at the workshop and were quite excited to discover that there were
 two other talks dealing with Poincar\'e
duality: one by Zavyalov establishing~\cite{Zav} 
Poincar\'e duality for proper analytic varieties over $C$, and
one by Mann, developing~\cite{Mann} a 6-functor formalism for $\O^+/p$-local systems from which he
could also deduce Poincar\'e duality for proper analytic varieties over $C$ (note that, for
proper varieties, the pro-\'etale cohomology groups -- for constant coefficients --
are finite dimensional $\Q_p$-vector spaces, and there
is no need to consider $\mathrm{BC}$-duality). 

Since arithmetic duality only involved familiar objects, we decided that it would be wise to start
with the proof of our conjecture for arithmetic duality for curves~\cite{CGN}, but in the end we had to use
the condensed formalism to handle functional analytic questions arising in topological dualities.
 For the geometric duality almost all the tools we needed were at hand: luckily for us, Andreychev \cite{And21} developed the theory of solid quasi-coherent sheaves that created a framework in which to express the duality on the Fargues-Fontaine curve that we envisaged.
That was sufficient to prove a duality on the Fargues-Fontaine curve.
 What was missing to carry out our strategy in full was a condensed version of {\rm TVS}'s from \cite{CN5} 
and the $\Ext$-vanishing we mentioned above.
 That took us longer than expected mostly because we experimented with various possible definitions, settling in the end on an almost verbatim translation.
 And, finally, in \cite{TVS}, we deduced the vanishing in {\rm TVS}'s from the one in {\rm VS}'s via a full faithfulness result. 
In the meantime, Zhenghui Li~\cite{ZL} proved in his thesis, much to our surprise\footnote{We were rather expecting a Galois descent from {\rm TVS}'s.}, that 
the duality at the level of the Fargues-Fontaine curve was enough to deduce the existence
of an arithmetic duality in general, which prompted us to put out~\cite{CGN2}, a preliminary version of this paper.

\subsubsection{Related work}  In a related work,  Ansch\"utz, Le Bras, and  Mann (see \cite{ALBM}) followed a different path towards the
proof of Conjecture \ref{conj-main} in the {\rm VS}-form (see Remark \ref{rks}).
 They also proceed in two steps: the first step is a duality on the Fargues-Fontaine curve, which is a byproduct of the  6-functor formalism for solid quasi-coherent sheaves on the Fargues-Fontaine curve they have developed; the second step is a descent to the world of {\rm VS}'s.
 Their first step is very different in nature and techniques from ours and includes coefficients.
 On the other hand, the second step, while it passes through solid sheaves of Fargues-Scholze instead 
of our {\rm TVS}'s,  is very similar to ours: it relies on the Hyodo-Kato comparison theorems -- the deepest part of the $p$-adic comparison theorems 
via syntomic cohomology of~\cite{CN1,CDN3,CN4,CN5} (or their versions using cohomology of period sheaves
of~\cite{GB2}) --
for   dagger varieties and properties of their Hyodo-Kato cohomology (and filtered $\B^+_{\dr}$-cohomology) to control the functional analytic properties of pro-\'etale cohomology sheaves on the Fargues-Fontaine curve to be able to apply full faithfulness results (akin to  the arguments sketched  in  Section \ref{before}). 

  Finally, we would like to mention a different approach to duality theorems developed recently by  Shizhang Li, Reinecke,  and Zavyalov (see \cite{LRZ}), which works for proper smooth rigid analytic varieties (and their relative incarnations) and, after some modifications, can be transferred to the Fargues-Fontaine curve yielding, after descending to {\rm VS}'s,  a version of duality allowing $\Q_p$-local systems (see \cite{LNRZ}).

\begin{acknowledgments}
We have profited from the mathematical  generosity of many of our colleagues while writing this paper.
 In particular we would like to thank 
Piotr Achinger, Johannes Ansch\"utz, Guido Bosco,   Dustin Clausen, Gabriel Dospinescu, David Hansen, Shizhang Li, Zhenghui Li, Lucas Mann, Akhil Mathew, Juan Esteban Rodriguez Camargo,  Peter Scholze, Bogdan Zavyalov, and Mingjia Zhang  for many helpful comments and discussions concerning the content of this paper and for listening patiently to our expositions of the work in progress.  We are grateful to the referee for a very careful reading of the manuscript and many comments that improved the paper. 

  During the revision of this paper, we used GPT-6 (OpenAI) and Claude Opus 5.5 (Anthropic) to assist with proofreading, checking references, and reviewing mathematical arguments and their exposition. No AI tools were used in preparing the original submission. 
  
Parts of this paper were written during the first and  third authors' stay at the Hausdorff Research Institute for Mathematics in Bonn, in the Summer of 2023, and at IAS at Princeton, in the Spring of 2024. We would like to thank these institutes for their support and hospitality. The second author would like to thank the MPIM of Bonn and the IAS of Princeton for their support and hospitality during the academic years 2022-2023 and 2023-2024, respectively. 
\end{acknowledgments}

 \subsubsection*{Notation and conventions.}
 Let $p$ be a prime and let $K$ be a complete discrete valuation field with a perfect residue field, of mixed characteristic. 
 Let $\so_K$ be the ring of integers in~$K$, and $k$ be its
residue field. 
Let $W(k)$ be the ring of Witt vectors of $k$ and let $F$ be its
fraction field (i.e., $W(k)=\so_F$).   

Let $\ovk$ be an algebraic closure of $K$ and let $\so_{\ovk}$ denote the integral closure of $\so_K$ in $\ovk$. Let $C=\wh{\ovk}$ be the $p$-adic completion of $\ovk$.  
Set $\sg_K=\Gal(\overline {K}/K)$ and 
let $\phi$ be the absolute
Frobenius on $W(\overline {k})$. \index{C@\Cbreve}Let $\breve{C}={\rm Frac}(W(\overline{k}))$. 
 
  We will denote \index{BRING@\BRING}by $ \B_{\crr}, \B_{\st},\B_{\dr}$ the crystalline, semistable, and  de Rham period rings of Fontaine. 

 All rigid analytic spaces and dagger spaces considered will be over $K$ or $C$;  we assume that they are separated, taut, and countable at infinity.  Huber pairs will always be sheafy. The category of affinoid perfectoid spaces over an affinoid perfectoid space $S$ over $C$ will be denoted 
\index{PERF@\PERF}by  ${\rm Perf}_S$. 

We will use condensed mathematics as developed in  \cite{Sch19}, \cite{Sch20}.   
 We fix an implicit cut-off cardinal $\kappa$ (in the sense of \cite[Sec.\,4]{SchD}), 
and assume all our perfectoid spaces and condensed sets to be $\kappa$-small. 

If $L=\Q_p,K,C$, we will denote by  \index{CC@\CC}$C_L$ the category of locally convex topological vector spaces over $L$.
 
 We will use the bracket notation for certain limits:
  $[C_1\stackrel{f}{\to} C_2]$ denotes the mapping fiber of $f$ and we set
   \[ \left[\vcenter{\xymatrix @C=1cm@R=4mm{ 
C_1\ar[r]^-{f_1} \ar[d] & K_1 \ar[d] \\
C_2 \ar[r]^-{f_2} & K_2}} \right] := 
\big[[C_1 \xrightarrow{f_1} K_1] \to [C_2 \xrightarrow{f_2} K_2]\big]. \]

 \section{Quasi-coherent sheaves on the Fargues-Fontaine curve} 
Here, we will review  briefly basic facts concerning quasi-coherent sheaves on the Fargues-Fontaine curve. This is  partly based on \cite{And21}, \cite{And23},  and \cite[Sec.\,6.2]{GB2}.

 \subsection{Fargues-Fontaine curve} \label{red1}
Recall the definition of the relative Fargues-Fontaine curve (see \cite[Lecture 12]{SW}). Let  $S = {\rm Spa}(R,R^+)$ be an affinoid perfectoid space over the finite field ${\mathbf F}_p$, and choose a pseudo-uniformizer $\varpi\in R^+$.  \index{FFF@\FFF}Let
$$Y_{\FF,S} := {\rm Spa}(W(R^+),W(R^+)) \moins V (p[\varpi])
 $$
 be the relative mixed characteristic punctured unit disc. 
It is an analytic adic space over $\Q_p$. The
Frobenius on $R^+$ induces the Witt vector Frobenius and hence  a Frobenius  $\phi$ on $Y_{\FF,S}$ with  free and
totally discontinuous action. The Fargues-Fontaine curve relative to $S$
(and $\Q_p$) is defined \index{FFF@\FFF}as 
$$X_{\FF,S} := Y_{\FF,S}/\phi^{\Z}.
 $$
 
   For an interval $I = [s, r]\subset  (0,\infty)$ with rational endpoints, we have the open \index{FFF@\FFF}subset
$$Y_{\FF,S}^I := \{|\cdot|: |p|^r \leq |[\varpi]| \leq  |p|^s\} \subset  Y_{\FF,S}. 
 $$
  It is a rational open subset of ${\rm Spa}(W(R^+),W(R^+))$ hence 
an affinoid \index{BRING@\BRING}space,
$$
Y_{\FF,S}^I:={\rm Spa}(\B_{S}^I,\B^{I,+}_{S}).
$$ One can form $X_{\FF,S}$ as the quotient of $Y_{\FF,S}^{[1,p]}$ via the identification 
$\phi: Y_{\FF,S}^{[1,1]}\stackrel{\sim}{\to} Y_{\FF,S}^{[p,p]}$. 
      If $S={\rm Spa}(C^{\flat},\so_{C^{\flat}})$,  we will \index{BRING@\BRING}\index{FFF@\FFF}write $Y_{\FF},X_{\FF}, Y_{\FF}^I,
 \B^{I}, \B^{I,+}$. 
         
 We will denote by $x_{\infty}$ the $(C,\so_C)$-point of the curve $X_{\FF}$ corresponding to Fontaine's map
$\theta : W(\so_{C^{\flat}}) \to  \so_C $, by $y_{\infty}$ -- the corresponding point on $Y_{\FF}$,  and by
$\iota_{\infty} : {\rm Spa}(C,\so_C) \to  T_{\FF}$, $T=X,Y$, the corresponding  closed immersions. More generally,  if  $S$ is the tilt of a perfectoid
space $S^{\sharp}$ over ${\rm Spa} (\Q_p)$, there is an induced closed immersion $\theta : S^{\sharp} \hookrightarrow Y_{\FF,S}$ which
is locally given by Fontaine's map $ \theta : W (R^+) \to  R^{\sharp, +}$. We will denote by  $\iota_{\infty}: S^{\sharp}\stackrel{\theta}{\to} T_{\FF,S}$  the induced 
 closed immersions and by $y_{\infty}, x_{\infty}$, the corresponding divisors.

    We \index{BRING@\BRING}set 
   $$
   \B_S:=\lim_{I\subset (0,\infty)}\B_{S}^I,
   $$
   where $I$ varies over all the compact subintervals of $(0,\infty)$ with rational endpoints. For $S$ over $C^{\flat}$, we take $\varpi$ to be the image of $p^{\flat}\in\so_{C^{\flat}}$, defined by a fixed compatible system of $p$-power roots of $p$ in $C$. In this setting we will denote \index{BRING@\BRING}by 
   $\B_{S,\log}$  the log-crystalline period ring (see \cite[Sec.\,10.3.1]{FF18}). We have $\B_S[U]\stackrel{\sim}{\to}\B_{S,\log}, $ $U\mapsto \log([p^{\flat}]/p)$, with
   $\phi(U)=pU, \sigma(U)=U+\log[\sigma(p^{\flat})/p^{\flat}]$, for $\sigma\in\sg_K$, and $N=-d/dU$. We \index{BRING@\BRING}define $\B_{S,\log}^I$ in a similar manner.

 \subsection{Quasi-coherent sheaves on the Fargues-Fontaine curve} We will present now quasi-coherent sheaves on $X_{\FF}$ as $\phi$-modules on a convenient chart of $Y_{\FF}$. 
  \subsubsection{Solid quasi-coherent sheaves} 
We start with a brief survey of solid quasi-coherent sheaves. 
Let  $Y$ be an analytic adic space over $\Q_p$. We denote \index{Qcoh@${\rm QCoh}(-)$}by ${\rm QCoh}(Y )$ the $\infty$-category of solid quasi-coherent sheaves
on $Y$, and \index{Nuc@${\rm Nuc}(-)$}by  ${\rm Nuc}(Y )$ the full $\infty$-subcategory of solid nuclear sheaves on $Y$.
  See \cite{And21}, \cite{And23} for the  definitions of these categories and their basic properties.
 We will  often drop the word ``solid'' if this does not cause confusion.
 If $Y={\rm Spa}(R,R^+)$, then we have an \index{DER@\DER}equivalence \cite[Th.\,1.6]{And21}
 \begin{equation}\label{eq1}
  {\rm QCoh}(Y)\simeq \sd((R,R^+)_{\Box}),
  \end{equation}
  where the latter is the derived category of solid $(R,R^+)$-modules, i.e., modules over the analytic ring $(R,R^+)_{\Box}$.  In what follows, if this does not cause confusion, 
  we will \index{AN@\AN}write 
$$R_{\rm an}:=(R,R^+)_{\Box}.$$
 For a general $Y$, the category $ {\rm QCoh}(Y)$ is obtained by  gluing the categories $\sd((R,R^+)_{\Box})$
  in the analytic topology (see \cite[Th.\,4.1]{And21}).

   \index{Perf@${\rm Perf}(-)$}By ${\rm Perf}(Y )$, we denote
{\it the full $\infty$-subcategory of perfect complexes on $Y$}; 
that is,  complexes which locally for the analytic topology are quasi-isomorphic to a bounded complex of finite, locally free $\so_Y$-modules. If $Y={\rm Spa}(R,R^+)$ is affinoid, then the natural functor 
$$
{\rm Perf}(R )\to {\rm Perf}(Y )
$$
is an equivalence by \cite[Th.\,5.3]{And21}, where the left-hand side denotes the $\infty$-category of perfect complexes of $R$-modules (i.e., bounded complexes of finite projective $R$-modules).

   The $\infty$-categories ${\rm QCoh}(Y )$, ${\rm Nuc}(Y )$, and $ {\rm Perf}(Y )$ are (compatibly) symmetric monoidal. We use the cut-off cardinal $\kappa$ fixed in our notation and conventions (see \cite[Sec.\,4]{SchD}). With this convention, ${\rm QCoh}(Y )$ and ${\rm Nuc}(Y )$ are presentable and closed symmetric monoidal. The $\infty$-category ${\rm Perf}(Y )$ is stable and symmetric monoidal.

 \begin{remark}The categories ${\rm QCoh}(Y )$, ${\rm Nuc}(Y )$, and $ {\rm Perf}(Y )$ can be defined in a more general setting, where $Y={\rm Spa}(R,R^+)$ is a pair such that $R$ is a complete Huber ring and $R^+\subset R^0$ is an arbitrary subring (see \cite[Sec.\,3.3]{And21} for details). We will most often use the case when $R^+=\Z$. 
 \end{remark}

   \subsubsection{Quasi-coherent $\phi$-sheaves on $Y_{\FF}$}\label{fixed}
The $\infty$-category of quasi-coherent $\phi$-equivariant sheaves  over $Y_{\FF,S}$ (in short: $\phi$-sheaves  over $Y_{\FF,S}$) is   the equalizer
$$
{\rm QCoh}(Y_{\FF,S})^{\phi} := {\rm eq}\big(\xymatrix{{\rm QCoh}(Y_{\FF,S})\ar@<-1mm>[r]_-{\rm Id}\ar@<1mm>[r]^{\phi^*} & {\rm QCoh}(Y_{\FF,S})}\big).
 $$
It is  the $\infty$-category of  pairs $(\se, \phi_{\se })$, where $\se$ is a quasi-coherent sheaf on $Y_{\FF,S}$ and
$\phi_{\se} : \phi^*\se \stackrel{\sim}{\to}  \se $ is a quasi-isomorphism\footnote{We will call isomorphisms in the $\infty$-categories  ${\rm QCoh}(-)$ quasi-isomorphisms to be compatible with more classical set-ups.}.
The category  ${\rm Nuc}(Y_{\FF,S})^{\phi}$ (resp. ${\rm Perf}(Y_{\FF,S})^{\phi})$ is the full $\infty$-subcategory of ${\rm QCoh}(Y_{\FF,S})^{\phi}$ spanned by
the pairs $(\se, \phi_{\se} )$, where $\se$ is a nuclear (resp. perfect) sheaf on $Y_{\FF,S}$. 

    In what follows we will set 
$$ {\text{$u=(p-1)/p$, $v=p-1$ if $p\neq 2$;}}\quad{\text{ for $p=2$ we take $u=3/4$, $v=3/2$.}} $$  
If $S$ is an affinoid perfectoid over $C^{\flat}$, with $\varpi=p^{\flat}$ as above and untilt $S^{\sharp}$ over $C$, 
this choice of $u,v$ ensures that the divisor on $Y_{{\rm FF},S}^{[u,v]}$ associated to $t$ 
is $y_{\infty}$  and $t$ is a unit \index{BRING@\BRING}in $\B_{S}^{[u,v/p]}$, i.e., if 
$S^\sharp={\rm Spa}(R,R^+)$, then $\B_{S}^{[u,v]}/t=R$ and $\B_{S}^{[u,v/p]}/t=0$.

Via analytic descent, 
we like to describe the above categories of $\phi$-equivariant sheaves using  the \index{FFF@\FFF}chart $Y_{{\rm FF},S}^{[u,v]}$ (via Frobenius we have $\phi: Y_{{\rm FF},S}^{[u/p,v/p]}\stackrel{\sim}{\to}Y_{{\rm FF},S}^{[u,v]}$):
$$
{\rm QCoh}(Y_{\FF,S})^{\phi} \simeq  {\rm eq}\big(\xymatrix{{\rm QCoh}(Y_{\FF,S}^{[u,v]})\ar@<-1mm>[r]_-{j^*}\ar@<1mm>[r]^{\phi^*} & {\rm QCoh}(Y_{\FF,S}^{[u,v/p]}})\big).
 $$
 We wrote here $\phi,j$ for the Frobenius and the open embedding maps from $Y_{\FF,S}^{[u,v/p]}$ to $Y_{\FF,S}^{[u,v]}$, respectively. 
 That is, ${\rm QCoh}(Y_{\FF,S})^{\phi}$ is  the $\infty$-category of  pairs $(\se, \phi_{\se })$, where $\se$ is a quasi-coherent sheaf on $Y_{\FF,S}^{[u,v]}$ and
$\phi_{\se} : \phi^*\se \stackrel{\sim}{\to}  j^*\se $ is a quasi-isomorphism. The categories  ${\rm Nuc}(Y_{\FF,S})^{\phi}$, ${\rm Perf}(Y_{\FF,S})^{\phi}$ can be described in an analogous way. 

     We note that, since we have the equivalence \eqref{eq1}, we can also \index{DER@\DER}\index{BRING@\BRING}write\footnote{We stress here that  $\sd(\B^{\FF}_{S,{\rm an}})$ and $\B^{\FF}_{S,{\rm an}}$ are just notation; the ring $\B^{\FF}_{S,{\rm an}}$ does not exist.}
 $$
{\rm QCoh}(Y_{\FF,S})^{\phi} \simeq  \sd(\B^{\FF}_{S,{\rm an}})^{\phi} 
:={\rm eq}\big(\xymatrix{\sd(\B^{[u,v]}_{S,{\rm an}})\ar@<-1mm>[r]_-{j^*}\ar@<1mm>[r]^{\phi^*} 
& \sd(\B^{[u,v/p]}_{S,{\rm an}})}\big).
 $$
 (Frobenius $\phi$ maps $\B^{[u,v]}_{S,{\rm an}}$ to $\B^{[u,v/p]}_{S,{\rm an}}$.)  It is  the $\infty$-category of  pairs 
$M=(M^{[u,v]},  \phi_{M })$, where $M^{[u,v]}$ is a   complex of  $\B^{[u,v]}_{S,{\rm an}}$-modules and 
the Frobenius $\phi_M$ is a quasi-isomorphism of complexes of  $\B^{[u,v/p]}_{S,{\rm an}}$-modules
$$\phi_{M} : \phi^*M^{[u,v]}\stackrel{\sim}{\to} M^{[u,v/p]}:= M^{[u,v]}\otimes^{\LL}_{\B^{[u,v]}_{S,{\rm an}}}\B^{[u,v/p]}_{S,{\rm an}}.$$  

\begin{remark}\label{52}
 In what follows it will be convenient 
to consider the following variant $\sd({\B}^{\FF}_{S})^{\phi}$ (in which we drop the ``an'', i.e., we
consider $(R,\Z)_\Box$ instead of $(R,R^+)_\Box$)
of the $\infty$-category $ \sd({\B}^{\FF}_{S,{\rm an}})^{\phi}$:
 $$
 \sd({\B}^{\FF}_{S})^{\phi} :={\rm eq}\big(\xymatrix{\sd({\B}^{[u,v]}_{S,\Box})
\ar@<-1mm>[r]_-{j^*}\ar@<1mm>[r]^{\phi^*} & \sd({\B}^{[u,v/p]}_{S,\Box})}\big).
 $$
It is  the $\infty$-category of  pairs $M=(M^{[u,v]},  \phi_{M })$, where $M^{[u,v]}$ is a   complex of  solid $\B_{S}^{[u,v]}$-modules and 
the Frobenius $\phi_M$ is a quasi-isomorphism of complexes of  solid $\B_{S}^{[u,v/p]}$-modules
$$\phi_{M} : \phi^*M^{[u,v]}\stackrel{\sim}{\to} M^{[u,v/p]}
:= M^{[u,v]}\otimes^{\LL_{\Box}}_{\B^{[u,v]}_{S}}\B_{S}^{[u,v/p]}.$$  

We \index{DER@\DER}call $ \sd({\B}^{\FF}_{S})^{\phi}$ {\it the category of $\phi$-complexes of $\B^{\rm FF}_S$-modules}.
 Since  we have the equivalences of symmetric monoidal categories 
$\sd(({\B}_{S}^{I},\Z)_{\Box})=\sd(\B^I_{S,{\Box}})$  
(see \cite[Lemma A.16]{GB1}), this corresponds to using the analytic structure with respect to $\Z$ in place of ${\B}^{I,+}_{S}$.
 In particular, we have a canonical monoidal functor $ \sd({\B}^{\FF}_{S})^{\phi} \to \sd({\B}^{\FF}_{S,{\rm an}})^{\phi}$. 
\end{remark}
\subsubsection{Monoidal structure on quasi-coherent sheaves on $Y_{\FF}$}\label{monoidal}
 The equalizer defining ${\rm QCoh}(Y_{\FF,S})^{\phi}$ is taken in the $\infty$-category of symmetric monoidal $\infty$-categories and symmetric monoidal functors: both $\phi^*$ and $j^*$ are symmetric monoidal.  We now describe its tensor product and internal $\Hom$ on the level of $\sd(\B^{\FF}_{S,{\rm an}})^{\phi}$. In what follows we have set 
 $\B_1:=\B^{[u,v]}_{S,{\rm an}}$, $\B_{2}:=\B^{[u,v/p]}_{S,{\rm an}}$. 
  
   The (derived) tensor product in  ${\sd}(\B^{\rm FF}_{S,{\rm an}})^{\phi}$, 
     denoted by $(-)\otimes^{\LL}_{\B^{\FF}_{S,{\rm an}}}(-)$,  is  inherited from the one of the 
    category   $\sd(\B_1)$. More precisely,
 for $(M,\phi_M), (N,\phi_N)\in {\sd}(\B^{\rm FF}_{S,{\rm an}})^{\phi} $, their tensor product is given on objects by: 
    \begin{align*}
    & M\otimes^{\LL}_{\B^{\FF}_{S,{\rm an}}}N:=(M^{[u,v]}\otimes^{\LL}_{\B_1}N^{[u,v]}, \phi_{M\otimes N}), \\
     & \phi_{M\otimes N}=\phi_M\otimes\phi_N: (M^{[u,v]}\otimes^{\LL}_{\B_1}N^{[u,v]})\otimes^{\LL}_{\B_1,\phi}\B_2
       \to    (M^{[u,v]}\otimes^{\LL}_{\B_1}N^{[u,v]})\otimes^{\LL}_{{\B_1}}\B_2
     =   (M^{[u,v/p]}\otimes^{\LL}_{\B_{2}}N^{[u,v/p]}).
    \end{align*}
 Frobenius $\phi_{M\otimes N}$ is a quasi-isomorphism because so are the Frobenius maps $\phi_M$ and $\phi_N$. 
   
    The  internal $\R\Hom$,   denoted \index{rhom@\rhom}by $\R\uHom_{\B^{\FF}_{S,{\rm an}}}(-,-)$, in 
 the category  ${\sd}(\B^{\rm FF}_{S,{\rm an}})^{\phi}$  is computed by: 
    \begin{align*}
    &\R\uHom_{\B^{\FF}_{S,{\rm an}}}(M,N):=( \R\uHom_{\B_1}(M^{[u,v]},N^{[u,v]}),\phi_{M,N}),\\
    & \phi_{M,N}:= (\phi_{M}^{-1},\phi_N):\quad \R\uHom_{\B_1}(M^{[u,v]},N^{[u,v]}) \otimes^{\LL}_{\B_1,\phi}\B_2\to 
    \R\uHom_{\B_1}(M^{[u,v]},N^{[u,v]})\otimes^{\LL}_{\B_1}\B_2.
    \end{align*}
    In the definition of Frobenius $\phi_{M,N}$ we have used the following (non-obvious\footnote{One
usually needs some finiteness condition for this kind of statement to hold.}) 
fact:
    \begin{lemma}\label{fun1-kol}
The canonical maps 
    \begin{align*}
    \R\uHom_{\B_1}(M^{[u,v]},N^{[u,v]}) \otimes^{\LL}_{\B_1,\phi}\B_2 & \to
    \R\uHom_{\B_2}(M^{[u,v]} \otimes^{\LL}_{\B_1,\phi}\B_2,N^{[u,v]} \otimes^{\LL}_{\B_1,\phi}\B_2),\\
    \R\uHom_{\B_1}(M^{[u,v]},N^{[u,v]})\otimes^{\LL}_{\B_1}\B_2 & \to \R\uHom_{\B_2}(M^{[u,v/p]},N^{[u,v/p]})
    \end{align*}
    are quasi-isomorphisms. 
    \end{lemma}
    \begin{proof}To start, note that, since the first map is induced by the composition of the maps 
$$\phi: \B^{[u,v]}_{S,{\rm an}}{\to }\B^{[u/p,v/p]}_{S,{\rm an}},
\quad \can: \B^{[u/p,v/p]}_{S,{\rm an}}\to \B^{[u,v/p]}_{S,{\rm an}}$$ 
where the first map is an isomorphism, it suffices to argue for the second quasi-isomorphism in the lemma.
   
      Write $M^{[u,v]}=\colim_{i\in I} M_{i}^{[u,v]}$ as a  colimit of compact 
projective objects $\{M_{i}^{[u,v]}=\B_1[T_i]\}$, $i\in I$, for extremally disconnected sets $T_i$. Then
     \begin{align*}
      \R\uHom_{\B_1}(M^{[u,v]},N^{[u,v]}) &=  
\R\uHom_{\B_1}(\colim_{i\in I} M_{i}^{[u,v]}, N^{[u,v]})\\
       & \simeq {\rm Rlim}_I\R\uHom_{\B_1}( M_{i}^{[u,v]}, N^{[u,v]})
     \end{align*}
     and similarly for $[u,v/p]$ (we set $M_{i}^{[u,v/p]}:=M_{i}^{[u,v]}\otimes^{\LL}_{\B_1}\B_2\simeq \B_2[T_i]$). It follows that it suffices to show that 
     $$
    ( {\rm Rlim}_I\R\uHom_{\B_1}( M_{i}^{[u,v]}, N^{[u,v]}))\otimes^{\LL}_{\B_1}\B_2
\stackrel{\sim}{\to} {\rm Rlim}_I\R\uHom_{\B_2}( M_{i}^{[u,v/p]}, N^{[u,v/p]}).
     $$
     But, by \cite[Prop.\,5.38]{And21}, we have 
     $$
      \R\uHom_{\B_1}( M_{i}^{[u,v]}, N^{[u,v]})\otimes^{\LL}_{\B_1}\B_2
\stackrel{\sim}{\to}\R\uHom_{\B_2}( M_{i}^{[u,v/p]}, N^{[u,v/p]}).
 $$
 Hence it suffices to show that
 $$
     ( {\rm Rlim}_I\R\uHom_{\B_1}( M_{i}^{[u,v]}, N^{[u,v]}))\otimes^{\LL}_{\B_1}\B_2\stackrel{\sim}{\to}
        {\rm Rlim}_I(\R\uHom_{\B_1}( M_{i}^{[u,v]}, N^{[u,v]})\otimes^{\LL}_{\B_1}\B_2).
 $$
 That is, we must show that the functor $(-)\otimes^{\LL}_{\B_1}\B_2$ commutes with derived limits. 
 
   To show this write $\B_{S}^{[u,v/p]}=\B_{S}^{[u,v]}\langle f\rangle$, where $f=p^a/[\varpi]^b$ for positive integers $a,b$ such that $a/b=v/p$.
By \cite[Prop.\,4.11]{And21}, we have 
 \begin{align}
 (-)\otimes^{\LL}_{\B_1}\B_2\simeq (-)\otimes^{\LL}_{(\Z[T],\Z)_{\Box}}(\Z[T],\Z[T])_{\Box},
 \end{align}
 where the map $(\Z[T],\Z)_{\Box}\to (\B_{S}^{[u,v]},\B^{[u,v],+}_{S})_{\Box}$ 
is induced by $T\mapsto f$. But, by \cite[Prop.\,3.12]{And21}, for $M\in \sd((\Z[T],\Z)_{\Box})$, we have 
 $$
 M\otimes^{\LL}_{(\Z[T],\Z)_{\Box}}(\Z[T],\Z[T])_{\Box}\simeq \R\uHom_{R}(R_{\infty}/R,M)[1],
 $$ where $R=\Z[T]$, $R_{\infty}=\Z((T^{-1}))$. It follows that the functor $(-)\otimes^{\LL}_{\B_1}\B_2$ commutes with derived limits, as wanted.
    \end{proof}

      Finally, we note that Frobenius $\phi_{M, N}$ is a quasi-isomorphism because so are the Frobenius maps $\phi_M$ and $\phi_N$.
 \begin{remark}\label{zet1}
(1)  The tensor product computations above are valid     for the category $ \sd({\B}^{\FF}_{S})^{\phi}$. For the internal $\Hom$, they go through as well if one assumes that $(M,\phi_M), (N,\phi_N)$ are nuclear and so is the internal $\Hom$  between them (see \cite[Lemma 4.7]{ZL} for a proof of an analog of Lemma \ref{fun1-kol} in this setting). In this paper we will always be in this setting. 

(2) Let $M,N\in \sd((R,\Z)_{\Box})$. We note that the natural map 
$$
\R\uHom_{(R,\Z)_{\Box}}(M,N)\otimes^{\LL}_{(R,\Z)_{\Box}}(R,R^+)_{\Box}\to \R\uHom_{(R,R^+)_{\Box}}(M\otimes^{\LL}_{(R,\Z)_{\Box}}(R,R^+)_{\Box},N\otimes^{\LL}_{(R,\Z)_{\Box}}(R,R^+)_{\Box})
$$
is a quasi-isomorphism in the case $N$ is $(R,R^+)_{\Box}$-complete. Indeed, the two analytic structures have the same underlying condensed ring. By the completion adjunction and \cite[Lemma 5.27(i),(ii)]{And21}, the internal $\R\Hom$ into a complete target is already complete, and completing its source does not change it. Thus $\R\uHom_{(R,\Z)_{\Box}}(M,N)$ is $(R,R^+)_{\Box}$-complete and identifies with the internal $\R\Hom$ on the right. For example, this is the case when $N$ is nuclear. 
 \end{remark}
      \subsubsection{Quasi-coherent $\phi$-sheaves on $Y_{\FF}$ and $\phi$-modules}
We will  now describe the categories ${\rm Nuc}(Y_{\FF,S})^{\phi}$ and ${\rm Perf}(Y_{\FF,S})^{\phi}$ using complexes of (usual) solid modules. 

  Recall that the natural maps of analytic rings
$(\B_{S}^{I },\Z)_{\Box}\to   (\B_{S}^{I} ,\B^{I,+}_{S} )_{\Box}$
induce base change
functors 
\begin{equation}\label{mpr1}
(-) \otimes^{\LL} _{(\B_{S}^{I },\Z)_{\Box}} (\B_{S}^{I },\B^{I,+}_{S})_{\Box} :\quad  
\sd((\B_{S}^{I }, \Z)_{\Box}) \to  \sd((\B_{S}^{I} ,\B^{I,+}_{S})_{\Box}).
 \end{equation}
 By \cite[(6.13)]{GB2}, the functors \eqref{mpr1} induce equivalences on the full subcategories of nuclear and perfect complexes:
 \begin{align}\label{kawa1}
   {\rm Nuc}(\B_{S}^{I }):= {\rm Nuc}((\B_{S}^{I }, \Z)_{\Box})  & \stackrel{\sim}{\to} {\rm Nuc}((\B_{S}^{I} ,\B^{I,+}_{S})_{\Box}),\\
      {\rm Perf}(\B_{S}^{I })\simeq  {\rm Perf}((\B_{S}^{I }, \Z)_{\Box}) & \stackrel{\sim}{\to} {\rm Perf}((\B_{S}^{I} ,\B^{I,+}_{S})_{\Box}).\notag
 \end{align}

  We define the category  ${\rm Nuc}(\B^{\rm FF}_{S})^{\phi}$ (resp. ${\rm Perf}(\B^{\rm FF}_{S})^{\phi})$ as the full $\infty$-subcategory of $\sd(\B^{\rm FF}_{S})^{\phi}$ spanned by
the pairs  $( M^{[u,v]},  \phi_{M })$, where $M^{[u,v]} $ is a nuclear (resp.~perfect) 
complex over $\B_{S}^{[u,v]}$. 
That is,  the $\infty$-category ${\rm Nuc}(\B^{\rm FF}_{S})^{\phi}$ of nuclear $\phi$-complexes of $\B^{\rm FF}_S$-modules is defined as the equalizer:
$$
{\rm Nuc}(\B^{\rm FF}_{S})^{\phi}:= {\rm eq}\big( \xymatrix{{\rm Nuc}(\B_{S}^{[u,v]})\ar@<-1mm>[r]_-{\can}\ar@<1mm>[r]^{\phi^*}  &
{\rm Nuc}(\B_{S}^{[u,v/p]})}\big).
$$
Similarly, for the category ${\rm Perf}(\B^{\rm FF}_{S})^{\phi}$ of $\phi$-complexes of perfect $\B^{\rm FF}_S$-modules.
 
 We have the following simple fact:
\begin{lemma}\label{simple1}
 The canonical   functor
$$
\sd(\B^{\rm FF}_{S})^{\phi}\to {\rm QCoh}(Y_{\FF,S})^{\phi}
$$
 induces equivalences of $\infty$-categories:
\begin{equation}\label{mor1}
{\rm Nuc}(\B^{\rm FF}_{S})^{\phi}\stackrel{\sim}{\to} {\rm Nuc}(Y_{\FF,S})^{\phi},\quad {\rm Perf}(\B^{\rm FF}_{S})^{\phi}\stackrel{\sim}{\to }{\rm Perf}(Y_{\FF,S})^{\phi}.
\end{equation}
\end{lemma}
\begin{proof}
Our claim follows from the equivalences \eqref{kawa1}. 
\end{proof}
   
    Taking the displayed equalizers in the $\infty$-category of symmetric monoidal $\infty$-categories gives ${\rm Nuc}(\B^{\rm FF}_{S})^{\phi}$ and ${\rm Perf}(\B^{\rm FF}_{S})^{\phi}$ their symmetric monoidal structures. Their derived tensor products, denoted by $(-)\otimes^{\LL}_{\B^{\FF}_{S}}(-)$, are computed on the underlying complexes over $\B_{S}^{[u,v]}$. The canonical functors to ${\sd}(\B^{\rm FF}_{S})^{\phi}$ and the equivalences in Lemma \ref{simple1} are symmetric monoidal, since the functors inducing the equalizer diagrams are symmetric monoidal.

\subsubsection{Quasi-coherent sheaves on $X_{\FF}$}
The action of $\phi$ on $Y_{\FF,S}$ being free and totally discontinuous,  by the analytic descent for solid quasi-coherent sheaves, 
we obtain an equivalence \index{E@\EEE}of $\infty$-categories
$$\se_{\FF,S}:\quad {\rm QCoh}(Y_{\FF,S})^{\phi} \stackrel{\sim}{\to} {\rm QCoh}(Y_{\FF,S}/\phi^{\Z})={\rm QCoh}(X_{\FF,S}).
 $$
 Similarly, we get equivalences of closed symmetric monoidal $\infty$-categories
 \begin{equation}\label{mor2}
 {\rm Nuc}(Y_{\FF,S})^{\phi} \stackrel{\sim}{\to} {\rm Nuc}(X_{\FF,S}),\quad {\rm Perf}(Y_{\FF,S})^{\phi} \stackrel{\sim}{\to} {\rm Perf}(X_{\FF,S}).
 \end{equation}
By Lemma \ref{simple1},  this yields a functor
\begin{equation}\label{functor1}
 \se_{\FF,S}: \sd(\B^{\rm FF}_{S})^{\phi} \to {\rm QCoh}(X_{\FF,S}). 
 \end{equation}
We will often skip the subscript $S$ from $ \se_{\FF,S}$ if this does not cause confusion. Restricting to  nuclear or perfect complexes  we get the following result (see  \cite[Th.\,6.8]{GB2} for a similar statement):
\begin{proposition}\label{leje1}
{\rm (1)} The functor $\se_{\FF,S}$, from  \eqref{functor1}, induces  equivalences of $\infty$-categories
\begin{equation}\label{nuceq}
 {\rm Nuc}(\B^{\rm FF}_S)^{\phi}\stackrel{\sim}{\to} {\rm Nuc}(X_{\FF,S}), \quad 
 {\rm Perf}(\B^{\rm FF}_S)^{\phi}\stackrel{\sim}{\to} {\rm Perf}(X_{\FF,S}).
 \end{equation}
{\rm (2)} Let  $\se\in{\rm  Nuc}(X_{\FF,S})$.  Let $(M (\se)^{[u,v]}, \phi_M)$ be the  nuclear $\phi$-complex of $\B^{\rm FF}_S$-modules
corresponding to $\se$ via \eqref{nuceq}. Then, there
is a natural quasi-isomorphism in $\sd(\Q_{p}(S)_{\Box})$ 
$$\R\Gamma(X_{\FF,S}, \se) \simeq [M(\se)^{[u,v]}\lomapr{\phi-1} M(\se)^{[u,v/p]}].
 $$
 \end{proposition}
 \begin{proof}The first claim is a combination of \eqref{mor1} and \eqref{mor2}. 
 For the second claim, we compute
 \begin{align*}
 \R\Gamma(X_{\FF,S}, \se)  & \simeq \R\Gamma(\phi^{\Z}, \R\Gamma(Y_{{\rm FF},S},\se_{|Y_{{\rm FF},S}}))
 \simeq [\Gamma(Y_{{\rm FF},S}^{[u,v]},\se_{|Y_{{\rm FF},S}})\lomapr{\phi-1} \Gamma(Y_{{\rm FF},S}^{[u,v/p]},
\se_{|Y_{{\rm FF},S}})] \\ &\simeq [M(\se)^{[u,v]}\lomapr{\phi-1} M(\se)^{[u,v/p]}].
 \end{align*}
 \end{proof}

 \section{Syntomic complexes on the Fargues-Fontaine curve} 
In this section we define quasi-coherent sheaves on the Fargues-Fontaine curve 
representing various cohomologies of smooth partially proper rigid analytic varieties:
de Rham (Proposition~\ref{nyc1}), Hyodo-Kato (Proposition~\ref{sobota12}),
and syntomic (Formula~\eqref{def1} and Proposition~\ref{comp2}).
We will do the same for pro-\'etale cohomology in the next section (Proposition~\ref{comp2b}).

 \subsection{de Rham  cohomology} 
We start with  the cohomologies of de Rham type.  
We use \cite[Sec.\,3, Sec.\,4, Sec.\,5]{AGN} as the basic reference.
  \subsubsection{$\B^+_{\dr}$-cohomology} 
Let $X$ be a smooth partially proper   rigid analytic variety over $K$. We have the (filtered) de Rham complexes in $\sd(K_{\Box})$  and (filtered)  $\B^+_{\dr}$-cohomology complexes in $\sd(\B^+_{\dr,\Box})$, \index{DR@\DR}respectively: 
$$
  F^r\R\Gamma_{\dr,?}(X),\quad 
  F^r\R\Gamma_{\dr,?}(X_C/\B^+_{\dr}),\quad r\in\N,\ ?=-,c
  $$
 as well as  the \index{DR@\DR}quotients
  $$
  \R\Gamma_{\dr,?}(X_C,r)  :=\R\Gamma_{\dr,?}(X_C/\B^+_{\dr})/F^r. 
  $$

   The latter complexes can be represented by   quasi-coherent sheaves on $X_{\rm FF}$. For $r\in \N$, we define the {\em de Rham \index{DR@\DR}modules}
$$
\R\Gamma^{{[u,v]}}_{\dr,?}(X_C,r):=  \R\Gamma_{\dr,?}(X_C,r).
$$
Since $\B^{[u,v]}/t^i=\B^+_{\dr}/t^i$, these are $\B^{[u,v]}$-modules. They are nuclear: for the usual cohomology, in the Stein case this follows from Section \ref{kwak-kwak} below; in the general case, it follows from the fact that nuclearity is preserved by countable products. For the cohomology with compact support: we use the Stein case again and then pass to a colimit which preserves nuclearity. Since $ \R\Gamma^{{[u,v]}}_{\dr,?}(X_C,r)\otimes^{\LL_{\Box}}_{\B^{[u,v]}}\B^{[u,v/p]}=0$ (recall that $t$ is invertible in $\B^{[u,v/p]}$),
these complexes taken as \index{DR@\DR}pairs 
$$\R\Gamma^{\B}_{\dr,?}(X_C,r)=( \R\Gamma^{{[u,v]}}_{\dr,?}(X_C,r),0)$$
 define nuclear $\phi$-complexes over $\B^{\rm FF}$ 
(see Remark~\ref{52}).

 We denote \index{E@\EEE}by
$$
\se_{\dr,?}(X_C,r):=\se_{\FF}(\R\Gamma^{\B}_{\dr,?}(X_C,r))
$$
the corresponding nuclear quasi-coherent sheaves on $X_{\FF}$. We will call them {\em de Rham sheaves}. We record the following simple fact:
\begin{proposition} \label{nyc1}
 Let $r\in\N$. We have a natural quasi-isomorphism in ${\rm QCoh}(X_{\FF})$
$$\se_{\dr,?}(X_C,r) \simeq i_{\infty,*}\R\Gamma_{\dr,?}(X_C,r).$$
\end{proposition}
For $S\in{\rm Perf}_C$, 
\index{PERF@\PERF}by replacing $\B, \B^+_{\dr}, X_{\FF}$ with $\B_{S^{\flat}}, \B^+_{\dr}(S), X_{\FF,S^{\flat}}$ in the above, we obtain de Rham modules and sheaves on $X_{\FF,S^{\flat}}$: 
$\R\Gamma^{\B}_{\dr,?}(X_S,r), \se_{\dr,?}(X_S,r)$.  These are functors on ${\rm Perf}_C$.

  \subsubsection{Stein varieties} \label{kwak-kwak} 
Let $X$ be a smooth Stein rigid analytic variety over $K$. In this case the above cohomology complexes can be made more explicit. 

\vskip2mm
 ($\bullet$) {\em De Rham cohomology.} Let $r\in\N$. Since coherent cohomology of $X$ is trivial in nonzero degrees and we have Serre duality, the (filtered) de Rham cohomology of $X$ can be computed by the following complexes
 in $\sd(K_{\Box})$: 
\begin{align*}
F^r\R\Gamma_{\dr}(X) & \simeq (\Omega^r(X)\to\cdots\to \Omega^d(X))[-r],\\
F^r\R\Gamma_{\dr,c}(X) & \simeq (H^d_c(X,\Omega^r)\to H^d_c(X,\Omega^{r+1})\to\cdots \to H^d_c(X,\Omega^d))[-d-r].
\end{align*}
The second quasi-isomorphism follows from the fact that $H^i_c(X,\Omega^j)=0$,  for $i\neq d$. The terms of the first complex are nuclear Fr\'echet spaces over $K$ 
and those of the second complex are of compact type over $K$ (in classical terminology).

\vskip2mm
($\bullet$) {\em $\B^+_{\dr}$-cohomology.} Let $r\in\N$. We set $t^a\B^+_{\dr}:=\B^+_{\dr}$,  for $a<0$. The (filtered) $\B^+_{\dr}$-cohomology of $X$ can be computed by the following complexes in $\sd(\B^+_{\dr,\Box})$: 
\begin{align}
F^r\R\Gamma_{\dr}(X_C/\B^+_{\dr}) & \simeq \so(X)\otimes^{\Box}_K t^r\B^+_{\dr}\to\Omega^1(X)\otimes^{\Box}_K t^{r-1}\B^+_{\dr}\to\cdots\to \Omega^d(X)\otimes^{\Box}_K t^{r-d}\B^+_{\dr},\\
F^r\R\Gamma_{\dr,c}(X_C/\B^+_{\dr}) & \simeq (H^d_c(X,\so)\otimes^{\Box}_K t^{r}\B^+_{\dr}\to H^d_c(X,\Omega^1)\otimes^{\Box}_K t^{r-1}\B^+_{\dr}\to\cdots \to H^d_c(X,\Omega^d)\otimes^{\Box}_K t^{r-d}\B^+_{\dr})[-d].\notag
\end{align}
The tensor products are actually  derived because $\B^+_{\dr}$ is a Fr\'echet space, hence flat. 

  This yields  the quasi-isomorphisms in $\sd(\B^+_{\dr,\Box})$:
\begin{align}\label{kolobrzeg1a}
\R\Gamma_{\dr}(X_C,r) &  \simeq \so(X)\otimes^{\Box}_K (\B^+_{\dr}/t^r)\to\Omega^1(X)\otimes^{\Box}_K (\B^+_{\dr}/t^{r-1})\to\cdots\to \Omega^d(X)\otimes^{\Box}_K (\B^+_{\dr}/t^{r-d}),\\
\R\Gamma_{\dr,c}(X_C,r) & \simeq (H^d_c(X,\so)\otimes^{\Box}_K (\B^+_{\dr}/t^{r})\to H^d_c(X,\Omega^1)\otimes^{\Box}_K (\B^+_{\dr}/t^{r-1})\to\cdots \to H^d_c(X,\Omega^d)\otimes^{\Box}_K( \B^+_{\dr}/t^{r-d}))[-d].\notag
\end{align}
We will denote the respective cohomology groups by $ H^{i}_{\dr}(X,r) $ and $ H^{i}_{\dr,c}(X,r) $.

 For $i\geq 0$, we have short exact sequences in $\sd(\B^+_{\dr,\Box})$ (see \cite[Example 3.30]{CDN3}, \cite[Lemma 3.14(1)]{AGN})
\begin{align}\label{kolo10}
  0 \to (\Omega^i(X)/ {\rm Im}\, d)\otimes^{\Box}_K {\rm gr}^{r-i-1}_F \B^+_{\dr} \to  & H^{i}_{\dr}(X_C,r) \to H^{i}_{\dr}(X) \otimes^{\Box}_K  (\B^+_{\dr}/ F^{r-i-1}) \to 0,\\
0 \to (H^d_c(X, \Omega^{i-d})/ {\rm Im}\, d) \otimes^{\Box}_K  {\rm gr}^{r-i+d-1}_F\B^+_{\dr}\to  & H^{i}_{\dr,c}(X_C,r) \to H^{i}_{\dr,c}(X) \otimes^{\Box}_K  (\B^+_{\dr}/ F^{r-i+d-1}) \to 0.\notag
\end{align}

  \subsection{Hyodo-Kato cohomology} 
Let $X$ be a smooth rigid analytic variety over $K$.  
\index{HK@\HK}Let  $\R\Gamma_{\hk}(X_C)\in \sd_{\phi,N,\sg_K}(\breve{C}_{\Box})$ be the Hyodo-Kato cohomology defined in \cite[Sec.\,4]{CN4} (see also \cite[Sec.\,3]{GB2}).
Here $ \sd_{\phi,N,\sg_K}(\breve{C}_{\Box})$ is the derived $\infty$-category of solid $(\phi,N,\sg_K)$-modules over $\breve{C}$.  Equivalently, this is the $\infty$-category of objects of
$\sd(\breve{C}_{\Box})$ equipped with a homotopy-coherent
$(\phi,N,\sg_K)$-structure, with continuity encoded by
viewing $\sg_K$ as a condensed group\footnote{
This identification follows from the description by modules
over the corresponding operator algebra; see
\cite[Rem.\,2.4, Sec.\,2.2.3]{CN4},
\cite[Rem.\,3.12]{GB2}, and \cite[Sec.\,4.2]{RR1}.}.

  \subsubsection{Hyodo-Kato cohomology on the Fargues-Fontaine curve} 
  Let $r\in\Z$. Consider the twisted Hyodo-Kato cohomology in $ \sd_{\phi,\sg_K}(\breve{C}_{\Box})$
 \begin{align*}
\R\Gamma_{\hk}^{I}(X_C,r)  := [\R\Gamma_{\hk}(X_C)\{r\}\otimes^{\LL_{\Box}}_{\breve{C}}\B^{I}_{\log}]^{N=0},
 \end{align*}
 where the twist $\{r\}$ means Frobenius divided by $p^r$ and $I \subset (0,\infty)$ is a  compact interval with rational endpoints. We define $\R\Gamma_{\hk}^{\B}(X_C,r)$ in a similar way. We claim  that, for  compact intervals $I\subset J\subset (0,\infty)$ with rational endpoints, we have the canonical quasi-isomorphism
$$
\R\Gamma_{\hk}^{J}(X_C,r)\otimes^{\LL_{\Box}}_{\B^J}\B^I\stackrel{\sim}{\to} \R\Gamma_{\hk}^{I}(X_C,r).
$$
Indeed, for that, it suffices to show that the canonical map
$$
\R\Gamma_{\hk}(X_C)\{r\}\otimes^{\LL_{\Box}}_{\breve{C}}\B^{J}_{\log}\otimes^{\LL_{\Box}}_{\B^J}\B^I\to \R\Gamma_{\hk}(X_C)\{r\}\otimes^{\LL_{\Box}}_{\breve{C}}\B^{I}_{\log}
$$
is a quasi-isomorphism. But this is clear since the solid tensor product commutes with direct sums (see also \cite[Lemma 4.12]{ZL}).

 We define the pair\footnote{There is a certain doubling of notation with the previous paragraph but we hope that this will not cause confusion in what follows.}
 $$
 \R\Gamma^{\B}_{\hk}(X_C,r):=( \R\Gamma_{\hk}^{{[u,v]}}(X_C,r),\phi),\quad \phi:  \R\Gamma_{\hk}^{{[u,v]}}(X_C,r)\to  \R\Gamma_{\hk}^{{[u,v/p]}}(X_C,r),
 $$
 where the Frobenius $\phi$ is induced from  the Hyodo-Kato Frobenius and the Frobenius $\phi: \B^{[u,v]}\to \B^{[u,v/p]}$. It
 yields a quasi-isomorphism in $\sd(\B^{[u,v/p]}_{\Box})$
 $$\phi:  \R\Gamma_{\hk}^{{[u,v]}}(X_C,r)\otimes^{\LL_{\Box}}_{\B^{[u,v]},\phi}\B^{[u,v/p]}\stackrel{\sim}{\to } \R\Gamma_{\hk}^{{[u,v/p]}}(X_C,r).$$
 The \index{HK@\HK}pair  $\R\Gamma^{\B}_{\hk}(X_C,r)$
  defines  a  nuclear $\phi$-complex (actually a $(\phi,\sg_K)$-complex) over $\B^{\rm FF}$, which we will call a {\em Hyodo-Kato module}. 
 
  We define {\em Hyodo-Kato  sheaves} on $X_{\rm FF}$ \index{E@\EEE}as  
$$
 \se_{\hk}(X_C,r)  :=\se_{\FF}(\R\Gamma^{\B}_{\hk}(X_C,r)).
$$
By Proposition \ref{leje1}, these are  nuclear quasi-coherent sheaves on $X_{\FF}$.   If the  cohomology groups of $\R\Gamma_{\hk}(X_C)$ are of finite rank over $\breve{C}$ then the sheaf  $\se_{\hk}(X_C,r)$ is perfect. 
  By Proposition \ref{leje1} and  \cite[Th.\,6.3]{GB2}, we have  natural quasi-isomorphisms in $\sd(\Q_{p,\Box})$
 \begin{align}\label{sobota1}
 \R\Gamma(X_{\FF}, \se_{\hk}(X_C,r)) & \simeq  [\R\Gamma_{\hk}(X_C)\{r\}\otimes^{\LL_{\Box}}_{\breve{C}}
\B^{[u,v]}_{\log}]^{N=0,\phi=1}\\
  &\stackrel{\sim}{\leftarrow}  [\R\Gamma_{\hk}(X_C)\{r\}\otimes^{\LL_{\Box}}_{\breve{C}}\B_{\log}]^{N=0,\phi=1}\notag
 \end{align}
 where we set, for $M=\R\Gamma_{\hk}(X_C)\{r\}\otimes^{\LL_{\Box}}_{\breve{C}}\B^{[u,v]}_{\log}$ or 
$M=\R\Gamma_{\hk}(X_C)\{r\}\otimes^{\LL_{\Box}}_{\breve{C}}\B_{\log}$ and $M^{\prime}=\R\Gamma_{\hk}(X_C)\{r\}\otimes^{\LL_{\Box}}_{\breve{C}}\B^{[u,v/p]}_{\log}$ or
$M^{\prime}=\R\Gamma_{\hk}(X_C)\{r\}\otimes^{\LL_{\Box}}_{\breve{C}}\B_{\log}$, 
$$[M]^{N=0,\phi=1} :=\left[\vcenter{\xymatrix @C=1cm@R=5mm{
M\ar[r]^-{\phi-1} \ar[d]^{N} & M^{\prime} \ar[d]^{N} \\
M \ar[r]^-{p\phi-1} & M^{\prime}}} \right]
$$

   For $S\in {\rm Perf}_C$, by changing $\B,\B^I,\B^{I}_{\log}$ 
to $\B_{S^{\flat}},\B^I_{S^{\flat}},\B^I_{S^{\flat},\log}$, we obtain  Hyodo-Kato modules and sheaves:
   $$
   \R\Gamma^{\B}_{\hk}(X_S,r),  \quad \se_{\hk}(X_S,r).
   $$
   These are functors on ${\rm Perf}_C$. 
     In the case $X$ is partially proper, we have analogs 
$\R\Gamma^{\B}_{\hk,c}(X_S,r)$,  $\se_{\hk,c}(X_S,r)$ for Hyodo-Kato cohomology with compact support\footnote{See \cite[Sec.\,3.2, Sec.\,4.3, Sec.\,5.4]{AGN} for the definition and basic properties of compactly supported Hyodo-Kato cohomology.} and the following analog of quasi-isomorphism \eqref{sobota1}:
   \begin{proposition}\label{sobota12}
Let $r\in\Z$. We have a natural quasi-isomorphism in $\sd(\Q_{p}(S)_{\Box})$
$$
\R\Gamma(X_{\FF,S^{\flat}},\se_{\hk,?}(X_S,r))\simeq [\R\Gamma_{\hk,?}^{\B}(X_S,r)]^{\phi=1}. 
$$
\end{proposition}

\subsubsection{Hyodo-Kato map} Let $X$ be a smooth partially proper rigid analytic variety   over $K$. 
Recall that we have the natural Hyodo-Kato maps (see \cite[Sec.\,4]{CN4}) in $\sd(\breve{C}_{\Box})$ and $\sd(\B^+_{\dr,\Box})$, \index{HKi@\HKi}respectively:
$$
\iota_{\hk}: \R\Gamma_{\hk}(X_C) \to \R\Gamma_{\dr}(X_C/\B^+_{\dr}),\quad \iota_{\hk}: \R\Gamma_{\hk}(X_C)\otimes^{\LL_{\Box}}_{\breve{C}}\B^+_{\dr}\stackrel{\sim}{\to}  \R\Gamma_{\dr}(X_C/\B^+_{\dr}).
$$
Combined with the  canonical map
$\iota: \B^{[u,v]}_{\log}\to \B^{[u,v]}/t^i$,   it defines a map between complexes of solid
$\B^{[u,v]}$-modules:
\begin{equation}\label{sobota11}
\iota_{\hk}: \R\Gamma_{\hk}^{{[u,v]}}(X_C,r) =[ \R\Gamma_{\hk}(X_C)\{r\}\otimes^{\LL_{\Box}}_{\breve{C}}
\B^{[u,v]}_{\log}]^{N=0}\to \R\Gamma^{{[u,v]}}_{\dr}(X_C,r).
\end{equation}
Since we have a commutative diagram
\begin{equation}\label{tiger2}
\xymatrix@R=5mm{\R\Gamma_{\hk}^{{[u,v]}}(X_C,r)\ar[d]^{\iota_{\hk}}\ar[r]^-{\phi\otimes\phi} 
& \R\Gamma_{\hk}^{{[u,v/p]}}(X_C,r)\ar[d]\\
 \R\Gamma^{{[u,v]}}_{\dr}(X_C,r)\ar[r] &  0
}
\end{equation}
the map \eqref{sobota11} clearly lifts to a map of $\phi$-modules over $\B^{\rm FF}$:
$$
\iota_{\hk}: \quad\R\Gamma^{\B}_{\hk}(X_C,r) \to \R\Gamma^{\B}_{\dr}(X_C,r).
$$

  This  Hyodo-Kato map descends to the level of nuclear quasi-coherent sheaves on $X_{\FF}$:
$$\iota_{\hk}:\quad \se_{\hk}(X_C,r)\to  \se_{\dr}(X_C,r).
$$ Everything above has a version for compactly supported cohomologies (see \cite[Prop.\,3.19]{AGN} for Hyodo-Kato morphisms), as well as for $S$-cohomologies, for $S\in{\rm Perf}_C$ (varying functorially in $S$).

 \subsection{Syntomic cohomology}\label{sing2}
We pass now to  syntomic cohomology. 
 \subsubsection{Classical syntomic cohomology}Let $X$ be a smooth partially proper rigid analytic variety  over $K$. Let $r\in\N$.  Consider the classical syntomic cohomology \index{SYN@\SYN}(\`a la Bloch-Kato) 
(see \cite[Sec.\,5.4]{CN4})
 $$
 \R\Gamma^{\B_{\crr}^+}_{\synt,?}(X_C,\Q_p(r))
:=\big[[\R\Gamma_{\hk,?}(X_C)\otimes^{\LL_{\Box}}_{\breve{C}}\B^+_{\st}]^{N=0,\phi=p^r}\lomapr{\iota_{\hk}\otimes\iota}\R\Gamma_{\dr,?}(X_C/\B_{\dr}^+)/F^r\big].
 $$
 It satisfies the following comparison theorem:
  \begin{theorem} \label{comp0} 
{\rm(Period isomorphism, \cite[Th.\,6.9]{CN4})}

Let $r\in\N$. There is a natural  quasi-isomorphism in $\sd(\Q_{p,\Box})$
 \begin{equation}\label{tea1}
\alpha_r: \quad  \tau_{\leq r}\R\Gamma^{\B_{\crr}^+}_{\synt,?}(X_C,\Q_p(r))\simeq \tau_{\leq r}\R\Gamma_{\proeet,?}(X_C,\Q_p(r)).
 \end{equation}
 Moreover,  it yields   a natural quasi-isomorphism  in $\sd(\Q_{p,\Box})$
 \begin{align*}
  \alpha_r:\quad  \R\Gamma^{\B_{\crr}^+}_{\synt,?}(X_C,\Q_p(r)) & \simeq \R\Gamma_{\proeet,?}(X_C,\Q_p(r)),\quad r\geq 2d.
\end{align*}
 \end{theorem}
 \begin{proof}Only the second claim requires justification. For the usual cohomology,  this follows from quasi-isomorphism \eqref{tea1} and the fact that the complexes $ \R\Gamma^{\B_{\crr}^+}_{\synt}(X_C,\Q_p(r))$, $\R\Gamma_{\proeet}(X_C,\Q_p(r))$ live in the $[0,2d]$-range.  To see the latter fact in the case $X$ is Stein, note that  using \eqref{kolobrzeg1a} we get $H^{\B_{\crr}^+,i}_{\synt}(X_C,\Q_p(r))=0$, for $i\geq d+1$.  From this and  \eqref{tea1} we get that 
 $H^{i}_{\proeet}(X_C,\Q_p(d+j))=0$, for $d+j\geq i\geq d+1,j\geq 1$, and then, by twisting, that  $H^{i}_{\proeet}(X_C,\Q_p(r))=0$, for $i\geq d+1$, as wanted. Now, for a general partially proper $X$, we need to add $d$ for the analytic dimension of cohomology yielding the range $[0,2d]$, as wanted.
 
  For the cohomology with compact support,   we argue similarly but using \eqref{kolo10} instead of \eqref{kolobrzeg1a} in the case $X$ is Stein. The case of  partially proper $X$
  follows from that by a (co)-\v{C}ech argument. 
 \end{proof}
   
   The above has a version in families. Let $S\in{\rm Perf}_C$ and  let $r\in\N$.  We have  the classical (crystalline) syntomic cohomology in $\sd(\Q_{p}(S)_{\Box})$: 
 \begin{equation}\label{kawa2}
 \R\Gamma^{\B_{\crr}^+}_{\synt,?}(X_{S},\Q_p(r))
:=\big[[\R\Gamma_{\hk,?}(X_C)\otimes^{\LL_{\Box}}_{\breve{C}}\B^+_{\st}(S)]^{N=0,\phi=p^r}\lomapr{\iota_{\hk}\otimes\iota}\R\Gamma_{\dr,?}(X_C/\B_{\dr}^+(S))/F^r\big]
 \end{equation}
 It satisfies the following comparison theorem:
  \begin{theorem}\label{comp01}
{\rm(Period isomorphism in families, \cite[Cor. 7.37]{CN4}, \cite[Prop.\,6.16]{AGN})}

Let $r\in\N$. There is a natural (in $X$ and $S$) quasi-isomorphism in $\sd(\Q_{p}(S)_{\Box})$
 \begin{equation}
\alpha_r: \quad  \tau_{\leq r}\R\Gamma^{\B_{\crr}^+}_{\synt,?}(X_S,\Q_p(r))\simeq \tau_{\leq r}\R\Gamma_{\proeet,?}(X_S,\Q_p(r)).
 \end{equation}
 Moreover, for $S\in \mathrm{sPerf}$,  it yields a natural (in $X$ and $S$) quasi-isomorphism in $\sd(\Q_{p}(S)_{\Box})$
 \begin{align*}
  \alpha_r:\quad  \R\Gamma^{\B_{\crr}^+}_{\synt,?}(X_S,\Q_p(r))  \simeq \R\Gamma_{\proeet,?}(X_S,\Q_p(r)),\quad r\geq 2d.
\end{align*}
 \end{theorem}
\begin{proof}  
The argument is analogous to the one used in the proof of Theorem \ref{comp0}. For the second claim, the required Frobenius
surjectivity follows from the argument in the proof
of Lemma \ref{grey1} below.
\end{proof}
\subsubsection{Variants of syntomic cohomology}

We will need the  following \index{SYN@\SYN}variant of syntomic cohomology in $\sd(\Q_{p}(S)_{\Box})$:
    \begin{equation}
    \R\Gamma^{\B^{[u,v]}}_{\synt,?}(X_S,\Q_p(r))
:=\big[[\R\Gamma_{\hk,?}^{{[u,v]}}(X_S,r)]^{\phi=1}\lomapr{\iota_{\hk}}\R\Gamma^{{[u,v]}}_{\dr,?}(X_S,r)\big],\quad r\in\N.
    \end{equation}
    \begin{lemma}\label{grey1} 
Let $r\in\N$.  There is a natural (in $X$ and $S$) quasi-isomorphism in $\sd(\Q_{p}(S)_{\Box})$:
    \begin{align}\label{change1}\tau_{\leq r}\R\Gamma^{\B^{[u,v]}}_{\synt,?}(X_S,\Q_p(r)) &  \simeq\tau_{\leq r}\R\Gamma^{\B_{\crr}^+}_{\synt,?}(X_S,\Q_p(r)).
    \end{align}
    Moreover, for $S\in \mathrm{sPerf}_C$,  it yields  a quasi-isomorphism in $\sd(\Q_{p}(S)_{\Box})$:
    \begin{align}\label{cieplo1}
 \R\Gamma^{\B^{[u,v]}}_{\synt,?}(X_S,\Q_p(r)) &  \simeq \R\Gamma^{\B_{\crr}^+}_{\synt,?}(X_S,\Q_p(r)),\quad r\geq d.
    \end{align}
    \end{lemma}
    \begin{proof}
     Let  
   $\B_{S^{\flat}}^{[u,\infty]} := W(R^{\flat,+})\langle[p^{\flat}]/p^u\rangle[1/p].$  We set
$$
\R\Gamma_{\hk,?}^{[u,\infty]}(X_S,r):=[\R\Gamma_{\hk,?}(X_C)\{r\}\otimes^{\LL_{\Box}}_{\breve{C}}\B_{S^{\flat},\log}^{[u,\infty]}]^{N=0},
\qquad \B_{S^{\flat},\log}^{[u,\infty]}:=\B_{S^{\flat}}^{[u,\infty]}[U],
$$
where $U$, Frobenius,  and monodromy are defined as in Section~\ref{red1}.
Define yet another variant of syntomic cohomology in $\sd(\Q_{p}(S)_{\Box})$:
    \begin{equation*}
    \R\Gamma^{\B^{[u,\infty]}}_{\synt,?}(X_S,\Q_p(r))
:=\big[[\R\Gamma_{\hk,?}^{{[u,\infty]}}(X_S,r)]^{\phi=1}\lomapr{\iota_{\hk}}\R\Gamma_{\dr,?}(X_S,r)\big].
    \end{equation*}
    
     By
\cite[Sec.\,2.4.2]{CN1}, we have
\[
\begin{aligned}
\phi(\B_{S^\flat}^{[u,\infty]})
  &\subset \B^+_{\crr}(S)
   \subset \B_{S^\flat}^{[u,\infty]},
  && p>2,\\
\phi(\B^+_{\crr}(S))
  &\subset \B_{S^\flat}^{[u,\infty]}
   \subset \B^+_{\crr}(S),
  && p=2.
\end{aligned}
\]
For $p=2$, the first arrow in the following diagram is reversed.
The same convention applies to the corresponding coefficient-ring
map and to all arrows labelled $f'_1$ below.
 The three different variants of syntomic cohomology introduced above are linked via maps
    $$
    \xymatrix{
    \R\Gamma^{\B_{\crr}^+}_{\synt,?}(X_S,\Q_p(r))\ar[r]^{f_1} &  \R\Gamma^{\B^{[u,\infty]}}_{\synt,?}(X_S,\Q_p(r)) \ar[r]^{f_2} &  \R\Gamma^{\B^{[u,v]}}_{\synt,?}(X_S,\Q_p(r))
    }
    $$
    induced by canonical maps $\B^+_{\crr}(S)\to \B_{S^{\flat}}^{[u,\infty]}$,  
$\B_{S^{\flat}}^{[u,\infty]}\to \B_{S^{\flat}}^{[u,v]}$, and $\B_{S^{\flat}}^{[u,\infty]}\to \B_{S^{\flat}}^{[u,v/p]}$ (see \cite[Sec.\,2.4.2]{CN1}). 
    We claim that the map $f_1$ is a quasi-isomorphism and the map $ f_2$ is a  quasi-isomorphism after truncation $\tau_{\leq r}$. To show that, it suffices to prove that the related maps 
  \begin{align*}
&  f^{\prime}_1:\quad [\R\Gamma_{\hk,?}(X_C)\otimes^{\LL_{\Box}}_{\breve{C}}\B^+_{\st}(S)]^{N=0,\phi=p^r}  \to 
  [\R\Gamma_{\hk,?}^{{[u,\infty]}}(X_S,r)]^{\phi=1},\\
&  f^{\prime}_2:\quad [\R\Gamma_{\hk,?}^{{[u,\infty]}}(X_S,r)]^{\phi=1} \to [\R\Gamma_{\hk,?}^{{[u,v]}}(X_S,r)]^{\phi=1}
  \end{align*}
  are quasi-isomorphisms in the wanted ranges.  Or, using the Frobenius-compatible identification induced by
$\exp(NU)$ on Hyodo-Kato cohomology, as in \eqref{evening1}, and 
 removing  $N=0$ and $\log$ simultaneously  on both sides, that so are the maps 
\begin{align*}
&  f^{\prime}_1:\quad [\R\Gamma_{\hk,?}(X_C)\otimes^{\LL_{\Box}}_{\breve{C}}\B^+_{\crr}(S)]^{\phi=p^j}  \to 
  [\R\Gamma_{\hk,?}(X_C)\otimes^{\LL_{\Box}}_{\breve{C}}\B_{S^{\flat}}^{[u,\infty]}]^{\phi=p^j},\quad j\in\Z;\\
&  f^{\prime}_2:\quad\tau_{\leq r} [\R\Gamma_{\hk,?}(X_C)\otimes^{\LL_{\Box}}_{\breve{C}}\B_{S^{\flat}}^{[u,\infty]}]^{\phi=p^r} \to \tau_{\leq r}
 [\R\Gamma_{\hk,?}(X_C)\otimes^{\LL_{\Box}}_{\breve{C}}\B_{S^{\flat}}^{[u,v]}]^{\phi=p^r}.
  \end{align*}
  
  Let us first look at the map $f'_1$. In either case, the inclusions
above imply that Frobenius is null-homotopic on the fiber of the
underlying coefficient map, and hence also after derived tensoring
with $\R\Gamma_{\hk,?}(X_C)$.
Thus $\phi-p^j$ is invertible on this fiber, proving that $f'_1$
is a quasi-isomorphism for every $j\in\Z$.

   Concerning the map $f_2^{\prime}$ and the usual cohomology, we first pass to cohomology in degree $i$ and then, since $H^i_{\hk}(X_C)$ is  a countable limit of finite rank $\phi$-isocrystals over $\breve{C}$, we  assume that  the Hyodo-Kato cohomology groups are of finite rank.   Let $j\in\N$. We then claim that the map
  \begin{equation}\label{deszcz1}
  H^i_{\hk}(X_C)\{j\}\otimes^{\LL_{\Box}}_{\breve{C}}\B_{S^{\flat}}^{[u,v]}\lomapr{1-\phi}H^i_{\hk}(X_C)\{j\}\otimes^{\LL_{\Box}}_{\breve{C}}\B_{S^{\flat}}^{[u,v/p]}
  \end{equation}
  is surjective for $i< j$.  Indeed, by Proposition \ref{leje1}, the complex \eqref{deszcz1} computes the cohomology of the vector bundle $\se_i$ on $X_{\FF,S^{\flat}}$ associated to $ H^i_{\hk}(X_C)\{j\}$. Our claim now follows from the fact that the slopes of Frobenius on $  H^i_{\hk}(X_C)$ are $\leq i$ (see \cite[proof of Prop.\,5.20]{CN4}) hence the slopes of $\se_i$ are $> 0$ and $H^1(X_{\FF,S^{\flat}},\se_i)=0$, as wanted. 
  
      The analogous Frobenius map over \(\B_{S^\flat}^{[u,\infty]}\) is also surjective for \(i<j\), by \cite[Prop.  II.2.5]{FS}. Hence, it suffices to show that, for $i\in\N, j\geq -1$,  the map 
 \begin{align*}
 f^{\prime}_2:\quad (H^i_{\hk,?}(X_C)\otimes^{\LL_{\Box}}_{\breve{C}}\B_{S^{\flat}}^{[u,\infty]})^{\phi=p^j}  \to 
   (H^i_{\hk,?}(X_C)\otimes^{\LL_{\Box}}_{\breve{C}}\B_{S^{\flat}}^{[u,v]})^{\phi=p^j}
  \end{align*}
  is an isomorphism. But in the case $S=C$ this follows  from \cite[Prop.\,3.2]{Ber} and the same Frobenius regularization argument applies for general \(S\), since our period algebras are spectral.
  
      This proves   \eqref{change1} for the usual cohomology; we derive  the case of compactly supported  cohomology  from the case of usual cohomology  by a colimit argument. 

    Concerning the last sentence of our lemma, assuming $S\in \mathrm{sPerf}_C$, we start with  the case of compactly supported cohomology (assuming again finite rank Hyodo-Kato cohomology).
  Arguing as above, we see that  the map
    \begin{equation}
  H^i_{\hk,c}(X_C)\{j\}\otimes^{\LL_{\Box}}_{\breve{C}}\B_{S^{\flat}}^{[u,v]}\lomapr{1-\phi}H^i_{\hk,c}(X_C)\{j\}\otimes^{\LL_{\Box}}_{\breve{C}}\B_{S^{\flat}}^{[u,v/p]}
  \end{equation}
  is surjective for $j \geq d$ using the fact that the slopes of Frobenius on $ H^i_{\hk,c}(X_C)$ are in the $[i-d,d]$ range (use Poincar\'e duality for  Hyodo-Kato cohomology to flip to the usual cohomology). There is a subtle point for $j=d$:  to treat  the slope-zero summands
  of $\se_i$, we use
$H^1_{\proeet}(S,\Q_p)=0$, since $S$ is strictly totally disconnected. The analogous surjectivity over $\B_{S^\flat}^{[u,\infty]}$
follows by the argument of \cite[proof of Prop.~II.2.5]{FS}.

    For the usual cohomology, first assume that $X$ is Stein.
Then $H^i_{\hk}(X_C)=0$ for $i>d$, and for $i\leq d$ its
Frobenius slopes are at most $i\leq d$. Hence the same argument
gives the untruncated comparison for $r\geq d$.
The general case follows by \v{C}ech descent.
    \end{proof}
    \begin{remark}
  For arbitrary \(S\in\mathrm{Perf}_C\),  the comparison \eqref{cieplo1}  holds for \(r>d\). The same argument applies.  At \(r=d\), slope-zero summands may occur in the relevant vector bundle and 
$H^1(X_{\mathrm{FF},S^\flat},\mathcal O)
\simeq H^1_{\mathrm{pro\acute et}}(S,\mathbf Q_p) $
need not vanish for general \(S\). 
    \end{remark}
    \begin{remark}
 Bosco in \cite[Th.\,6.3]{GB2} considered the following  \index{SYN@\SYN}variant of syntomic cohomology in $\sd(\Q_{p}(S)_{\Box})$:
    \begin{equation*}
    \R\Gamma^{\FF}_{\synt}(X_S,\Q_p(r))
:=\big[[\R\Gamma_{\hk}^{\B}(X_S,r)]^{\phi=1}\lomapr{\iota_{\hk}}\R\Gamma^{\B}_{\dr}(X_S,r)\big],\quad r\in\N.
    \end{equation*}
\begin{lemma}
The canonical map $\B_{S^{\flat}}\to \B_{S^\flat}^{[u,v]}$ induces a morphism in $\sd(\Q_{p}(S)_{\Box})$
$$
  \R\Gamma^{\FF}_{\synt}(X_S,\Q_p(r))\to   \R\Gamma^{\B^{[u,v]}}_{\synt}(X_S,\Q_p(r)).
$$
This is a quasi-isomorphism. 
\end{lemma}
\begin{proof}
Arguing as in the proof of Lemma \ref{grey1}, it suffices to show that the induced morphism
$$
 [ H^i_{\hk}(X_C)\{r\}\otimes^{\LL_{\Box}}_{\breve{C}}\B_{S^{\flat}}\lomapr{1-\phi}H^i_{\hk}(X_C)\{r\}\otimes^{\LL_{\Box}}_{\breve{C}}\B_{S^{\flat}}]  \to 
  [H^i_{\hk}(X_C)\{r\}\otimes^{\LL_{\Box}}_{\breve{C}}\B_{S^{\flat}}^{[u,v]}\lomapr{1-\phi}H^i_{\hk}(X_C)\{r\}\otimes^{\LL_{\Box}}_{\breve{C}}\B_{S^{\flat}}^{[u,v/p]}]
$$
is a quasi-isomorphism in the case $ H^i_{\hk}(X_C)$ is of finite rank. But this follows from  Proposition \ref{sobota12}. 
\end{proof}
    \end{remark}

 \subsubsection{Syntomic $\phi$-modules over $\B^{\rm FF}$}\label{tiger1}
  Let $X$ be a smooth partially proper  rigid analytic variety over $K$.
 \begin{definition}Let $r\in\N$.   Let $S\in {\rm Perf}_C$. 
\begin{enumerate}
 \item \index{SYN@\SYN}Set
\begin{equation*}
 \R\Gamma^{\B}_{\synt,?}(X_S,\Q_p(r)):= [\R\Gamma^{\B}_{\hk,?}(X_S,r)\lomapr{\iota_{\hk}} \R\Gamma^{\B}_{\dr,?}(X_S,r)].
 \end{equation*}
 This is a nuclear  $\phi$-module over $\B^{\rm FF}_{S^{\flat}}$. We call it a {\em syntomic module}. We have 
 $$ \R\Gamma^{\B}_{\synt,?}(X_S,\Q_p(r))=( \R\Gamma^{[u,v]}_{\synt,?}(X_S,\Q_p(r)),\phi),$$ 
\index{SYN@\SYN}where
 $$
  \R\Gamma^{[u,v]}_{\synt,?}(X_S,\Q_p(r)):= [\R\Gamma^{[u,v]}_{\hk,?}(X_S,r)\lomapr{\iota_{\hk}} \R\Gamma^{[u,v]}_{\dr,?}(X_S,r)].
 $$
 \item  The (nuclear) {\em syntomic sheaves}  on $X_{\FF,S^{\flat}}$ are defined \index{E@\EEE}by 
 $$
 \se_{\synt,?}(X_S,\Q_p(r)):=\se_{\FF}(\R\Gamma^{\B}_{\synt,?}(X_S,\Q_p(r))).
 $$
 \end{enumerate}
\end{definition}
 We have a distinguished triangle in ${\rm QCoh}(X_{{\rm FF},S^{\flat}})$
 \begin{equation}\label{def1}
  \se_{\synt,?}(X_S,\Q_p(r))\to \se_{\hk,?}(X_S,r)\lomapr{\iota_{\hk}} \se_{\dr,?}(X_S,r).
 \end{equation}

   \begin{proposition}\label{comp2}
Let $r\geq 2d$.  We have natural (in $X$ and $S$) quasi-isomorphisms in $\sd(\Q_{p}(S)_{\Box})$:
\begin{align*}
 \R\Gamma(X_{\FF,S^{\flat}},\se_{\synt,?}(X_S,\Q_p(r))) & \simeq\R\Gamma^{\B^{[u,v]}}_{\synt,?}(X_S,\Q_p(r)),\\
\R\Gamma(X_{\FF,S^{\flat}},\se_{\synt,?}(X_S,\Q_p(r))) & \simeq\R\Gamma_{\proeet,?}(X_S,\Q_p(r)).
\end{align*}
 \end{proposition}
 \begin{proof}The first quasi-isomorphism follows from Proposition \ref{leje1}.  The second quasi-isomorphism follows by applying \(\R\Gamma(X_{\mathrm{FF},S^\flat},-)\) to Proposition  \ref{china1} and using Proposition \ref{comp2b}.
 \end{proof}

 \section{Pro-\'etale  complexes on the Fargues-Fontaine curve}
In this section we define quasi-coherent sheaves on the Fargues-Fontaine curve representing
$p$-adic (geometric) pro-\'etale cohomology (with $\Q_p(r)$-coefficients)
of smooth partially proper rigid analytic varieties 
(Proposition~\ref{comp2b}) 
and prove a comparison theorem with the quasi-coherent sheaves representing syntomic cohomology
(Proposition~\ref{china1}).  This follows from a comparison theorem (Theorem~\ref{hot1} and Corollary~\ref{hot1K})
for the ${\mathbb B}$-period sheaf, which   amounts
to a comparison theorem on $Y_{\rm FF}$ instead of $X_{\rm FF}$, i.e., to untangling
comparison theorems from the action of $\phi$ (that this could be done without much pain came to us as a surprise).

 \subsection{Definitions} We start with definitions. 
 
\subsubsection{Twisted coefficients} Let $S\in {\rm Perf}_C$. 
Let $n,k\geq 0$. Define the line bundle $\so(n,k)$ on $X_{\rm FF,S^{\flat}}$ by the exact sequence of $\so_{\FF, S^{\flat}}$-modules
$$
0\to \so(n,k)\to \so(n)\to i_{\infty, *}(\so/t^k)\to 0,
$$
where the first map is an inclusion. The sheaf $\so(n,n)$  will be the target of our trace maps. Note that  $\so(n,k)$ is just $\so(n-k)$ with (Galois-)Tate twist $k$; in particular, we have $H^0(X_{\rm FF,S^{\flat}},\so(n,n))=\underline{\Q_p}(S)(n)$. 

   On the level of $\phi$-modules over $\B^{\FF}_{S^{\flat}}$, the sheaf $\so(n,k)$ is the module $\B_{S^{\flat}}\{n,k\}$ represented by  the module  
$\B_{S^{\flat}}^{[u,v]}\{n,k\}$ defined by the exact sequence
\begin{equation}\label{alter3}
0\to \B_{S^{\flat}}^{[u,v]}\{n,k\}\to \B_{S^{\flat}}^{[u,v]}\{n\}\to \B_{S^{\flat}}^{[u,v]}\{n\}/t^k\to 0,
\end{equation}
where the first map is an inclusion and $\B_{S^{\flat}}^{[u,v]}\{n\}$ denotes $\B_{S^{\flat}}^{[u,v]}$ with Frobenius divided by $p^n$. We have $ \B_{S^{\flat}}^{[u,v]}\{n,k\} \simeq \B_{S^{\flat}}^{[u,v]}\{n-k\}(k)$ as a Frobenius module. 
Note that the Frobenius map:
$$
\phi:  \B_{S^{\flat}}^{[u,v]}\{n,k\} \otimes^{\LL_{\Box}}_{\B_{S^{\flat}}^{[u,v]},\phi}\B_{S^{\flat}}^{[u,v/p]}\to \B_{S^{\flat}}^{[u,v]}\{n,k\} \otimes^{\LL_{\Box}}_{\B_{S^{\flat}}^{[u,v]}}\B_{S^{\flat}}^{[u,v/p]}
$$
is an isomorphism because, under the preceding isomorphism, it is identified with the Frobenius on $\B_{S^{\flat}}^{[u,v/p]}\{n-k\}$.

\subsubsection{Pro-\'etale modules and sheaves} \label{sing1} 
Let $X$ be a smooth partially proper dagger variety over $K$. 
 For  $r\in\N$, $v^{\prime}=v,v/p$, and $S\in {\rm Perf}_C$, we set 
$$
\R\Gamma^{{[u,v^{\prime}]}}_{\proeet,?}(X_S,\Q_p(r)) :=\R\Gamma_{\proeet,?}(X_S,{\mathbb B}^{[u,v^{\prime}]})(r),
$$
\index{BRING@\BRING}where ${\mathbb B}^{[u,v^{\prime}]}$ denotes the relative period sheaf corresponding to ${\B}^{[u,v^{\prime}]}$ (see \cite[Def.\,2.7, Sec.\,2.2]{GB1}  for a description of the condensed structure on these modules). We will need the following fact. 
\begin{lemma}The canonical map
$$\R\Gamma_{\proeet,?}(X_S,{\mathbb B}^{[u,v]})\otimes^{\LL_{\Box}}_{\B_{S^{\flat}}^{[u,v]}}\B_{S^{\flat}}^{ [u,v/p]}{\to}\R\Gamma_{\proeet,?}(X_S,{\mathbb B}^{[u,v/p]})$$
is a quasi-isomorphism. 
\end{lemma}
\begin{proof} It suffices to argue for the ordinary cohomology. By pro-\'etale descent, it suffices to show that, for a family of affinoid perfectoids $\{S_i\to S\},$ $i\in I$, the canonical map
$$
(\prod_{I}\B_{S^{\flat}_i}^{{[u,v]}})\otimes^{\LL_{\Box}}_{\B_{S^{\flat}}^{[u,v]}}\B_{S^{\flat}}^{[u,v/p]}\to \prod_{I}\B_{S^{\flat}_i}^{{[u,v/p]}}
$$
is a quasi-isomorphism. (We used here the fact that $\R\Gamma_{\proeet,?}({S_i},{\mathbb B}^J)\simeq \B^J_{S_i^{\flat}}$.) But this follows from the fact that this tensor product commutes with derived limits (see the proof of Lemma \ref{fun1-kol}) and the canonical map
$$
\B_{S^{\flat}_i}^{{[u,v]}}\otimes^{\LL_{\Box}}_{\B_{S^{\flat}}^{[u,v]}}\B_{S^{\flat}}^{[u,v/p]}\to \B_{S^{\flat}_i}^{{[u,v/p]}}
$$
is an isomorphism.  To see the last claim,  we compute
$$
\B_{S^{\flat}_i}^{{[u,v]}}\otimes^{\LL_{\Box}}_{\B_{S^{\flat}}^{[u,v]}}\B_{S^{\flat}}^{[u,v/p]}\simeq \B_{S^{\flat}_i}^{{[u,v]}}\otimes^{\LL}_{(\Z[T],\Z)_{\Box}}(\Z[T],\Z[T])_{\Box}\simeq\B_{S^{\flat}_i}^{{[u,v/p]}},
$$
where we wrote 
$\B_{S^{\flat}}^{[u,v/p]}\simeq \B_{S^{\flat}}^{[u,v]}\langle f\rangle$ 
for $f=(p/[p^{\flat}]^{p/v})\in \B_{S^{\flat}}^{[u,v]}$.
\end{proof}

    We define the {\em pro-\'etale  modules} as the pairs
 \begin{align*}
&  \R\Gamma^{\B}_{\proeet,?}(X_S,\Q_p(r))  :=( \R\Gamma_{\proeet,?}^{{[u,v]}}(X_S,\Q_p(r)),\phi),\\
&   \phi:  \R\Gamma_{\proeet,?}^{{[u,v]}}(X_S,\Q_p(r))  \to  \R\Gamma_{\proeet,?}^{{[u,v/p]}}(X_S,\Q_p(r)),
 \end{align*}
 where the Frobenius $\phi$ is induced by the   Frobenius $\phi: {\mathbb B}^{[u,v]}\to {\mathbb B}^{[u,v/p]}$. It
 yields a quasi-isomorphism in $\sd(\B^{[u,v/p]}_{S^{\flat}, \Box})$
 $$\phi:  \R\Gamma_{\proeet,?}^{{[u,v]}}(X_S,\Q_p(r))\otimes^{\LL_{\Box}}_{\B_{S^{\flat}}^{[u,v]},\phi}\B_{S^{\flat}}^{[u,v/p]}\stackrel{\sim}{\to } \R\Gamma_{\proeet,?}^{{[u,v/p]}}(X_S,\Q_p(r)).$$
 Indeed, it suffices  to show that the Frobenius map
 $$
  \phi:  \R\Gamma_{\proeet,?}(X_S,{\mathbb B}^{[u,v]})\otimes^{\LL_{\Box}}_{\B_{S^{\flat}}^{[u,v]},\phi}\B_{S^{\flat}}^{[u,v/p]}\stackrel{\sim}{\to } \R\Gamma_{\proeet,?}(X_S, {\mathbb B}^{[u,v/p]})
$$
 is a quasi-isomorphism. But this follows directly from \cite[Lemma 4.8]{GB2}.

 The pairs  $\R\Gamma^{\B}_{\proeet,?}(X_S,\Q_p(r))$
  define nuclear $\phi$-complexes over $\B^{\rm FF}_{S^{\flat}}$
(and $(\phi,\sg_K)$-complexes when $S=\Spa(C,\so_C)$),
which we will call {\em pro-\'etale modules}. For the nuclear property use \cite[Lemma 6.15]{GB2} plus preservation of nuclearity by countable products and finite limits for the usual cohomology.
  The case of  cohomology with compact support follows since  colimits preserve nuclearity.    We will denote \index{E@\EEE}by
$$
\se_{\proeet,?}(X_S,\Q_p(r)):=\se_{\FF}(\R\Gamma^{\B}_{\proeet,?}(X_S,\Q_p(r)))
$$
the corresponding nuclear quasi-coherent sheaves on $X_{\FF,S^{\flat}}$. We will call them {\em pro-\'etale  sheaves}. 
Pro-\'etale  modules and sheaves  are functors on ${\rm Perf}_C$.

 \begin{proposition}\label{comp2b}
  We have a quasi-isomorphism, natural in $X$ and $S$, in $\sd(\Q_{p}(S)_{\Box})$:
$$\R\Gamma(X_{\FF,S^{\flat}},\se_{\proeet,?}(X_S,\Q_p(r)))  \simeq\R\Gamma_{\proeet,?}(X_S,\Q_p(r)).
$$
 \end{proposition}
 \begin{proof}By Proposition \ref{leje1} we have quasi-isomorphisms natural in $X$ and $S$
 \begin{align*}
 \R\Gamma(X_{\FF,S^{\flat}},\se_{\proeet,?}(X_S,\Q_p(r)))   & \simeq [\R\Gamma_{\proeet,?}^{{[u,v]}}(X_S,\Q_p(r))\lomapr{\phi-1}\R\Gamma_{\proeet,?}^{{[u,v/p]}}(X_S,\Q_p(r))]\\
 & \simeq [\R\Gamma_{\proeet,?}(X_S,{\mathbb B}^{[u,v]})(r)\lomapr{\phi-1}\R\Gamma_{\proeet,?}(X_S,{\mathbb B}^{[u,v/p]})(r)]\\
  & \stackrel{\sim}{\leftarrow}\R\Gamma_{\proeet,?}(X_S,\Q_p(r)).
 \end{align*}
 Here, in the last quasi-isomorphism, we have used the exact sequence (see \cite[Lemma 2.23]{CN1})
 $$
 \phantom{XXXXXXXXXXXXXXX}0\to\Q_p\to {\mathbb B}^{[u,v]}\lomapr{\phi-1}{\mathbb B}^{[u,v/p]}\to 0\phantom{XXXXXXXXXXXXX} \qedhere
 $$
 \end{proof}

 \subsection{Comparison theorems on the Fargues-Fontaine curve} 
We move now to the comparison theorems on the two curves of Fargues-Fontaine. 
\subsubsection{Comparison theorem on the $X_{\rm FF}$-curve} We start with the ``bottom'' curve. Let $X$ be a smooth partially proper variety over $K$, of dimension $d$. 
 \begin{proposition} \label{china1}
Let $r\geq 2d$.
There is a quasi-isomorphism, natural in $X$ and $S$, in ${\rm QCoh}(X_{\FF,S^{\flat}})$:
\begin{equation}
\alpha_r:\quad \se_{\synt,?}(X_S,\Q_p(r)) \simeq \se_{\proeet,?}(X_S,\Q_p(r)).
\end{equation}
 \end{proposition}
\begin{proof}  It suffices  to construct a natural quasi-isomorphism of $\phi$-modules over $\B^{\FF}_{S^{\flat}}$
$$
\R\Gamma^{\B}_{\synt,?}(X_S,\Q_p(r))\simeq \R\Gamma^{\B}_{\proeet,?}(X_S,\Q_p(r)).
$$
That is, a natural quasi-isomorphism of pairs
$$
(\R\Gamma^{{[u,v]}}_{\synt,?}(X_S,\Q_p(r)),\phi)\simeq  (\R\Gamma^{{[u,v]}}_{\proeet,?}(X_S,\Q_p(r)),\phi).
$$
 But this  follows 
   from a ``Frobenius untwisted'' version of  Theorem \ref{comp01} presented in Theorem~\ref{hot1} below. We just have to argue that we can drop truncations in \eqref{kin1}: but this follows from the fact that both sides live in degrees $[0,2d]$, which can be seen   as in the proof of Theorem \ref{comp0}.    \end{proof}
   \begin{remark}
   We did not list the truncated version of Theorem \ref{comp01} in Proposition \ref{china1} because the issue of truncation vis-\`a-vis localization is a subtle one.
   \end{remark}
   \subsubsection{Comparison theorem on the $Y_{\rm FF}$-curve} We pass now to the ``top'' curve. 
\begin{theorem}{\rm (Comparison theorem on the $Y_{\FF}$-curve)} \label{hot1} 
Let $X$ be a smooth partially proper variety over $K$, of dimension $d$.  Let $r\geq 0$. We have quasi-isomorphisms, natural in $X$ and $S$ and compatible with Frobenius, in $\sd(\B^{[u,v]}_{S^{\flat},\Box})$ 
and $\sd(\B^{[u,v/p]}_{S^{\flat},\Box})$, respectively:
\begin{align}\label{kin1}
\tau_{\leq r}\R\Gamma_{\proeet,?}(X_S,{\mathbb B}^{[u,v]})(r) & \simeq \tau_{\leq r}[\R\Gamma^{{[u,v]}}_{\hk,?}(X_S,r)\lomapr{\iota_{\hk}} \R\Gamma_{\dr,?}^{{[u,v]}}(X_S,r)],\\
\R\Gamma_{\proeet,?}(X_S,{\mathbb B}^{[u,v/p]})(r) & \simeq \R\Gamma^{{[u,v/p]}}_{\hk,?}(X_S,r).\notag
\end{align}
For $r\geq 2d$, the truncations can be omitted in the first quasi-isomorphism.
\end{theorem}
\begin{proof} For $v^{\prime}=v,v/p$, we define $F^r{\mathbb B}^{[u,v^{\prime}]}:=t^r{\mathbb B}^{[u,v^{\prime}]}$.
We clearly have the isomorphism
$
t^r:  {\mathbb B}^{[u,v^{\prime}]}(r)\stackrel{\sim}{\to} F^r{\mathbb B}^{[u,v^{\prime}]}\{r\}.
$
We want to construct quasi-isomorphisms, natural in $X$ and $S$ and compatible with Frobenius,
 in $\sd(\B^{[u,v]}_{S^{\flat},\Box})$ and $\sd(\B^{[u,v/p]}_{S^{\flat},\Box})$, respectively:
\begin{align*}
\tau_{\leq r}\R\Gamma_{\proeet,?}(X_S, F^r{\mathbb B}^{[u,v]})\{r\} & \simeq \tau_{\leq r}[\R\Gamma^{{[u,v]}}_{\hk,?}(X_S,r)\to \R\Gamma_{\dr,?}^{{[u,v]}}(X_S,r)],\\
\R\Gamma_{\proeet,?}(X_S, F^r{\mathbb B}^{[u,v/p]})\{r\} & \simeq \R\Gamma^{{[u,v/p]}}_{\hk,?}(X_S,r).
\end{align*}
 
  For the usual cohomology, these quasi-isomorphisms were constructed  in \cite[Sec.\,7]{CN4}. They are  not explicitly stated there because we almost always carry through the constructions the eigenspaces of Frobenius but, in fact, the latter   can be dropped
 as they are only used to pass between various period rings and here we work with one fixed period ring.  For the gist of the construction the interested reader should consult  the diagram (7.16) (with the top row moved a step lower and with  added  $[u,v]$-decoration), its refinement (7.31), Section 7.4 in general, and diagram (7.36) (with decoration changed again to $[u,v]$) in particular.  Alternatively, a more detailed description of this construction, for usual
cohomology over $C$ and $r\geq 2d$, is given in
\cite[proof of Prop.\,4.7(2)]{DoN}: it proceeds
via the Fontaine--Messing comparison followed by the passage
from crystalline to Hyodo--Kato cohomology.

  The case of   compactly supported cohomology follows now easily from the case of usual cohomology by taking  colimits and finite limits. The assertion for $r\geq 2d$ follows from the cohomological bounds used in the proof of Theorem \ref{comp0}.
\end{proof}
 
\begin{remark} We fix the notation used in Corollary \ref{hot1K}. Let $X$ be a smooth partially proper variety over $K$, let $S\in{\rm Perf}_C$, and let $r\geq0$.

(1) For a compact interval $I=[u,v]\subset(0,\infty)$ with rational endpoints, set 
$$\R\Gamma_{\dr,?}^{I}(X_S,r):=\bigoplus_{\Z(I)}(\R\Gamma_{\dr,?}(X)\otimes^{\LL_{\Box}}_K{\mathbb B}^+_{\dr}(S))/F^r,$$
where $\Z(I):=\{n\in\Z| \phi^{-n}(y_{\infty})\in Y^I_{\FF,S^{\flat}}\}$. 
 The Hyodo-Kato morphism $\iota^{I}_{\hk}$ in degree $n\in\Z(I)$ is defined by precomposing the usual Hyodo-Kato morphism $\iota_{\hk}$ with $\phi^{-n}$. It is 
 ${\mathbb B}^{[u,v]}(S)$-linear via the composition 
 $${\mathbb B}^{[u,v]}(S)\lomapr{\phi^{-n}}{\mathbb B}^{[p^nu,p^nv]}(S)\to {\mathbb B}^{[p^nu,p^nv]}(S)/\xi_{\infty}^r\simeq {\mathbb B}^+_{\dr}(S)/t^r,$$
 where $\xi_{\infty}$ is a generator of the ideal defining $y_{\infty}$.
 
  (2) We also \index{BRING@\BRING}set
\begin{align*}
\R\Gamma_{\dr}^{{\mathbb B}}(X_S,r) & :=\prod_{\Z}((\R\Gamma_{\dr}(X)\otimes^{\LL_{\Box}}_K{\mathbb B}^+_{\dr}(S))/F^r),\\
\R\Gamma^{{\mathbb B}}_{\hk}(X_S,r) &:=[\R\Gamma_{\hk}(X_C)\{r\}\otimes^{\LL_{\Box}}_{\breve{C}}{\mathbb B}_{\log}(S)]^{N=0}.
\end{align*}
 \end{remark}

  The following result follows easily from the above theorem though it will not be used in this paper. 
\begin{corollary}{\rm \index{BRING@\BRING}(${\mathbb B}$-comparison theorem)} \label{hot1K} 
Let $X$ be a smooth partially proper variety over $K$.  Let $r\geq 0$.
\begin{enumerate}
\item Let $I=[u,v]\subset (0,\infty)$ be a compact interval with rational endpoints containing the fixed intervals from Section \ref{fixed}.  We have a quasi-isomorphism, natural in $X$ and $S$,
in $\sd(\B^I_{S^{\flat},\Box})$: 
\begin{align*}
\tau_{\leq r}\R\Gamma_{\proeet,?}(X_S,{\mathbb B}^{I})(r) & \simeq \tau_{\leq r}[\R\Gamma^{{I}}_{\hk,?}(X_S,r)\lomapr{\iota^I_{\hk}} \R\Gamma_{\dr,?}^{I}(X_S,r)].
\end{align*}
This quasi-isomorphism is  also compatible with Frobenius, i.e., the following diagram commutes (we set $u^{\prime}=u/p, v^{\prime}=v/p$)
$$
\xymatrix@R=5mm{
\tau_{\leq r}\R\Gamma_{\proeet,?}(X_S,{\mathbb B}^{[u,v]})(r) \ar[r]^-{\sim}\ar[d]^{\phi}_{\wr}
&  \tau_{\leq r}[\R\Gamma^{{[u,v]}}_{\hk,?}(X_S,r)\lomapr{\iota^{[u,v]}_{\hk}} \R\Gamma_{\dr,?}^{[u,v]}(X_S,r)]\ar[d]^{\phi}_{\wr}\\
\tau_{\leq r}\R\Gamma_{\proeet,?}(X_S,{\mathbb B}^{[u^{\prime},v^{\prime}]})(r) \ar[r]^-{\sim}
&  \tau_{\leq r}[\R\Gamma^{{[u^{\prime},v^{\prime}]}}_{\hk,?}(X_S,r)\lomapr{\iota^{[u^{\prime},v^{\prime}]}_{\hk}} \R\Gamma_{\dr,?}^{[u^{\prime},v^{\prime}]}(X_S,r)].
}
$$
\item We have a quasi-isomorphism, natural in $X$ and $S$ and compatible with Frobenius, 
\index{HK@\HK}\index{DR@\DR}in  $\sd(\B_{S^{\flat},\Box})$: 
\begin{align*}
\tau_{\leq r}\R\Gamma_{\proeet}(X_S,{\mathbb B})(r) & \simeq \tau_{\leq r}[\R\Gamma^{\mathbb B}_{\hk}(X_S,r)\lomapr{\iota^{\mathbb B}_{\hk}} \R\Gamma^{\mathbb B}_{\dr}(X_S,r)]
\end{align*}
For $r\geq 2d$, this yields a quasi-isomorphism
\begin{equation}\label{bosco-thesis}
\R\Gamma_{\proeet}(X_S,{\mathbb B})(r)  \simeq [\R\Gamma^{\mathbb B}_{\hk}(X_S,r)\lomapr{\iota^{\mathbb B}_{\hk}} \R\Gamma_{\dr}^{\mathbb B}(X_S,r)].
\end{equation}
\end{enumerate}
\end{corollary}

\begin{proof} The first quasi-isomorphism is proven in the same way as Theorem \ref{hot1}. The second one follows from the first one by passing to limits (we use here heavily that we may assume the Hyodo-Kato cohomology to be of finite rank) once we know that $\R^i\lim_{I\subset (0,\infty)}{\mathbb B}^{I}(S)=0, i>0$, where the limit is taken over compact intervals $I$ with rational endpoints (recall that we have  $\lim_{I\subset (0,\infty)}{\mathbb B}^{I}(S)\simeq {\mathbb B}(S)$). But this was checked in \cite[proof of Lemma 2.41]{GB2}. 
\end{proof}
\begin{remark} (1) For $r\geq 2d$, as an immediate consequence of \eqref{bosco-thesis}, we get the quasi-isomorphism:
$$
\R\Gamma_{\proeet}(X_S,{\mathbb B})(r)[1/t]  \simeq \R\Gamma^{\mathbb B}_{\hk}(X_S,r)[1/t].
$$

 (2)   Bosco in \cite{GB2} proved a  version of the comparison quasi-isomorphism \eqref{bosco-thesis}, where the torsion on the right-hand side is incorporated into the left-hand side via the $\LL\eta_t$ operator. 
\end{remark}

\section{Poincar\'e dualities on the Fargues-Fontaine curve} 
We are now ready to state and prove pro-\'etale  duality on the Fargues-Fontaine curve. The same techniques also allow us to prove a pro-\'etale K\"unneth formula. 

\subsection{Hyodo-Kato and de Rham dualities} \label{sing3} Let $X$ be a  smooth partially proper rigid analytic variety over $K$, of dimension $d$. 
\subsubsection{De Rham dualities} 
 Recall the following dualities (see \cite[Cor.\,5.18, Th.\,5.23, Cor.\,5.26]{AGN}).
 \begin{proposition}\label{derhamduality} Let $L=K,C$.
 \begin{enumerate}
 \item  {\em (Serre duality)}     There is a   \index{TRACE@\TRACE}trace map of solid $L$-modules
$${\rm Tr}_{\rm coh}: \R\Gamma_c(X_L,\Omega^d)[d]\to L. $$
 The pairing
$$
\R\Gamma(X_L,\Omega^j)\otimes^{\LL_{\Box}}_L\R\Gamma_c(X_L,\Omega^{d-j})[d]\to \R\Gamma_c(X_L,\Omega^{d})[d]\lomapr{{\rm Tr}_{\rm coh}} L
$$
is perfect, i.e., it yields the quasi-isomorphism in $\sd(L_{\Box})$:
 $$
 \R\Gamma(X_L,\Omega^j)\simeq \R\uHom_{L_{\Box}}(\R\Gamma_c(X_L,\Omega^{d-j})[d],L).
 $$

 \item {\em (De Rham duality)}
There are natural \index{TRACE@\TRACE}trace maps in $\sd(L_{\Box})$ and $L_{\Box}$, respectively:
$${\rm Tr}_{\dr}: \, \R\Gamma_{\dr,c}(X_L)[2d]\to L,\quad {\rm Tr}_{\dr}: \, H^{2d}_{\dr,c}(X_L)\to L. 
$$
The  pairing in $\sd(L_{\Box})$
$$
\R\Gamma_{\dr}(X_L)\otimes^{\LL_{\Box}}_L \R\Gamma_{\dr,c}(X_L)[2d]\to \R\Gamma_{\dr,c}(X_L)[2d]\xrightarrow{{\rm Tr}_{\dr}} L
$$
is a perfect duality, i.e., we have an induced quasi-isomorphism in $\sd(L_{\Box})$
\begin{align*}
& \R\Gamma_{\dr}(X_L)\stackrel{\sim}{\to} \R\uHom_{L_{\Box}} (\R\Gamma_{\dr,c}(X_L)[2d],L).
\end{align*}
\item {\em (Filtered de Rham duality)} More generally, let $r,r^{\prime}\in\N, r+r^{\prime}=d$. The  pairing in $\sd(L_{\Box})$
$$
(\R\Gamma_{\dr}(X_L)/F^{r^{\prime}+1})\otimes^{\LL_{\Box}}_L F^r\R\Gamma_{\dr,c}(X_L)[2d]\to \R\Gamma_{\dr,c}(X_L)[2d]\xrightarrow{{\rm Tr}_{\dr}} L
$$
is a perfect duality, i.e., we have induced quasi-isomorphisms in $\sd(L_{\Box})$
\begin{align*}
& \R\Gamma_{\dr}(X_L)/F^{r^{\prime}+1}\stackrel{\sim}{\to} \R\uHom_{L_{\Box}} (F^{r}\R\Gamma_{\dr,c}(X_L)[2d],L),\\
& F^{r^{\prime}+1}\R\Gamma_{\dr}(X_L)\stackrel{\sim}{\to} \R\uHom_{L_{\Box}} (\R\Gamma_{\dr,c}(X_L)/F^{r}[2d],L).
\end{align*}
  \end{enumerate}
 \end{proposition}

\subsubsection{$\B^+_{\dr}$-dualities} 
 The duality for $\B^+_{\dr}$-cohomology has a slightly different form. For  $r\geq d$, a natural trace map in $\sd(\B^+_{\dr,\Box})$ is the \index{TRACE@\TRACE}morphism
 $$
 {\rm Tr}_{\B^+_{\dr}}:\quad F^{r}\R\Gamma_{\dr,c}(X_C/\B^+_{\dr})[2d]\to  F^{r-d}\B^+_{\dr}
 $$
induced by the coherent trace
$$  {\rm Tr}_{\rm coh}\otimes{\rm Id}: \R\Gamma_c(X,\Omega^d)[d]\otimes_K^{\LL_{\Box}}F^{r-d}\B^+_{\dr}\to  F^{r-d}\B^+_{\dr}.
$$
 \begin{corollary} {\rm (Filtered $\B^+_{\dr}$-duality \cite[Cor. 5.27]{AGN})}\label{bdrduality} 
Let $r,r^{\prime}\geq d, s=r+r^{\prime}-d$.  The  pairing in $\sd(\B^+_{\dr,\Box})$
$$
F^{r^{\prime}}\R\Gamma_{\dr}(X_C/\B^+_{\dr})\otimes^{\LL_{\Box}}_{\B^+_{\dr}} F^r\R\Gamma_{\dr,c}(X_C/\B^+_{\dr})[2d]\to F^{r^{\prime}+r}\R\Gamma_{\dr,c}(X_C/\B^+_{\dr})[2d]\xrightarrow{{\rm Tr}_{\B^+_{\dr}}} F^s\B^+_{\dr}
$$
is a perfect duality, i.e., we have an  induced quasi-isomorphism in $\sd(\B^+_{\dr,\Box})$
\begin{align*}
F^{r^{\prime}} \R\Gamma_{\dr}(X_C/\B^+_{\dr})\stackrel{\sim}{\to} \R\uHom_{\B^+_{\dr,\Box}} (F^{r}\R\Gamma_{\dr,c}(X_C/\B^+_{\dr})[2d], F^s\B^+_{\dr}).
\end{align*}
\end{corollary}

    We will need a variant of  the above result. To state it, take $r,r^{\prime}\geq d, s=r+r^{\prime}-d$ and consider the  pairing in $\sd(\B^+_{\dr,\Box})$
   \begin{equation}\label{ias223}
(\R\Gamma_{\dr}(X_C/\B^+_{\dr})  /F^{r^{\prime}})\otimes^{\LL_{\Box}}_{\B^+_{\dr}} (F^r\R\Gamma_{\dr,c}(X_C/\B^+_{\dr})/t^s)[2d-1] \to  F^{s}\B^+_{\dr}
\end{equation}
defined as the composition
\begin{align*}
(\R\Gamma_{\dr}(X_C/\B^+_{\dr}) & /F^{r^{\prime}})\otimes^{\LL_{\Box}}_{\B^+_{\dr}} (F^r\R\Gamma_{\dr,c}(X_C/\B^+_{\dr})/t^s)[2d-1] \stackrel{\cup}{ \to} F^r\R\Gamma_{\dr,c}(X_C/\B^+_{\dr})/t^s[2d-1]\\
 & \to
(F^{r-d}\B_{\mathrm{dR}}^+/F^s\B_{\mathrm{dR}}^+)[-1]
\xrightarrow{\partial}
F^s\B_{\mathrm{dR}}^+.
\end{align*}
Here the third morphism is the boundary map induced by the exact sequence
$$
0\to F^{s}\B^+_{\dr}\lomapr{\can}F^{r-d}\B^+_{\dr}\to F^{r-d}\B^+_{\dr}/t^s\to 0
$$

 \begin{corollary} \label{bdrduality1} 
  The  pairing \eqref{ias223} is a perfect duality, i.e., we have an  induced quasi-isomorphism in $\sd(\B^+_{\dr,\Box})$
\begin{align}\label{niedziela1}
\gamma:\quad \R\Gamma_{\dr}(X_C/\B^+_{\dr})/F^{r^{\prime}}\stackrel{\sim}{\to} \R\uHom_{\B^+_{\dr,\Box}} (F^{r}\R\Gamma_{\dr,c}(X_C/\B^+_{\dr})/t^s[2d-1],F^s\B^+_{\dr}).
\end{align}
The same claim  holds functorially for $S\in{\rm Perf}_C$,
with $\B^+_{\dr}$ replaced by $\B^+_{\dr}(S)$.
\end{corollary}
\begin{proof} Consider the following diagram with distinguished triangles for columns
$$
\xymatrix@R=4mm@C=6mm{
F^{r^{\prime}}\R\Gamma_{\dr}(X_C/\B^+_{\dr})\ar[d] \ar[r]^-{\sim} & \R\uHom_{\B^+_{\dr,\Box}} (F^{r}\R\Gamma_{\dr,c}(X_C/\B^+_{\dr})[2d],F^s\B^+_{\dr})\ar[d]\\
\R\Gamma_{\dr}(X_C/\B^+_{\dr})\ar[d] \ar[r]^-{\sim} & \R\uHom_{\B^+_{\dr,\Box}} (t^s\R\Gamma_{\dr,c}(X_C/\B^+_{\dr})[2d],F^s\B^+_{\dr})\ar[d]\\
\R\Gamma_{\dr}(X_C/\B^+_{\dr})/F^{r^{\prime}}\ar[r]^-{\gamma} & \R\uHom_{\B^+_{\dr,\Box}} (F^{r}\R\Gamma_{\dr,c}(X_C/\B^+_{\dr})/t^s[2d-1],F^s\B^+_{\dr})
}
$$
We claim that this is a map of distinguished triangles. Indeed, 
cup product followed by trace gives a pairing
$$
\R\Gamma_{\dr}(X_C/\B^+_{\dr})
\otimes^{\LL_{\Box}}_{\B^+_{\dr}}
F^r\R\Gamma_{\dr,c}(X_C/\B^+_{\dr})[2d]\longrightarrow F^{r-d}\B^+_{\dr}.
$$
Its restrictions defining the top and middle horizontal arrows
take values in $F^s\B^+_{\dr}$ and agree on their common restriction.
Their adjoints therefore form the upper commutative square.
Using the standard mapping-cone models for the quotients, the
adjoint of the induced map on the third terms is the quotient
cup product followed by trace, with values in $(J/I)[-1]$, for $I=F^s\B^+_{\dr}, J=F^{r-d}\B^+_{\dr}$, 
and then the boundary $(J/I)[-1]\to I$.
Indeed, this boundary is the standard cone projection, with
the usual cone signs.
This is precisely the pairing \eqref{ias223}; hence the induced
third arrow is $\gamma$, as wanted.

 The middle arrow in the above diagram  is the de Rham duality map ($\B^+_{\dr}$-linearized) and the top arrow is the $\B^+_{\dr}$-duality map from Corollary \ref{bdrduality}. Both are quasi-isomorphisms (see Proposition \ref{derhamduality}). Hence so is the bottom duality map, as wanted.
\end{proof}

   The duality map \eqref{niedziela1} can be lifted to the Fargues-Fontaine curve:  the pairing \eqref{ias223} induces a pairing of ${\B_{S^{\flat}}^{[u,v]}}$-modules 
 $$
 \R\Gamma^{{[u,v]}}_{\dr}(X_S,r)\otimes^{\LL_{\Box}}_{\B_{S^{\flat}}^{[u,v]}}(F^{r^{\prime}}\R\Gamma^{{[u,v]}}_{\dr,c}(X_S)/t^s)\to \B_{S^{\flat}}^{[u,v]}\{s,s\}[-2d+1], 
 $$
 which, in turn, induces a pairing of nuclear $\phi$-modules over $\B^{\FF}_{S^{\flat}}$
 $$
  \R\Gamma^{\B}_{\dr}(X_S,r)\otimes^{\LL}_{\B^{\FF}_{S^{\flat}}}(F^{r^{\prime}}\R\Gamma^{\B}_{\dr,c}(X_S)/t^s)\to\B_{S^{\flat}}\{s,s\}[-2d+1],
 $$
 where we set $F^{r^{\prime}}\R\Gamma^{\B}_{\dr,c}(X_S)/t^s:=(F^{r^{\prime}}\R\Gamma^{{[u,v]}}_{\dr,c}(X_S)/t^s,0)$. 
This descends to a pairing on $X_{\FF,S^{\flat}}$:
 \begin{equation}\label{niedziela5}
 \se_{\dr}(X_S,r)\otimes^{\LL}_{\so}i_{\infty,*}(F^{r^{\prime}}\R\Gamma_{\dr,c}(X_S/\B^+_{\dr})/t^s) \to \so(s,s)[-2d+1],
 \end{equation}
  where we set $\R\Gamma_{\dr,c}(X_S/\B^+_{\dr}):=\R\Gamma_{\dr,c}(X_K)\otimes^{\LL_{\Box}}_K\B^+_{\dr}(S)$.
The pairing \eqref{niedziela5}  induces a duality map in ${\rm QCoh}(X_{\FF,S^{\flat}})$: 
 \begin{equation}\label{niedziela3}
 \gamma_{X_S}:\quad  \se_{\dr}(X_S,r){\to} {\mathbb D}(i_{\infty,*}(F^{r^{\prime}}\R\Gamma_{\dr,c}(X_S/\B^+_{\dr})/t^s)[2d-1],\so(s,s)),
 \end{equation}
 where we \index{D@\bD}set 
$${\mathbb D}(-,-):=\R\Hhom_{{\rm QCoh}(X_{\FF,S^{\flat}})}(-,-).$$ 
 
\begin{lemma}\label{kolobrzeg5a}The duality map \eqref{niedziela3} is a quasi-isomorphism. 
 \end{lemma}
 \begin{proof}We need to show that the duality map
$$
\gamma^{\FF}_{X_S}: \R\Gamma^{\B}_{\dr}(X_S,r){\to} \R\uHom_{\B^{\FF}_{S^{\flat}}}(F^{r^{\prime}}\R\Gamma^{\B}_{\dr,c}(X_S)/t^s[2d-1],\B_{S^{\flat}}\{s,s\})
$$
is a quasi-isomorphism in $\sd(\B^{\FF}_{S^{\flat}})$. Or, passing to solid $\B^{\prime}:=\B_{S^{\flat}}^{[u,v]}$-modules, that the duality map
$$
\gamma_{X_S}:  \R\Gamma^{{[u,v]}}_{\dr}(X_S,r){\to} \R\uHom_{\B^{\prime}_{\Box}}(F^{r^{\prime}}\R\Gamma^{{[u,v]}}_{\dr,c}(X_S)/t^s[2d-1],\B^{\prime})
$$
is a quasi-isomorphism in $\sd(\B^{\prime}_{\Box})$. But this is Corollary \ref{bdrduality1}.
\end{proof}
\subsubsection{Hyodo-Kato duality} This is based on \cite[Sec.\,5.4.1, Th.\,5.30]{AGN}.  
There exists a natural \index{TRACE@\TRACE}trace map in $\sd_{\phi,N,\sg_K}(\breve{C}_{\Box})$:
$$
{\rm Tr}_X: \R\Gamma_{\hk,c}(X_C)\to \breve{C}\{-d\}[-2d].
$$
The pairing  in $\sd_{\phi,N,\sg_K}(\breve{C}_{\Box})$ ($s=r+r^{\prime}-d$)
\begin{equation}\label{ias22}
\R\Gamma_{\hk}(X_C)\{r\}\otimes^{\LL_{\Box}}_{\breve{C}}\R\Gamma_{\hk,c}(X_C)\{r^{\prime}\}\to \R\Gamma_{\hk,c}(X_C)\{r+r^{\prime}\}\lomapr{{\rm Tr}_X}\breve{C}\{s\}[-2d]
\end{equation}
is perfect, i.e., 
it induces a quasi-isomorphism  in $\sd_{\phi,N,\sg_K}(\breve{C}_{\Box})$
\begin{equation}\label{HKduality}
\R\Gamma_{\hk}(X_C)\{r\}\simeq \R\Hhom_{\sd_{\phi,N,\sg_K}(\breve{C}_{\Box})}(\R\Gamma_{\hk,c}(X_C)\{r^{\prime}\},\breve{C}\{s\}[-2d]),
\end{equation}
where the internal $\Hom$ is just $\R\Hhom_{\breve{C}_{\Box}}(\R\Gamma_{\hk,c}(X_C),\breve{C}[-2d])$ -- the internal $\Hom$ in $\sd(\breve{C}_{\Box})$ -- equipped with 
$(\phi,N,\sg_K)$-actions via  $\R\Gamma_{\hk,c}(X_C)\{r^{\prime}-s\}$. 

 The above duality can be lifted to the Fargues-Fontaine curve: the pairing \eqref{ias22} induces a pairing of ${\B_{S^{\flat}}^{[u,v]}}$-modules
 $$
 \R\Gamma^{{[u,v]}}_{\hk}(X_S,r)\otimes^{\LL_{\Box}}_{\B_{S^{\flat}}^{[u,v]}}\R\Gamma^{{[u,v]}}_{\hk,c}(X_S,r^{\prime})\to \breve{C}\{s\}\otimes_{\breve{C}}^{\LL_{\Box}}\B_{S^{\flat}}^{[u,v]}[-2d], 
 $$
 which, in turn, induces a pairing of nuclear $\phi$-modules over $\B^{\FF}_{S^{\flat}}$
 $$
  \R\Gamma^{\B}_{\hk}(X_S,r)\otimes^{\LL}_{\B^{\FF}_{S^{\flat}}}\R\Gamma^{\B}_{\hk,c}(X_S,r^{\prime})\to\B_{S^{\flat}}\{s\}[-2d].
 $$
This descends to a pairing on $X_{\FF,S^{\flat}}$:
 $$
 \se_{\hk}(X_S,r)\otimes^{\LL}_{\so}\se_{\hk,c}(X_S,r^{\prime})\to \so(s)[-2d],
 $$
 which induces a duality map in ${\rm QCoh}(X_{\FF,S^{\flat}})$: 
 \begin{equation}\label{niedziela11}
 \gamma_{X_S}:\quad  \se_{\hk}(X_S,r){\to} {\mathbb D}(\se_{\hk,c}(X_S,r^{\prime})[2d],\so(s)).
 \end{equation}

 \begin{lemma}\label{kolobrzeg4a}
The map $ \gamma_{X_S}$ above is a quasi-isomorphism in ${\rm QCoh}(X_{\FF,S^{\flat}})$. 
\end{lemma}
\begin{proof} Since $\B_{S^{\flat}}^{[u,v]}$ is $\B^{[u,v]}_{S^{\flat},{\rm an}}$-complete (see \cite[Lemma 3.24]{And21}), by Remark \ref{zet1}, 
we may pass from $\B^{[u,v],+}_{S^{\flat}}$ to $\Z$, i.e.,  to $\B^{\FF}_{S^{\flat}}$-modules.  Hence we need to show that the duality map
$$
\gamma^{\FF}_{X_S}: \R\Gamma^{\B}_{\hk}(X_S,r){\to} \R\uHom_{\B^{\FF}_{S^{\flat}}}(\R\Gamma^{\B}_{\hk,c}(X_S,r^{\prime})[2d],\B_{S^{\flat}}\{s\})
$$
is a quasi-isomorphism in $\sd(\B^{\FF}_{S^{\flat}})$. Or, passing to solid 
$\B^{\prime}:=\B_{S^{\flat}}^{[u,v]}$-modules, that the duality map
\begin{equation}\label{ias23}
\gamma_{X_S}:  \R\Gamma^{{[u,v]}}_{\hk}(X_S,r){\to} \R\uHom_{\B^{\prime}_{\Box}}(\R\Gamma^{{[u,v]}}_{\hk,c}(X_S,r^{\prime})[2d],\B^{\prime})
\end{equation}
is a quasi-isomorphism in $\sd(\B^{\prime}_{\Box})$. We claim that, for that,
 it suffices to check that, for $j\in\N$, the duality map on the level of cohomology groups
\begin{equation}\label{ias231}
\gamma^j_{X_S}: H^{{[u,v]},j}_{\hk}(X_S,r){\to} \uHom_{\B^{\prime}_{\Box}}(H^{{[u,v]},2d-j}_{\hk,c}(X_S,r^{\prime}),\B^{\prime})
\end{equation}
is an isomorphism in $\B^{\prime}_{\Box}$. 
Indeed, passing to cohomology in \eqref{ias23}, we need to check that the duality map
$$
\gamma^j_{X_S}: H^{{[u,v]},j}_{\hk}(X_S,r){\to} H^j(\R\uHom_{\B^{\prime}_{\Box}}(\R\Gamma^{{[u,v]}}_{\hk,c}(X_S,r^{\prime})[2d],\B^{\prime}))
$$
is an isomorphism in $\B^{\prime}_{\Box}$. But  $H^i_{\hk,c}(X_C)$ is a direct sum of copies of $\breve{C}$  hence we have 
$$
H^j(\R\uHom_{\B^{\prime}_{\Box}}(\R\Gamma^{{[u,v]}}_{\hk,c}(X_S,r^{\prime})[2d],\B^{\prime}))\simeq \uHom_{\B^{\prime}_{\Box}}(H^{{[u,v]},2d-j}_{\hk,c}(X_S,r^{\prime}),\B^{\prime}),
$$
as wanted.

   To prove \eqref{ias231}, we observe that, 
    for $i\in\Z$, we have the natural isomorphisms\footnote{We can ignore the Galois action here.}
\begin{equation}\label{evening1}
H^{{[u,v]},j}_{\hk,?}(X_S,i)  \simeq (H^j_{\hk,?}(X_C)\{i\}\otimes^{{\Box}}_{\breve{C}}\B^{\prime}_{\log})^{N=0}\stackrel{\sim}{\leftarrow} H^j_{\hk,?}(X_C)\{i\}\otimes^{{\Box}}_{\breve{C}}\B^{\prime}.
\end{equation}
Here the second quasi-isomorphism is defined by the map $\exp(NU)$ (this makes sense because the monodromy operator on the Hyodo-Kato cohomology 
$ H^j_{\hk,?}(X_C)$ is  nilpotent). For the first quasi-isomorphism 
$$
H^{{[u,v]},j}_{\hk,?}(X_S,i) =H^j([\R\Gamma_{\hk,?}(X_C)\{i\}\otimes^{\LL_{\Box}}_{\breve{C}}\B^{\prime}_{\log}]^{N=0}) 
\simeq (H^j_{\hk,?}(X_C)\{i\}\otimes^{{\Box}}_{\breve{C}}\B^{\prime}_{\log})^{N=0}
$$
we used the fact that 
$$
H^j(\R\Gamma_{\hk,?}(X_C)\{i\}\otimes^{\LL_{\Box}}_{\breve{C}}\B^{\prime}_{\log}) \simeq H^{j}_{\hk,?}(X_C)\{i\}\otimes^{{\Box}}_{\breve{C}}\B^{\prime}_{\log},
$$
that $N$ is nilpotent on $H^{j}_{\hk,?}(X_C)$ (so we can do d\'evissage by the kernels of the action of $N$), and  that $\B^{\prime}\stackrel{\sim}{\to}[\B^{\prime}_{\log}]^{N=0}$.

  It is easy to check that the maps in \eqref{evening1} are compatible with products. Hence we can write the duality map \eqref{ias231} as the Hyodo-Kato pairing 
$$
\gamma^j_{X_S}: H^j_{\hk}(X_C)\otimes^{{\Box}}_{\breve{C}}\B^{\prime}{\to} \uHom_{\B^{\prime}_{\Box}}( H^{2d-j}_{\hk,c}(X_C)\otimes^{{\Box}}_{\breve{C}}\B^{\prime},\B^{\prime}).
$$
To show that it is an isomorphism in $\B^{\prime}_{\Box}$ it suffices thus to invoke the Hyodo-Kato duality \eqref{HKduality} and to show that  the natural map
\begin{align*}
\uHom_{\breve{C}_{\Box}}( H^{2d-j}_{\hk,c}(X_C),\breve{C})\otimes^{{\Box}}_{\breve{C}}\B^{\prime} & {\to }\uHom_{\breve{C}_{\Box}}( H^{2d-j}_{\hk,c}(X_C), \B^{\prime})
\end{align*}
is an isomorphism  in $\B_{\Box}^{\prime}$. 
But this is an isomorphism by \cite[Th.\,3.40]{RR1}
since $\B^{\prime} $ is a Banach space over $\breve{C}$. 
\end{proof}

\subsection{Syntomic  duality}Let $X$ be a smooth partially proper rigid analytic variety over $K$ of dimension $d$. Let $S\in {\rm Perf}_C$. 
Recall that syntomic $\phi$-modules over $\B^{\FF}_{S^{\flat}}$ are defined as (see Sec. \ref{tiger1})
\begin{equation*}
 \R\Gamma^{\B}_{\synt,?}(X_S,\Q_p(r)):= [\R\Gamma^{\B}_{\hk,?}(X_S,r)\lomapr{\iota_{\hk}} \R\Gamma^{\B}_{\dr,?}(X_S,r)],
 \end{equation*}
 where the Hyodo-Kato map is described by diagram \eqref{tiger2}.  The Hyodo-Kato and de Rham cup products are compatible with this diagram, i.e., the Hyodo-Kato morphism is compatible with products. Indeed,
Lemma 2.6 of \cite{CN4}, applied to tensor powers, implies that
the Hyodo-Kato section of \cite[Th.\,2.12]{CN4} is multiplicative,
and this compatibility is preserved by the constructions in
\cite[Secs.\,2.3, 4.1.2]{CN4}; 
it extends to the usual--compactly supported pairings by the
mapping-fiber construction of compact support
(see \cite[Sec.\,2.1, Prop.\,3.19]{AGN}).
 Hence we get  a cup product on the syntomic
$\phi$-modules: 
$$
  \R\Gamma^{\B}_{\synt}(X_S,\Q_p(r))  \otimes^{\LL_{\Box}}_{\B^{\FF}_{S^{\flat}}}\R\Gamma^{\B}_{\synt,c}(X_S,\Q_p(r^{\prime}))  \to \R\Gamma^{\B}_{\synt,c}(X_S,\Q_p(r+r^{\prime})).
$$
  This product can be described by an analogous product on the $\B^{\prime}:=\B_{S^{\flat}}^{[u,v]}$-chart:
\begin{equation}\label{cup11}
  \R\Gamma^{{[u,v]}}_{\synt}(X_S,\Q_p(r))  \otimes^{\LL_{\Box}}_{\B^{\prime}}\R\Gamma^{{[u,v]}}_{\synt,c}(X_S,\Q_p(r^{\prime}))  \to \R\Gamma^{{[u,v]}}_{\synt,c}(X_S,\Q_p(r+r^{\prime})).
\end{equation}
  It is compatible with the products on $\R\Gamma^{{[u,v]}}_{\hk,?}(X_S,i)$ and $ F^i\R\Gamma^{\B}_{\dr,?}(X_S/\B^{\prime})$.
   Here we \index{DR@\DR}defined $F^i\R\Gamma_{\dr,?}(X_S/\B^{\prime})$ as $F^i\R\Gamma_{\dr,?}(X_C/\B^+_{\dr})$ with $\B^+_{\dr}$ replaced by $\B^{\prime}$. We write $\B^{\prime}_{\log}:=\B_{S^{\flat},\log}^{[u,v]}$.

  Let $s\geq d$. There is a trace map
 $$
 {\rm Tr}_X: \quad   \R\Gamma^{\B}_{\synt,c}(X_S,\Q_p(s))\to \B^{\FF}_{S^{\flat}}\{s-d,s-d\}[-2d]
$$
 defined on the $\B^{\prime}$-chart via the \index{TRACE@\TRACE}trace map
 \begin{equation}\label{trace11}
  {\rm Tr}^{[u,v]}_X: \quad   \R\Gamma^{{[u,v]}}_{\synt,c}(X_S,\Q_p(s))\to \B^{\prime}\{s-d,s-d\}[-2d],
 \end{equation}
which is compatible with the Hyodo-Kato and de Rham trace maps. 
The map $  {\rm Tr}^{[u,v]}_X$ is defined using the exact sequence
$$
 H^{{[u,v]},2d}_{\synt,c}(X_S,\Q_p(s))\to H^{2d}_{\hk,c}(X_S,s)\lomapr{\iota_{\hk}} H^{2d}_{\dr,c}(X_S,s),
$$
which can be written more explicitly as the exact sequence
\begin{equation}\label{lunch1}
 H^{{[u,v]},2d}_{\synt,c}(X_S,\Q_p(s))\to (H^{2d}_{\hk,c}(X_C)\{s\}\otimes^{\LL_{\Box}}_{\breve{C}}\B^{\prime}_{\log})^{N=0}\lomapr{\iota_{\hk}} H^{2d}_{\dr,c}(X)\otimes_K^{\LL_{\Box}}(\B^{\prime}\{s-d\}/F^{s-d}).
\end{equation}
Using the Hyodo-Kato and de Rham trace maps, which are compatible by \cite[Th.\,5.30(1)]{AGN}, 
$$
{\rm Tr}_X: H^{2d}_{\hk,c}(X_C)\{s\}{\to}\breve{C}\{s-d\},\quad {\rm Tr}_X: H^{2d}_{\dr,c}(X){\to}K,
$$
\eqref{lunch1} yields  a map 
$$
 H^{{[u,v]},2d}_{\synt,c}(X_S,\Q_p(s))\to {\rm Ker}(\B^{\prime}\{s-d\}\to \B^{\prime}\{s-d\}/F^{s-d})=\B^{\prime}\{s-d,s-d\},
$$
hence the trace \eqref{trace11}, as wanted. 

 For  $ s:=r+r^{\prime}-d$, the above can be lifted to the Fargues-Fontaine curve: the cup product \eqref{cup11} and trace map \eqref{trace11}  induce a pairing of $\B^{\prime}$-modules
 $$
 \R\Gamma^{{[u,v]}}_{\synt}(X_S,\Q_p(r))\otimes^{\LL_{\Box}}_{\B^{\prime}}\R\Gamma^{{[u,v]}}_{\synt,c}(X_S,\Q_p(r^{\prime}))\lomapr{\cup}\R\Gamma^{{[u,v]}}_{\synt,c}(X_S,\Q_p(r+r^{\prime})) \lomapr{{\rm Tr}^{[u,v]}_X} \B^{\prime}\{s,s\}[-2d], 
 $$
 which, in turn, induces a pairing of nuclear $\phi$-modules over $\B^{\FF}_{S^{\flat}}$
 $$
  \R\Gamma^{\B}_{\synt}(X_S,\Q_p(r))\otimes^{\LL}_{\B^{\FF}_{S^{\flat}}}\R\Gamma^{\B}_{\synt,c}(X_S,\Q_p(r^{\prime}))\lomapr{\cup} \R\Gamma^{\B}_{\synt,c}(X_S,\Q_p(r+r^{\prime}))\lomapr{{\rm Tr}_X}\B_{S^{\flat}}\{s,s\}[-2d].
 $$
 This descends to a pairing  in ${\rm QCoh}(X_{\FF,S^{\flat}})$:
$$
\se_{\synt}(X_S,\Q_p(r))\otimes^{\LL}_{\so}\se_{\synt,c}(X_S,\Q_p(r^{\prime}))\lomapr{\cup} \se_{\synt,c}(X_S,\Q_p(r+r^{\prime}))\lomapr{{\rm Tr}_X}\so(s,s)[-2d],
$$
which induces a natural map in ${\rm QCoh}(X_{\FF,S^{\flat}})$
\begin{equation}\label{niedziela10}
\gamma_{X_S}:\quad  \se_{\synt}(X_S,\Q_p(r))\to  {\mathbb D}(\se_{\synt,c}(X_S,\Q_p(r^{\prime}))[2d],\so(s,s)).
\end{equation}
\begin{theorem}{\rm (Syntomic Poincar\'e duality on the Fargues-Fontaine curve)}\label{curve-duality} 

Let $r,r^{\prime}\geq d, s:=r+r^{\prime}-d$. 
The map $\gamma_{X_S}$ is a quasi-isomorphism in ${\rm QCoh}(X_{\FF,S^{\flat}})$. 
\end{theorem}
\begin{proof}
It is enough  to show this in $\phi$-modules over $\B^{\rm FF}_{S^{\flat}}$ for the corresponding map  
\begin{equation}
\gamma_{X_S}:\quad \R\Gamma^{\B}_{\synt}(X_S,\Q_p(r))\to \R\uHom_{\B^{\rm FF}_{S^{\flat}}}(\R\Gamma^{\B}_{\synt,c}(X_S,\Q_p(r^{\prime}))[2d],\B_{S^{\flat}}\{s,s\}).
\end{equation}
Or in $\sd(\B^{[u,v]}_{S^{\flat},\Box})$ for the induced map
$$
\gamma^{[u,v]}_{X_S}:\quad \R\Gamma^{[u,v]}_{\synt}(X_S,\Q_p(r))\to  \R\uHom_{\B^{\prime}_{\Box}}(\R\Gamma^{[u,v]}_{\synt,c}(X_S,\Q_p(r^{\prime}))[2d],\B^{\prime}(s)).
$$
But for that,  it is enough to check 
that base changes of  $\gamma^{[u,v]}_{X_S}$ to both $\B^{\prime}[1/t]$ and $\B^{\prime}/t$ are quasi-isomorphisms in $\sd(\B^{\prime}_{\Box})$. This last claim requires a bit of justification. We have the exact sequence of solid $\B^{\prime}$-modules
$$
0\to \B^{\prime}\to \B^{\prime}[1/t]\to \B^{\prime}[1/t]/\B^{\prime}\to 0.
$$
Hence it suffices to check that base changes of  $\gamma^{[u,v]}_{X_S}$ to both $\B^{\prime}[1/t]$ and $\B^{\prime}[1/t]/\B^{\prime}$ are quasi-isomorphisms. Writing $\B^{\prime}[1/t]/\B^{\prime}=\colim_n(\B^{\prime}/t^n)$ and using the fact that the tensor products commute with filtered colimits, we see that  it suffices to check that base changes of  $\gamma^{[u,v]}_{X_S}$ to both $\B^{\prime}[1/t]$ and $\B^{\prime}/t^i$ are quasi-isomorphisms. Finally, by d\'evissage, we can drop $i$ to $1$, as wanted.

 For the first base change, we  have quasi-isomorphisms in $\sd(\B^{\prime}_{\Box})$
\begin{align*}
 \R\Gamma^{{[u,v]}}_{\synt}(X_S,\Q_p(r))[1/t] & \stackrel{\sim}{\to} \R\Gamma^{{[u,v]}}_{\hk}(X_S,r)[1/t],\\
\R\uHom_{\B^{\prime}_{\Box}}(\R\Gamma^{{[u,v]}}_{\synt,c}(X_S,\Q_p(r^{\prime}))[2d],\B^{\prime})[1/t] &  \stackrel{\sim}{\leftarrow} \R\uHom_{\B^{\prime}_{\Box}}(\R\Gamma^{{[u,v]}}_{\hk,c}(X_S,r^{\prime})[2d], 
\B^{\prime}(s))[1/t].
\end{align*}
And $\gamma^{[u,v]}_{X_S}$ is just the canonical map
$$
\gamma_{X_S}: \R\Gamma^{{[u,v]}}_{\hk}(X_S,r)[1/t] \to \R\uHom_{\B^{\prime}_{\Box}}(\R\Gamma^{{[u,v]}}_{\hk,c}(X_S,r^{\prime})[2d], \B^{\prime})[1/t]
$$
induced by the Hyodo-Kato pairing \eqref{ias22}. Since it is compatible with $t$-action, it suffices to show that 
the canonical map
$$
\gamma_{X_S}: \R\Gamma^{{[u,v]}}_{\hk}(X_S,r) \to \R\uHom_{\B^{\prime}_{\Box}}(\R\Gamma^{{[u,v]}}_{\hk,c}(X_S,r^{\prime})[2d],\B^{\prime})
$$
is a quasi-isomorphism in $\sd(\B^{\prime}_{\Box})$. But this was shown  in \eqref{ias23}, in  the proof of Lemma \ref{kolobrzeg4a}.

   For the base change to $\B^{\prime}/t$, write $S={\rm Spa}(R,R^+)$; then $\B^{\prime}/t=R$. We
claim that we have a quasi-isomorphism in $\sd(\B^{\prime}_{\Box})$
\begin{align}\label{morning1}
 \R\Gamma^{{[u,v]}}_{\synt,?}(X_S,  \Q_p(r))\otimes^{\LL_{\Box}}_{\B^{\prime}} R & 
 \simeq F^r\R\Gamma_{\dr,?}(X_S/\B^{\prime})\otimes^{\LL_{\Box}}_{\B^{\prime}} R.
 \end{align}
 To show 
 \eqref{morning1} we compute:
\begin{align*}
 \R\Gamma^{{[u,v]}}_{\synt,?}(X_S,  \Q_p(r))\otimes^{\LL_{\Box}}_{\B^{\prime}} R & 
 =[\R\Gamma^{{[u,v]}}_{\hk,?}(X_S,r)\lomapr{\iota_{\hk}}\R\Gamma^{{[u,v]}}_{\dr,?}(X_S,r)]\otimes^{\LL_{\Box}}_{\B^{\prime}} R\\
  &  \stackrel{\sim}{\to}[\R\Gamma^{{[u,v]}}_{\hk,?}(X_S,r)\otimes^{\LL_{\Box}}_{\B^{\prime}} R\verylomapr{\iota_{\hk}\otimes{\rm Id}}\R\Gamma^{{[u,v]}}_{\dr,?}(X_S,r)\otimes^{\LL_{\Box}}_{\B^{\prime}} R]
\end{align*}
Then we use the following commutative diagram
$$
\xymatrix@R=4mm{
\R\Gamma^{{[u,v]}}_{\hk,?}(X_S,r)\otimes^{\LL_{\Box}}_{\B^{\prime}} R\ar[r]\ar[d]^{\iota_{\hk}}_{\wr} & \R\Gamma^{{[u,v]}}_{\dr,?}(X_S,r)\otimes^{\LL_{\Box}}_{\B^{\prime}} R\ar@{=}[d]\\
 \R\Gamma_{\dr,?}(X_S/\B^{\prime})\otimes^{\LL_{\Box}}_{\B^{\prime}} R\ar[r] & \R\Gamma^{{[u,v]}}_{\dr,?}(X_S,r)\otimes^{\LL_{\Box}}_{\B^{\prime}} R\\
 F^r\R\Gamma_{\dr,?}(X_S/\B^{\prime})\otimes^{\LL_{\Box}}_{\B^{\prime}} R\ar[r] \ar[u]& 0\ar[u]
 }
$$
It defines quasi-isomorphisms between the mapping fibers of the rows yielding \eqref{morning1}. The quasi-isomorphism in the above diagram needs a justification: take the composition
$$
(\R\Gamma_{\hk,?}(X_C)\{r\}\otimes^{\LL_{\Box}}_{\breve{C}}\B^{\prime})\otimes^{\LL_{\Box}}_{\B^{\prime}}R \stackrel{\sim}{\to}\R\Gamma^{{[u,v]}}_{\hk,?}(X_S,r)\otimes^{\LL_{\Box}}_{\B^{\prime}} R  \lomapr{\iota_{\hk}}\R\Gamma_{\dr,?}(X_S/\B^{\prime})\otimes^{\LL_{\Box}}_{\B^{\prime}} R.
$$
 The first arrow is the inverse of evaluation at $U=0$,
which is a quasi-isomorphism by \eqref{evening1}.
Since the coefficient map $\B^{\prime}_{\log}\to R$
sends $U=\log([p^{\flat}]/p)$ to $0$, the composite
is the usual Hyodo--Kato comparison after extension
of scalars from $C$ to $R$, hence a quasi-isomorphism.

  The quasi-isomorphism \eqref{morning1}  is compatible with products. It also  identifies  the syntomic trace  with
the filtered de Rham trace: both are induced by the compatible
Hyodo--Kato and de Rham traces used in \eqref{trace11}. Moreover, from  \eqref{morning1} we get the quasi-isomorphisms in $\sd(\B^{\prime}_{\Box})$
 \begin{align*}
  \R\uHom_{\B^{\prime}_{\Box}}(\R\Gamma^{{[u,v]}}_{\synt,c}(X_S,\Q_p(r^{\prime}))[2d],\B^{\prime})\otimes^{\LL_{\Box}}_{\B^{\prime}}R 
 & \simeq  \R\uHom_{\B^{\prime}_{\Box}}(\R\Gamma^{{[u,v]}}_{\synt,c}(X_S,\Q_p(r^{\prime}))[2d],R)\\
 &  \simeq  \R\uHom_{\B^{\prime}_{\Box} }( F^{r^{\prime}}\R\Gamma_{\dr,c}(X_S/\B^{\prime})[2d],R).
\end{align*}
It follows that  $\gamma^{[u,v]}_{X_S}$ is just the canonical map
$$
\gamma^{\dr}_{X_S}: \quad  F^{r}\R\Gamma_{\dr}(X_S/\B^{\prime})\otimes^{\LL_{\Box}}_{\B^{\prime}} R\to \R\uHom_{\B^{\prime}_{\Box} }( F^{r^{\prime}}\R\Gamma_{\dr,c}(X_S/\B^{\prime})[2d],R)
$$
induced by the de Rham pairing and trace map. We will show that it is a quasi-isomorphism in $\sd(\B^{\prime}_{\Box})$. 

   We  have
quasi-isomorphisms in $\sd(\B^{\prime}_{\Box})$ 
 (see \cite[Prop.\,3.6(3), Prop.\,3.10(3)]{AGN})
  \begin{align}
 F^r\R\Gamma_{\dr}(X_S/\B^{\prime})\otimes^{\LL_{\Box}}_{\B^{\prime}} R & \xrightarrow[\beta_X]{\sim}  \bigoplus_{i=0}^{d}\R\Gamma(X,\Omega^i)\otimes^{\LL_{\Box}}_KR(r-i)[-i],\\
  F^{r^{\prime}}\R\Gamma_{\dr,c}(X_S/\B^{\prime})\otimes^{\LL_{\Box}}_{\B^{\prime}} R & \xrightarrow[\beta_X]{\sim} \bigoplus_{i=0}^{d}\R\Gamma_c(X,\Omega^i)\otimes^{\LL_{\Box}}_{K}R(r^{\prime}-i)[-i].\notag
 \end{align}
The maps $\beta_X$ are compatible with products and 
 carry the filtered de Rham trace to the projection
onto the $\Omega^d$-summand followed by ${\rm Tr}_{\rm coh}$.
Hence it suffices to show that the Serre duality pairing is perfect for all $i$
\begin{align*}
\R\Gamma(X,\Omega^i) \otimes^{\LL_{\Box}}_KR \to 
 \R\uHom_{R_{\Box}}( \R\Gamma_c(X,\Omega^{d-i})\otimes^{\LL_{\Box}}_{K}R[d],R).
\end{align*}
But this follows from Proposition  \ref{derhamduality}:
\begin{align*}
\R\Gamma(X,\Omega^i) \otimes^{\LL_{\Box}}_KR & \stackrel{\sim}{\to} \R\uHom_{K_{\Box}}( \R\Gamma_c(X,\Omega^{d-i})[d],K)\otimes^{\LL_{\Box}}_{K}R\\
 & \stackrel{\sim}{\to} \R\uHom_{R_{\Box}}( \R\Gamma_c(X,\Omega^{d-i})\otimes^{\LL_{\Box}}_{K}R[d],R).
\end{align*}
 \end{proof}

 \subsection{Syntomic duality: an alternative argument} 
We present here an alternative proof of Theorem \ref{curve-duality} (conditional on the unchecked tedious compatibilities in Lemma \ref{tedious1} below). 
It uses dual modifications to reverse the arrows in the defining syntomic distinguished triangles \eqref{def1}. 

 More precisely, let $j\geq i\geq 0$. We will construct a distinguished triangle in ${\rm QCoh}(X_{\FF,S^{\flat}})$
\begin{equation}\label{alter1}
 \se_{\hk,c}(X_S,i)\otimes^{\LL}_{\so}\so(0,j) \to \se_{\syn,c}(X_S,\Q_p(i))\to i_{\infty,*}F^i\R\Gamma_{\dr,c}(X_S/\B^+_{\dr})/t^j,
\end{equation}
which is a twisted version of \eqref{def1}. To do  that, consider the following map of distinguished triangles
\begin{equation}\label{alter2}
\xymatrix@R=4mm{
 \se_{\hk,c}(X_S,i)\otimes^{\LL}_{\so}\so(0,j) \ar[r] \ar@{-->}[d] & \se_{\hk,c}(X_S,i) \ar@{=}[d]  \ar[r]^-{\iota_{\hk}}& i_{\infty,*}\R\Gamma_{\dr,c}(X_S/\B^+_{\dr})/t^j\ar[d]^{\can}\\
\se_{\syn,c}(X_S,\Q_p(i))\ar[r] & \se_{\hk,c}(X_S,i) \ar[r]^{\iota_{\hk}}  & \se_{\dr,c}(X_S,i)
}
\end{equation}
Here,  the bottom distinguished  triangle is \eqref{def1}; the top one is induced from the distinguished triangle 
$$
 \R\Gamma_{\hk}^{{[u,v]}}(X_S,i)\otimes^{\LL_{\Box}}_{\B_{S^{\flat}}^{[u,v]}}\B_{S^{\flat}}^{[u,v]}\{0,j\}\to \R\Gamma_{\hk}^{{[u,v]}}(X_S,i)\lomapr{\iota_{\hk}} \R\Gamma_{\dr}(X_S/\B^+_{\dr})/t^j
$$
obtained by tensoring the exact sequence \eqref{alter3} for $0,j$ with  $\R\Gamma_{\hk}^{{[u,v]}}(X_S,i)$.
(Recall that $\R\Gamma_{\hk}^{{[u,v]}}(X_C,r)=[\R\Gamma_{\hk}(X_C)\{r\}\otimes^{\LL_{\Box}}_{\breve{C}}
\B^{[u,v]}_{S^{\flat}, \log}]^{N=0}$). The dashed  arrow in  diagram \eqref{alter2} is defined to make the diagram  a map of distinguished triangles. 
The diagram yields   quasi-isomorphisms
\begin{align*}
[ \se_{\hk,c}(X_S,i)\otimes^{\LL}_{\so}\so(0,j) \to \se_{\syn,c}(X_S,\Q_p(i))][1] & \stackrel{\sim}{\leftarrow} [ i_{\infty,*}\R\Gamma_{\dr,c}(X_S/\B^+_{\dr})/t^j\to \se_{\dr,c}(X_S,i)]\\
 & \stackrel{\sim}{\to} i_{\infty,*}F^i\R\Gamma_{\dr,c}(X_S/\B^+_{\dr})/t^j.
\end{align*}
That is, we get a distinguished triangle \eqref{alter1}, as wanted.

    Now, let $r, r^{\prime}\geq d, s=r+r^{\prime}-d$. Consider   
the following diagram in ${\rm QCoh}(X_{\FF,S^{\flat}})$ (note that $s\geq r^{\prime}$)
whose columns are distinguished triangles
\begin{equation}\label{alter4}
\xymatrix@R=4mm{
\se_{\syn}(X_S,\Q_p(r))\ar[r]^-{\gamma^{\synt}_{X_S}}\ar[d] & {\mathbb D}(\se_{\syn,c}(X_S,\Q_p(r^{\prime}))[2d],\so(s,s))\ar[d]\\
 \se_{\hk}(X_S,r) \ar[d] \ar[r]^-{\gamma^{\hk}_{X_S}}_-{\sim} &{\mathbb D}( \se_{\hk,c}(X_S,r^{\prime})\otimes^{\LL}_{\so}\so(0,s)[2d],\so(s,s))   \ar[d]\\
 \se_{\dr}(X_S,r) \ar[r]^-{\gamma^{\dr}_{X_S}}_-{\sim} &{\mathbb D}( i_{\infty,*}F^{r^{\prime}}\R\Gamma_{\dr,c}(X_S/\B^+_{\dr})/t^s[2d-1], \so(s,s))
}
\end{equation}
where the horizontal maps are defined by the syntomic, Hyodo-Kato, and $\B^+_{\dr}$-pairings, 
respectively (see \eqref{niedziela10}, \eqref{niedziela11}, \eqref{niedziela3}).

 Let us assume Lemma \ref{tedious1} below. To prove that the top horizontal arrow  in diagram \ref{alter4} 
is a quasi-isomorphism it suffices to show that so are the two lower arrows. But this follows 
 from Lemma  \ref{kolobrzeg4a} (we used the isomorphism  $\so(0,s)\otimes^{\LL}_{\so}\so(s)\simeq \so(s,s)$) and  Lemma  \ref{kolobrzeg5a}. 
\begin{lemma}\label{tedious1}
Diagram \eqref{alter4} above is a map of distinguished triangles.
\end{lemma}
\subsection{Pro-\'etale duality}  Let $X$ be a smooth partially proper rigid analytic variety over $K$ of dimension $d$. Let $S\in {\rm Perf}_C$. 
We define  a cup product on the pro-\'etale
$\phi$-modules: 
\begin{equation}\label{cup11et}
  \R\Gamma^{\B}_{\proeet}(X_S,\Q_p(r))  \otimes^{\LL_{\Box}}_{\B^{\FF}_{S^{\flat}}}\R\Gamma^{\B}_{\proeet,c}(X_S,\Q_p(r^{\prime}))  \to \R\Gamma^{\B}_{\proeet,c}(X_S,\Q_p(r+r^{\prime}))
\end{equation}
via the cup product on the $\B^{\prime}:=\B_{S^{\flat}}^{[u,v]}$-charts:
$$
  \R\Gamma_{\proeet}(X_S,{\mathbb B}^{[u,v]}(r))  \otimes^{\LL_{\Box}}_{\B^{\prime}}\R\Gamma_{\proeet,c}(X_S,{\mathbb B}^{[u,v]}(r^{\prime}))  \to \R\Gamma_{\proeet,c}(X_S,{\mathbb B}^{[u,v]}(r+r^{\prime}))
$$
induced by the cup product on pro-\'etale cohomology. 
This product is compatible with the syntomic product  (via the comparison quasi-isomorphism from Theorem~\ref{hot1}): to see this it suffices to argue for the usual cohomology and locally, where the comparison map is known to be compatible with products.
  
     Let $s\geq 2d$. We define  a trace map
 \begin{equation}\label{trace11et}
 {\rm Tr}_X: \quad   \R\Gamma^{\B}_{\proeet,c}(X_S,\Q_p(s))\to \B^{\FF}_{S^{\flat}}\{s-d,s-d\}[-2d]
\end{equation}
  as the composition
 $$
  \R\Gamma^{\B}_{\proeet,c}(X_S,\Q_p(s)) \simeq  \R\Gamma^{\B}_{\synt,c}(X_S,\Q_p(s))\lomapr{{\rm Tr}_X} \B^{\FF}_{S^{\flat}}\{s-d,s-d\}[-2d].
 $$
For $S=\Spa(C,\so_C)$, the induced trace on compactly supported pro-\'etale cohomology agrees, via the canonical map to compactly supported \'etale cohomology, with Berkovich's trace map (see \cite[Prop.\,7.17]{AGN}). 

 For  $ r,r^{\prime}\geq d, s:=r+r^{\prime}-d$, the above can be lifted to the Fargues-Fontaine curve: the cup product \eqref{cup11et} and trace map \eqref{trace11et}  induce a pairing of  nuclear $\phi$-modules over $\B^{\FF}_{S^{\flat}}$
 $$
  \R\Gamma^{\B}_{\proeet}(X_S,\Q_p(r))\otimes^{\LL}_{\B^{\FF}_{S^{\flat}}}\R\Gamma^{\B}_{\proeet,c}(X_S,\Q_p(r^{\prime}))\lomapr{\cup} \R\Gamma^{\B}_{\proeet,c}(X_S,\Q_p(r+r^{\prime}))\lomapr{{\rm Tr}_X}\B_{S^{\flat}}\{s,s\}[-2d].
 $$
 This descends to a pairing  in ${\rm QCoh}(X_{\FF,S^{\flat}})$:
$$
\se_{\proeet}(X_S,\Q_p(r))\otimes^{\LL}_{\so}\se_{\proeet,c}(X_S,\Q_p(r^{\prime}))\lomapr{\cup} \se_{\proeet,c}(X_S,\Q_p(r+r^{\prime}))\lomapr{{\rm Tr}_X}\so(s,s)[-2d],
$$
which induces a natural map in ${\rm QCoh}(X_{\FF,S^{\flat}})$
\begin{equation}\label{niedziela10-etale}
\gamma_{X_S}:\quad  \se_{\proeet}(X_S,\Q_p(r))\to  {\mathbb D}(\se_{\proeet,c}(X_S,\Q_p(r^{\prime}))[2d],\so(s,s)).
\end{equation}
By an abuse of notation, we will write
\begin{equation}\label{niedziela10-etale1}
\gamma_{X_S}:\quad  \se_{\proeet}(X_S,\Q_p)\to  {\mathbb D}(\se_{\proeet,c}(X_S,\Q_p(d))[2d],\so)
\end{equation}
for the Tate-untwisted version of the map \eqref{niedziela10-etale}.
\begin{corollary}\label{curve-duality-etale} 
{\rm (Pro-\'etale  Poincar\'e duality on the Fargues-Fontaine curve)}

The map $\gamma_{X_S}$ from \eqref{niedziela10-etale1}  is a quasi-isomorphism in ${\rm QCoh}(X_{\FF,S^{\flat}})$. 
\end{corollary}
\begin{proof}Choose  $r,r^{\prime}\geq 2d$ and set $s:=r+r^{\prime}-d$. It suffices to prove that the Tate twisted map \eqref{niedziela10-etale} is a quasi-isomorphism. 
This follows immediately from the syntomic duality from Theorem \ref{curve-duality} and the comparison result from  Proposition  \ref{china1}.
\end{proof}

 \subsection{K\"unneth formula} We will show that 
arguing as above one  can prove the K\"unneth formula. This section is not used in the proof of the duality results. 

 Let $X,Y$ be smooth Stein rigid analytic varieties over $K$. The simple observation that we have a quasi-isomorphism in $\sd(K_{\Box})$
 \begin{equation}\label{tu1}
\big(\Omega(X)\otimes_{K}^{\LL_{\Box}}\so(Y)\big)\oplus 
\big(\so(X)\otimes^{\LL_{\Box}}_{K}\Omega(Y)\big)\stackrel{\sim}{\to}  \Omega(X\times_KY),
 \end{equation}
which implies  the  K\"unneth formula for de Rham cohomology
 $$
 \R\Gamma_{\dr}(X)\otimes^{\LL_{\Box}}_K \R\Gamma_{\dr}(Y) \stackrel{\sim}{\to}\R\Gamma_{\dr}(X\times_K Y)
 $$
 leads to the syntomic K\"unneth formula in ${\rm QCoh}(X_{\FF})$ and hence the pro-\'etale one as well:
 
 \begin{theorem}{\rm(K\"unneth formula)} Let $X,Y$ be smooth partially proper rigid analytic varieties over $K$. Let $d$ be larger than the dimension of $X\times_KY$ and let $r,r^{\prime}\geq 2d$.  Let $S\in {\rm Perf}_C$.
 The natural \index{KAPPA@\KAPPA}maps 
 \begin{align*}
\kappa:\quad   \se_{\synt}(X_S,\Q_p(r))\otimes^{\LL}_{\so} \se_{\synt}(Y_S,\Q_p(r^{\prime}))  & \to \se_{\synt}((X\times_K Y)_S,\Q_p(r+r^{\prime})),\\
\kappa:\quad   \se_{\proeet}(X_S,\Q_p)\otimes^{\LL}_{\so} \se_{\proeet}(Y_S,\Q_p)  & \to \se_{\proeet}((X\times_K Y)_S,\Q_p)
 \end{align*}
 are quasi-isomorphisms in ${\rm QCoh}(X_{\FF,S^{\flat}})$. 
  \end{theorem}
 \begin{proof} The pro-\'etale case follows from the syntomic one via the comparison quasi-isomorphism from Proposition \ref{china1}. Its naturality and compatibility with products, recalled after \eqref{cup11et}, identify the two K\"unneth morphisms.
 
  For the syntomic case,  it is enough to show that on the level of $\phi$-modules over  $\B^{\FF}_{S^{\flat}}$ the corresponding map
 $$
 \kappa:  \quad\R\Gamma_{\synt}^{\B}(X_S,\Q_p(r))\otimes^{\LL}_{\B^{\FF}_{S^{\flat}}} \R\Gamma_{\synt}^{\B}(Y_S,\Q_p(r^{\prime})) \to \R\Gamma_{\synt}^{\B}((X\times_K Y)_S,\Q_p(r+r^{\prime}))
 $$
 is a quasi-isomorphism.  Or that in $\sd(\B^{\prime}_{\Box})$, for  $\B^{\prime}:=\B_{S^{\flat}}^{[u,v]}$, the induced \index{KAPPA@\KAPPA}map
 $$
  \kappa^{[u,v]}: \quad  \R\Gamma_{\synt}^{{[u,v]}}(X_S,\Q_p(r))\otimes^{\LL_{\Box}}_{\B^{\prime}} \R\Gamma_{\synt}^{{[u,v]}}(Y_S,\Q_p(r^{\prime})) \to \R\Gamma_{\synt}^{{[u,v]}}((X\times_K Y)_S,\Q_p(r+r^{\prime}))
  $$
 is a quasi-isomorphism.   But for that, as in the proof of Theorem \ref{curve-duality},  it is enough to check that the base changes of $  \kappa^{[u,v]}$ to $\B^{\prime}[1/t]$ and to $\B^{\prime}/t$ are quasi-isomorphisms. 
 
  For the first base change, we use the quasi-isomorphism in $\sd(\B^{\prime}_{\Box})$
$$
 \R\Gamma^{{[u,v]}}_{\synt}(X_S,\Q_p(r))[1/t]  \stackrel{\sim}{\to} \R\Gamma^{{[u,v]}}_{\hk}(X_S,r)[1/t]
 $$
 to write
 $$
  \kappa^{[u,v]}[1/t]: \quad (\R\Gamma^{{[u,v]}}_{\hk}(X_S,r)\otimes^{\LL_{\Box}}_{\B^{\prime}}\R\Gamma^{{[u,v]}}_{\hk}(Y_S,r^{\prime}))[1/t]\to \R\Gamma^{{[u,v]}}_{\hk}((X\times_KY)_S,r+r^{\prime})[1/t].
  $$
 This map is induced by the Hyodo-Kato pairing
$$
  \kappa^{[u,v]}_{\hk}: \quad \R\Gamma^{{[u,v]}}_{\hk}(X_S,r)\otimes^{\LL_{\Box}}_{\B^{\prime}}\R\Gamma^{{[u,v]}}_{\hk}(Y_S,r^{\prime})\to \R\Gamma^{{[u,v]}}_{\hk}((X\times_KY)_S,r+r^{\prime}).
  $$
  To check that this is a quasi-isomorphism we may pass to cohomology. Since $H^j_{\hk}(X_C)$ is Fr\'echet (hence flat for the solid tensor product over $\breve{C}$), this reduces
  to checking that the pairing
  $$
  \bigoplus_{a=0}^j (H^a_{\hk}(X_C)\otimes^{\LL_{\Box}}_{\breve{C}}\B^{\prime}_{\log})^{N=0}\otimes^{\LL_{\Box}}_{\B^{\prime}} (H^{j-a}_{\hk}(Y_C)\otimes^{\LL_{\Box}}_{\breve{C}}\B^{\prime}_{\log})^{N=0}
  \to  (H^j_{\hk}((X\times_KY)_C)\otimes^{\LL_{\Box}}_{\breve{C}}\B^{\prime}_{\log})^{N=0}
  $$
  is an isomorphism in $\B^{\prime}_{\Box}$. 
  
    Now, using the  identification \eqref{evening1},
we reduce to proving that the pairing 
  $$
   \bigoplus_{a=0}^j (H^a_{\hk}(X_C)\otimes^{\LL_{\Box}}_{\breve{C}}\B^{\prime})
\otimes^{\LL_{\Box}}_{\B^{\prime}} (H^{j-a}_{\hk}(Y_C)\otimes^{\LL_{\Box}}_{\breve{C}}\B^{\prime})
  \to  H^j_{\hk}(X_C\times_CY_C)\otimes^{\LL_{\Box}}_{\breve{C}}\B^{\prime}
  $$
  is an isomorphism in $\B^{\prime}_{\Box}$. Or,  that so is the pairing in $\breve{C}_{\Box}$
  $$
     \bigoplus_{a=0}^j H^a_{\hk}(X_C)\otimes^{\LL_{\Box}}_{\breve{C}}H^{j-a}_{\hk}(Y_C)
  \to  H^j_{\hk}(X_C\times_CY_C).
  $$
But this follows from the following:
 \begin{lemma}{\rm(Hyodo-Kato K\"unneth formula)}\label{Singapur1}  Let $U,V$ be  smooth partially proper rigid analytic varieties over $C$. Then the canonical \index{KAPPA@\KAPPA}pairing
$$
  \kappa_{\hk}:\quad \R\Gamma_{\hk}(U)\otimes^{\LL_{\Box}}_{\breve{C}}\R\Gamma_{\hk}(V)\to \R\Gamma_{\hk}(U\times_CV)
$$
is a quasi-isomorphism in $\sd_{\phi, N}(\breve{C}_{\Box})$. If $U$ and $V$ are defined over $K$, the pairing
is also $\sg_K$-equivariant.
\end{lemma}
\begin{proof}
This follows from the  comparison (via the Hyodo-Kato morphism) with the K\"unneth formula for de Rham \index{KAPPA@\KAPPA}cohomology
$$
  \kappa_{\dr}: \R\Gamma_{\dr}(U)\otimes^{\LL_{\Box}}_{{C}}\R\Gamma_{\dr}(V)\stackrel{\sim}{\to }\R\Gamma_{\dr}(U\times_CV).
$$
The latter clearly holds if both $U$ and $V$ are Stein. For general partially proper $U$ and $V$, we use coverings by a countable number (!) of Stein varieties, the fact that all the complexes in sight are bounded complexes of Fr\'echet spaces, \cite[Prop.\,8.33]{GB1}, and the Stein case.
\end{proof}

  For the base change to $R=\B^{\prime}/t$, we get from the proof of Theorem \ref{curve-duality} quasi-isomorphisms compatible with products in 
  $\sd(\B^{\prime}_{\Box})$ ($T=X,Y$, $s\geq d_T$)
\begin{align*}
 \R\Gamma^{{[u,v]}}_{\synt}(T_S,  \Q_p(s))\otimes^{\LL_{\Box}}_{\B^{\prime}} R& \simeq F^s\R\Gamma_{\dr}(T_S/\B^{\prime})\otimes^{\LL_{\Box}}_{\B^{\prime}} R \\
 & \simeq  \bigoplus_{i=0}^{d_T}\R\Gamma(T, \Omega^i) \otimes^{\LL_{\Box}}_KR(s-i)[-i].
 \end{align*}
 And the map $ \kappa^{[u,v]}$ can be identified with the map
 \begin{align*}
 \Big(\bigoplus_{i=0}^{d_X}\R\Gamma(X,\Omega^i) \otimes^{\LL_{\Box}}_KR(r-i)[-i] \Big)
 & \otimes^{\LL_{\Box}}_{R}
\Big(\bigoplus_{i=0}^{d_Y}\R\Gamma(Y,\Omega^i) \otimes^{\LL_{\Box}}_KR(r^{\prime}-i)[-i] \Big)\\
  & \to \bigoplus_{i=0}^{d_{X}+d_{Y}}\R\Gamma(X\times_K Y, \Omega^i) \otimes^{\LL_{\Box}}_KR(r+r^{\prime}-i)[-i].
 \end{align*}
 If $X, Y$ are Stein, this map 
  in degree $i$ is represented by  the map
  \begin{align*}
 \bigoplus_{a=0}^{d_X+d_Y}\Big(\Omega^a(X) \otimes^{\LL_{\Box}}_KR(r-a)\Big)
\otimes^{\LL_{\Box}}_{R}
\Big(\Omega^{i-a}(Y) \otimes^{\LL_{\Box}}_KR(r^{\prime}-i+a)\Big)
 \to\Omega^i(X\times_K Y) \otimes^{\LL_{\Box}}_KR(r+r^{\prime}-i).
 \end{align*}
 And the latter  map is   a quasi-isomorphism in  $R_{\Box}$ by \eqref{tu1}. If $X, Y$ are general smooth partially proper rigid analytic varieties, we can reduce to the Stein case as in the proof of Lemma \ref{Singapur1}.
 \end{proof}
\section{Poincar\'e duality for $p$-adic  geometric pro-\'etale cohomology}
Finally, we are  ready to state and prove pro-\'etale duality on the level of Topological Vector Spaces
(Theorem~\ref{VS-duality}). We descend it from the analogous Poincar\'e duality on the Fargues-Fontaine curve (Corollary \ref{curve-duality-etale}). 
\subsection{Topological Vector Spaces}
 In this paper, the category of Topological Vector Spaces 
\index{TVS@\TVS}({\rm TVS}'s for short) is the $\infty$-category of $\underline{\Q}_p$-modules in the $\infty$-derived category 
 $\underline{\sd}({\rm Spa}(C),{\rm Solid})$ of topologically enriched presheaves 
\index{PERF@\PERF}on ${\rm sPerf}_C$ -- the category of strictly totally disconnected affinoid perfectoids over $C$ -- with values in solid abelian groups.  We will denote it by  $\underline{\sd}({\rm Spa}(C),\Q_{p,\Box})$. This category was defined and studied in \cite{TVS}. We will denote by 
 $ {\sd}({\rm Spa}(C),\Q_{p,\Box})$ the corresponding $\infty$-category where we forget the enrichment; the objects of this category will be called ``topological presheaves''. 

\subsubsection{Construction}
We briefly recall some elements of the construction from \cite{TVS}.
\begin{enumerate}[leftmargin=*]
\item 
The enrichment records ``topological''  functoriality in 
families of morphisms. It is expressed by functorial
structure maps
$$
\sff(S\times\wh{T})\to
\R\uHom_{\Q_{p,\Box}}(\Q_{p,\Box}[T],\sff(S)),\quad S\in \mathrm{sPerf}_{C}, T\in \mathrm{ProFin},
$$
with the usual unit and composition compatibilities (see \cite[Sec.\,3.2.8]{TVS}). Here we set  
$\wh{T}:=\Spa(\scc(T,C),\scc(T,\so_C)).$ 

  \item  The derived projection
$$
\R\pi_*:\sd(\Spa(C)_{\proeet},\Q_p)
\longrightarrow\underline{\sd}(\Spa(C),\Q_p)
$$
sends $\sff$ to the enriched presheaf
$$
S\longmapsto
\bigl\{T\longmapsto
\R\Gamma_{\proeet}(S\times\wh{T},\sff)\bigr\},
$$
where the term in braces denotes the derived condensed $\Q_p$-module
$\R p_{S,*}(\sff|_{S_{\proeet}})$.
Here $p_S:S_{\proeet}\to *_{\proeet}$ is the projection to a point,
induced by $T\mapsto S\times\wh{T}$;  we have
\[
\R\Gamma\bigl(T,\R p_{S,*}(\sff|_{S_{\proeet}})\bigr)
\simeq \R\Gamma_{\proeet}(S\times\wh{T},\sff).
\]
The enrichment is induced by
pullback (see \cite[Sec.\,4.1.1]{TVS}).

 \item The realization of a solid quasi-coherent complex
$\sff$ on the Fargues--Fontaine curve is defined by
$$
(\R\tau_*\sff)(S)
=\R\Gamma(X_{\FF,S^\flat},\LL f_S^*\sff),
$$
where $f_S:X_{\FF,S^\flat}\to X_{\FF,C^\flat}$
is induced by $S\to\Spa(C)$.
Pullback and evaluation give the enrichment on this
presheaf (see \cite[Sec.\,4.2.2]{TVS}).
 
 \item  The functors $\R\pi_*$ and $\R\tau_*$ are canonically
lax symmetric monoidal, with structure maps given
by cup products. Moreover, $\R\tau_*$ commutes with
colimits and with countable limits of nuclear objects
(see \cite[Sec.\,4.1, Secs.\,4.2.3--4.2.4]{TVS}).
\end{enumerate}
   \subsubsection{Properties} We list  the following properties: 
\begin{proposition}{\rm (\cite[Th. 1.1]{TVS})} \label{recall1}
\begin{enumerate}
\item {\rm(Enriched full faithfulness)}  The canonical functor from Vector Spaces\footnote{We call Vector Spaces  ({\rm VS}'s for short) the objects in the $\infty$-derived category 
$\sd(\Spa(C)_{\proeet},\Q_p)$ of $\underline{\Q}_p$-modules in the category of pro-\'etale sheaves on ${\rm Perf}_C$, the category of perfectoid affinoids over $C$.} to Topological Vector \index{rtau@\rtau}Spaces
$$
\R\pi_*: \sd({\rm Spa}(C)_{\proeet},\Q_p)\to \underline{\sd}({\rm Spa}(C),\Q_p)
$$
is fully faithful. More precisely, let
$\sff,\sg\in\sd({\rm Spa}(C)_{\proeet},\Q_p)$.
Then the canonical morphism in
$\underline{\sd}({\rm Spa}(C),\Q_p)$
$$
\R\pi_*\R\Hhom_C(\sff,\sg)
\longrightarrow
\R\Hhom_{C^{\rm top}}(\R\pi_*\sff,\R\pi_*\sg)
$$
is a quasi-isomorphism.
 Here $\R\Hhom_{C^{\rm top}}$ refers to the
internal Hom in the enriched presheaf category.

\item  {\rm(Fargues-Fontaine full faithfulness)}  The \index{rtau@\rtau}functor 
$$
\R\tau_*: {\rm QCoh}(X_{\FF, C^{\flat}})\to \underline{\sd}({\rm Spa}(C),\Q_{p})
$$
is fully faithful when restricted to perfect complexes. That is, for $\sff,\sg\in {\rm Perf}(X_{\FF, C^{\flat}})$, the natural map in \(\underline{\mathcal D}(\mathrm{Spa}(C),\mathbf Q_p)\)
$$\R\tau_*\R\Hhom_{{\rm QCoh}(X_{\FF,C^{\flat}})}(\sff,\sg)\to \R\Hhom_{C^{\rm top}}(\R\tau_*\sff,\R\tau_*\sg)
$$
is a quasi-isomorphism. 
\item  {\rm (Compatibility of the algebraic and topological projections)} The \index{rtau@\rtau}functor 
$$
\R\tau^{\prime}_*: {\rm Nuc}(X_{\FF, C^{\flat}})\to {\sd}({\rm Spa}(C)_{\proeet},\Q_{p})
$$
is compatible with the functor $\R\tau_*$ when the latter is restricted to nuclear sheaves. That is, the
 following diagram commutes
$$
\xymatrix{
 {\rm Nuc}(X_{\FF,C^{\flat}})\ar[r]^-{\R\tau^{\prime}_*}\ar[rd]^{\R\tau_*} &  \sd({\rm Spa}(C)_{\proeet},{\Q}_p)\ar[d]^{\R\pi_*}\\
  & \underline{\sd}({\rm Spa}(C),{\Q}_p).
  }
  $$
\end{enumerate}
\end{proposition}
 \subsection{{\rm TVS}-version of  pro-\'etale cohomology presheaves} Let $X$ be a smooth partially proper  rigid analytic variety over $K$.  
For $?=-,c$, we define the presheaves on ${\rm sPerf}_C$ with values in $\sd({\rm Solid})$
$$
 {\mathbb R}_{\proeet,?}(X_C,\Q_p):\quad S\to \R\Gamma_{\proeet,?}(X_{S}, \Q_p).
 $$
The topology on pro-\'etale cohomology is induced via the \v{C}ech procedure from $p$-adic topologies. Note that the values of these presheaves on $S$ are actually in $\sd(\Q_p(S)_{\Box})$.
We also denote by \index{ETB@\ETB}${\mathbb R}^{\rm alg}_{\proeet,?}(X_C,\Q_p)$ the algebraic version, defined by the same formula but with discrete topology objectwise. 
\begin{proposition} \label{comp-elgin} 
\begin{enumerate}
\item The presheaf ${\mathbb R}_{\proeet,?}(X_C,\Q_p)$ is naturally  enriched: 
$${\mathbb R}_{\proeet,?}(X_C,\Q_p)\in \underline{\sd}({\rm Spa}(C),\Q_{p,\Box}).
$$
\item There exists a natural quasi-isomorphism in $\underline{\sd}({\rm Spa}(C),\Q_{p,\Box})$
$$
{\mathbb R}_{\proeet,?}(X_C,\Q_p)\simeq \R\pi_*{\mathbb R}^{\rm alg}_{\proeet,?}(X_C,\Q_p).
$$
\item   Let $X$ be a smooth Stein variety over $K$. There exists a natural quasi-isomorphism in $\underline{\sd}({\rm Spa}(C),\Q_{p,\Box})$
\begin{equation}
\R\tau_*\se_{\proeet,?}(X_C,\Q_p) \simeq {\mathbb R}_{\proeet,?}(X_C,\Q_p),\quad ?=-,c.
\end{equation}
\end{enumerate}
\end{proposition}
\begin{proof}   For claim (1), consider first the usual cohomology.  Using the description of enriched presheaves in
\cite[Sec.\,3.2.8, (3.47)]{TVS}, we construct the structure maps
$$
\R\Gamma_{\proeet}(X_{S\times T},\Q_p)\to \R\uHom_{\Q_{p,\Box}}(\Q_{p,\Box}[T],\R\Gamma_{\proeet}(X_{S},\Q_p)),
$$
for $S\in {\rm sPerf}_C$ and a profinite set $T$.  By pro-\'etale descent, it suffices to construct, for a set $\{S_i\}, i\in I$, $S_i\in {\rm sPerf}_C$, functorial structure  maps
$$
\prod_I\underline{\Q}_p(S_i\times T)\to \uHom_{\Q_{p,\Box}}(\Q_{p,\Box}[T], \prod_I\underline{\Q}_p(S_i))
$$
or functorial maps
$$
\underline{\Q}_p(S_i\times T)\to \uHom_{\Q_{p,\Box}}(\Q_{p,\Box}[T], \underline{\Q}_p(S_i))
$$
But these maps can be identified with the canonical isomorphisms  (see \cite[Ex. 2.1]{TVS})
$$
\underline{\scc(|S_i\times T|,\Q_p)}\stackrel{\sim}{\to} \underline{\scc( T,\underline{\scc(|S_i|,\Q_p)})}
$$
  
  To treat the compactly supported version ${\mathbb R}_{\proeet,c}(X_C,\Q_p)$, recall \index{BORD@\BORD}that
\begin{align}\label{kwak-kwak1}
\R\Gamma_{\proeet,c}(X_S,\Q_p) & =[\R\Gamma_{\proeet}(X_S,\Q_p)\to \R\Gamma_{\proeet}((\partial X)_S,\Q_p)],\\
\R\Gamma_{\proeet}((\partial X)_S,\Q_p) &=\colim_{Z\in\Phi_X}\R\Gamma_{\proeet}((X\moins Z)_S,\Q_p),\notag
\end{align}
where $\Phi_X$ is the set of quasi-compact opens in $X$. 
This canonically induces  the  enrichment on the presheaf ${\mathbb R}_{\proeet,c}(X_C,\Q_p)$. 

   For claims (2) and (3), we first establish the algebraic
identification
$$
\R\tau^{\prime}_*\se_{\proeet,?}(X_C,\Q_p)
\simeq {\mathbb R}^{\rm alg}_{\proeet,?}(X_C,\Q_p),
\quad ?=-,c.
$$
By the definition of $\R\tau^{\prime}_*$ and
Proposition \ref{comp2b}, respectively, we have as presheaves
\begin{align*}
\R\tau^{\prime}_*\se_{\proeet,?}(X_C,\Q_p)
&=\{S\mapsto
\R\Gamma(X_{\FF,S^{\flat}},
\LL\hskip-.4mm f^*_S\se_{\proeet,?}(X_C,\Q_p))\},\\
{\mathbb R}^{\rm alg}_{\proeet,?}(X_C,\Q_p)
&=\{S\mapsto
\R\Gamma(X_{\FF,S^{\flat}},
\se_{\proeet,?}(X_S,\Q_p))\}.
\end{align*}
Lemma \ref{dziecko1} below identifies these two presheaves.
In particular, ${\mathbb R}^{\rm alg}_{\proeet,?}(X_C,\Q_p)$
is a sheaf, since $\se_{\proeet,?}(X_C,\Q_p)$ is nuclear
and $\R\tau^{\prime}_*$ takes nuclear objects to sheaves (see \cite[Sec.\,4.2.3]{TVS}).

  Now, claim (2) is clear for the usual cohomology from the
argument proving claim (1). For compactly supported cohomology,
the same presheaf calculation applies to the algebraic
version of \eqref{kwak-kwak1}, using the sheaf property
just established.

Finally, claim (3) follows from the algebraic identification
above, Proposition \ref{recall1}, and claim (2).
\end{proof}
\begin{lemma}\label{dziecko1} 
On $X_{\FF,S^{\flat}}$ we have a natural quasi-isomorphism of solid quasi-coherent sheaves
\begin{align*}
\LL\hskip-.4mm f^*_S\se_{\proeet,?}(X_C,\Q_p) \simeq \se_{\proeet,?}(X_S,\Q_p).
\end{align*}
\end{lemma}
\begin{proof}
It suffices to show that, for a compact rational interval $[u,v]\subset (0,\infty)$, 
the canonical map
$$\R\Gamma_{\proeet,?}(X_C,{\mathbb B}^{[u,v]})\otimes^{\LL_{\Box}}_{\B^{[u,v]}_{C^{\flat}}}
\B_{S^{\flat}}^{ [u,v]}{\to}\R\Gamma_{\proeet,?}(X_S,{\mathbb B}^{[u,v]})$$
is a quasi-isomorphism. Since tensor product commutes with colimits, it suffices to do this for  the usual cohomology. 

  By  Theorem \ref{hot1}, it suffices to show the same for the twisted Hyodo-Kato and de Rham cohomologies. Assume thus that $r\geq 2d$
 and let us start  with the Hyodo-Kato cohomology. We want to show that  the  base change map 
(for $\B^{\prime}_S:=\B_{S^{\flat}}^{[u,v]}$)
 $$
 \R\Gamma_{\hk}^{{[u,v]}}(X_C,r)\otimes^{\LL_{\Box}}_{\B^{\prime}_C} \B^{\prime}_{S} \to \R\Gamma_{\hk}^{{[u,v]}}(X_S,r)
  $$
 is a quasi-isomorphism. 
For that,  we may pass to cohomology. Since $H^j_{\hk}(X_C)$ is Fr\'echet (hence flat for the solid tensor product over $\breve{C}$),  this reduces
  to checking that the base change, for $b\geq 0$, 
$$
  (H^b_{\hk}(X_C)\otimes^{\LL_{\Box}}_{\breve{C}}\B^{\prime}_{C,\log})^{N=0}\otimes^{\LL_{\Box}}_{\B^{\prime}_C} \B^{\prime}_S
  \to  (H^b_{\hk}(X_C)\otimes^{\LL_{\Box}}_{\breve{C}}\B^{\prime}_{S,\log})^{N=0}
  $$
  is a quasi-isomorphism in $\sd(\B^{\prime}_{S,\Box})$. 
  
    Now,  using the identification \eqref{evening1}, which is compatible
with change of coefficients, we reduce to proving that the base change
  $$
 (H^b_{\hk}(X_C)\otimes^{\LL_{\Box}}_{\breve{C}}\B^{\prime}_C)\otimes^{\LL_{\Box}}_{\B^{\prime}_C} \B^{\prime}_S
  \to  H^b_{\hk}(X_C)\otimes^{\LL_{\Box}}_{\breve{C}}\B^{\prime}_S
  $$
  is a quasi-isomorphism in $\sd(\B^{\prime}_{S,\Box})$. But this is clear.

   We pass now to  the de Rham cohomology.  As above it suffices to show that the base change map
   $$
 \R\Gamma_{\dr}^{{[u,v]}}(X_C,r)\otimes^{\LL_{\Box}}_{\B^{\prime}_C} \B^{\prime}_{S} \to \R\Gamma_{\dr}^{{[u,v]}}(X_S,r)
  $$
 is a quasi-isomorphism. But this reduces to showing that the base change maps
 \begin{align*}
& (\Omega^i(X)\otimes^{\Box}_K(\B^+_{\dr}/t^s))
\otimes^{\LL_{\Box}}_{\B^{\prime}_C}\B^{\prime}_S\to  \Omega^i(X)\otimes^{\Box}_K(\B^+_{\dr}(S)/t^s)
 \end{align*}
 are isomorphisms. And this is clear.    
\end{proof}
\subsection{Topological Poincar\'e duality} 
Let $X$ be a partially proper smooth  variety over $K$ of dimension $d$.
Let $i, j\geq 0$.  We define  a pairing in $\underline{\sd}({\rm Spa}(C),\Q_{p,\Box})$: 
\begin{align}\label{kicius1}
 & {\mathbb R}_{\proeet}(X_C,\Q_p(i))\otimes^{\LL_{\Box}}_{\Q_p} {\mathbb R}_{\proeet,c}(X_C,\Q_p(j))\to \Q_p(i+j-d)[-2d]
\end{align}
by inducing it from   the compatible family of pairings in $\sd(\Q_p(S)_{\Box})$
\begin{align*}
 & \R\Gamma_{\proeet}(X_{S},\Q_p(i))\otimes^{\LL_{\Box}}_{\Q_p(S)} \R\Gamma_{\proeet,c}(X_{S},\Q_p(j))  \stackrel{\cup}{\to }\R\Gamma_{\proeet,c}(X_{S},\Q_p(i+j))\stackrel{{\rm Tr}_X}{\longrightarrow} {\Q}_p(S)(i+j-d)[-2d],
 \end{align*}
 where the trace map comes  from the trace maps \eqref{trace11et} via the fundamental exact sequence.
The discrete versions of the same cup product and trace define the algebraic pairing in $\sd(\Spa(C)_{\proeet},\Q_p)$
$$
{\mathbb R}^{\rm alg}_{\proeet}(X_C,\Q_p(i))\otimes^{\LL}_{\Q_p}{\mathbb R}^{\rm alg}_{\proeet,c}(X_C,\Q_p(j))\to \Q_p(i+j-d)[-2d].
$$
 The compatibility of the pairing \eqref{kicius1} with the enrichments follows from Proposition \ref{comp-elgin}, since it is induced by the algebraic pairing and the projection functor $\R\pi_*$ is lax monoidal.

 The pairing in \eqref{kicius1}  induces a duality map in 
 $\underline{\sd}({\rm Spa}(C),\Q_{p,\Box})$
\begin{equation}\label{bonn10}
\gamma_{X_C}:\quad {\mathbb R}_{\proeet}(X_C,\Q_p)\to \R\Hhom_{\rm TVS}({\mathbb R}_{\proeet,c}(X_C,\Q_p(d))[2d],{\Q}_p).
\end{equation}

\begin{theorem}{\rm (Pro-\'etale  duality)}\label{VS-duality}
 Let $X$ be a smooth partially proper  rigid analytic variety over $K$ of dimension $d$.  The duality map
\eqref{bonn10} is a quasi-isomorphism. In particular,
we have a quasi-isomorphism in $\sd(\Q_{p,\Box})$
$$\gamma_{X_C}:\quad{\R}\Gamma_{\proeet}(X_C,\Q_p)\stackrel{\sim}{\to}\R\uHom_{\rm TVS}({\mathbb R}_{\proeet,c}(X_C,\Q_p(d))[2d],\Q_p).
$$
\end{theorem}

\begin{proof} 
($\bullet$)  Assume first that $X$ is Stein and the de Rham cohomology of $X$ has finite rank. It suffices to show that, for $r,r^{\prime}\geq 2d, s=r+r^{\prime}-d$,  the  pairing in \eqref{kicius1}  induces a duality map in 
 $\underline{\sd}({\rm Spa}(C),\Q_{p,\Box})$
$$
\gamma_{X_C}:\quad {\mathbb R}_{\proeet}(X_C,\Q_p(r))\to \R\Hhom_{\rm TVS}({\mathbb R}_{\proeet,c}(X_C,\Q_p(r^{\prime}))[2d],{\Q}_p(s)),
$$
which is a quasi-isomorphism. 
 But we have  the quasi-isomorphism  
$${\mathbb R}_{\proeet,?}(X_C,\Q_p(r))\simeq \R\tau_{*}\se_{\proeet,?}(X_C,\Q_p(r)) $$ 
from  Proposition \ref{comp-elgin} and the quasi-isomorphism $\R\tau_*\so(i,i)\simeq \Q_p(i)$, $i\geq 0$.
Moreover, 
the functor $\R\tau_*$ is lax symmetric monoidal 
(see \cite[Sec.\,4.2.4]{TVS}).
Under the preceding identifications, applying
$\R\tau_*$ to the pairing on the Fargues--Fontaine
curve gives \eqref{kicius1}: both pairings are induced
by \eqref{cup11et} and \eqref{trace11et} via the
fundamental exact sequence.
Hence it suffices to show that the duality map
$$
\gamma_{X_C}:\quad \R\tau_*\se_{\proeet}(X_C,\Q_p(r))\to \R\Hhom_{\rm TVS}(\R\tau_*\se_{\proeet,c}(X_C,\Q_p(r^{\prime}))[2d],\R\tau_*\so(s,s))
$$
is a quasi-isomorphism. 

   This map factorizes as 
$$
\xymatrix@C=-2.5cm{
 \R\tau_*\se_{\proeet}(X_C,\Q_p(r))\ar[rr] \ar[rd]^-{\R\tau_*\gamma_{X_C}}_{\sim}&  &\R\Hhom_{\rm TVS}(\R\tau_*\se_{\proeet,c}(X_C,\Q_p(r^{\prime}))[2d],\R\tau_*\so(s,s))\\
 & \R\tau_*{\mathbb D}(\se_{\proeet,c}(X_C,\Q_p(r^{\prime}))[2d],\so(s,s))\ar[ur]_-{\rm can}
}
$$
where we \index{D@\bD}set 
$${\mathbb D}(-,-):=\R\Hhom_{\rm QCoh(X_{\FF,C^{\flat}})}(-, -)$$
 The left slanted map is a quasi-isomorphism by the Poincar\'e duality on the Fargues-Fontaine curve from
 Corollary \ref{curve-duality-etale}. 
Hence it suffices to show that the canonical morphism
$$
 \R\tau_{*}{\mathbb D}(\se_{\proeet,c}(X_C,\Q_p(r^{\prime})),\so(s,s))\to \R\Hhom_{\rm TVS}(\R\tau_*\se_{\proeet,c}(X_C,\Q_p(r^{\prime})),\R\tau_*\so(s,s))
$$
is a quasi-isomorphism. Or, by Proposition \ref{china1}, that  so is the canonical morphism
$$
 \R\tau_{*}{\mathbb D}(\se_{\synt,c}(X_C,\Q_p(r^{\prime})),\so(s,s))\to \R\Hhom_{\rm TVS}(\R\tau_*\se_{\synt,c}(X_C,\Q_p(r^{\prime})),\R\tau_*\so(s,s)).
$$

     Applying $ \R\tau_{*}{\mathbb D}(-,\so(s,s))$ and  $ \R\Hhom_{\rm TVS}(\R\tau_*(-),\R\tau_*\so(s,s))$  to the distinguished triangle 
    \begin{equation} 
\se_{\synt,c}(X_C,\Q_p(r^{\prime}))\to\se_{\hk,c}(X_C,r^{\prime})\to \se_{\dr,c}(X_C,r^{\prime})
\end{equation}
and identifying $\R\tau_*\so(s,s)\simeq \Q_p(s)$, 
   we get compatible  distinguished triangles
\begin{align*}
& \R\tau_{*}{\mathbb D}(\se_{\synt,c}(X_C,\Q_p(r^{\prime})),\so(s,s))\leftarrow  \R\tau_{*}{\mathbb D}(\se_{\hk,c}(X_C,r^{\prime}),\so(s,s))\leftarrow 
  \R\tau_{*} {\mathbb D}(\se_{\dr,c}(X_C,r^{\prime}),\so(s,s)),\\
&\R\Hhom_{\rm TVS}(\R\tau_*\se_{\synt,c}(X_C,\Q_p(r^{\prime})),\Q_p(s))\leftarrow  \R\Hhom_{\rm TVS}(\R\tau_{*}\se_{\hk,c}(X_C,r^{\prime}),\Q_p(s))\\
&\qquad\leftarrow \R\Hhom_{\rm TVS}(\R\tau_{*}\se_{\dr,c}(X_C,r^{\prime}),\Q_p(s)).
\end{align*}
It suffices thus to show that the canonical morphisms
\begin{align*}
   \R\tau_{*} {\mathbb D}(\se_{\hk,c}(X_C,r^{\prime}),\so(s,s))& \to  \R\Hhom_{\rm TVS}(\R\tau_{*}\se_{\hk,c}(X_C,r^{\prime}),\Q_p(s)),\\
  \R\tau_{*}{\mathbb D}(\se_{\dr,c}(X_C,r^{\prime}),\so(s,s))& \to  \R\Hhom_{\rm TVS}(\R\tau_{*}\se_{\dr,c}(X_C,r^{\prime}),\Q_p(s))
\end{align*}
are quasi-isomorphisms. 

  Since the solid quasi-coherent complexes $\se_{\hk,c}(X_C,r^{\prime})$ and $\so(s,s)$ are perfect, the first quasi-isomorphism follows from Proposition \ref{recall1}. 
For the second quasi-isomorphism, since $\se_{\dr,c}(X_C,r^{\prime})=i_{\infty,*}\R\Gamma_{\dr,c}(X_C/\B_{\dr}^+)/F^{r^{\prime}}$, it suffices to show that the natural morphism
  \begin{equation}\label{kolo241}
    \R\tau_{*} {\mathbb D}(\iota_{\infty,*}(W\otimes^{\LL_{\Box}}_{K}\B^+_{\dr}/t^i),\so(s,s))\to\R\Hhom_{\rm TVS}(\R\tau_{*}(\iota_{\infty,*}(W\otimes^{\LL_{\Box}}_{K}\B^+_{\dr}/t^i)),\R\tau_{*}\so(s,s))
  \end{equation}
  is a quasi-isomorphism, for  any  space of compact type $W\in C_K$. (Recall that  $C_K$ denotes the category of locally convex topological
vector spaces over $K$). 
By d\'evissage we may assume that $i=1$.  
Also, if we write  $W\simeq \colim_n W_n$ as a compact colimit of Smith spaces $W_n$ over $K$ we may assume that $W$ is a Smith space over $K$.  This is because 
$W^*$ is then a compact limit of Banach spaces and such limits commute with solid tensors with Banach spaces yielding that both sides of \eqref{kolo241} will change the colimit into a derived limit.  
Moreover, assuming that $W$ is a Smith space over $K$, we can write $W=W_{\Q_p}{\otimes^{\LL_{\Box}}_{\Q_{p}}}K$, 
for a Smith space $W_{\Q_p}$ over $\Q_p$. 
Then we have $W^*=W^*_{\Q_p}{\otimes^{\LL_{\Box}}_{\Q_{p}}}K$. 
 Hence we may assume that $K=\Q_p$ in \eqref{kolo241}. To sum up, we need to show that  the natural morphism
  \begin{equation}\label{kolo241a}
    \R\tau_{*} {\mathbb D}(\iota_{\infty,*}(W{\otimes^{\LL_{\Box}}_{\Q_{p}}}C),\so(s,s))\to\R\Hhom_{\rm TVS}(\R\tau_{*}(\iota_{\infty,*}(W{\otimes^{\LL_{\Box}}_{\Q_{p}}}C)),\R\tau_{*}\so(s,s))
  \end{equation}
  is a quasi-isomorphism, for  any  Smith space $W\in C_{\Q_p}$. 
  We will need the following computation: 
  \begin{lemma} \label{frank2} Let $s\in\N$. Let $W\in C_{\Q_p}$ be  a Smith space. 
  \begin{enumerate}
  \item   The canonical morphism in ${\rm QCoh}(X_{\FF,C^{\flat}})$
      \begin{equation}
   f_1:\quad W^*{\otimes^{\LL_{\Box}}_{\Q_{p}}}{\mathbb D}(\iota_{\infty,*}C,\so(s,s))\to  {\mathbb D}(\iota_{\infty,*}(W{\otimes^{\LL_{\Box}}_{\Q_{p}}}C),\so(s,s)) 
\end{equation}
is a quasi-isomorphism. 
\item  The canonical morphisms in $\underline{\sd}(C,\Q_{p,\Box})$
\begin{align*}
 & f_2:\quad W^*\otimes^{\LL_{\Box}}_{\Q_{p}} \R\tau_*{\mathbb D}(\iota_{\infty,*}C,\so(s,s))\to \R\tau_*(W^*{\otimes^{\LL_{\Box}}_{\Q_{p}}}{\mathbb D}(\iota_{\infty,*}C,\so(s,s))),\\
& f_3:\quad  W\otimes^{\LL_{\Box}}_{\Q_{p}}\R\tau_*\iota_{\infty,*}C\to \R\tau_*(W{\otimes^{\LL_{\Box}}_{\Q_{p}}}\iota_{\infty, *}C)
\end{align*}
are  quasi-isomorphisms.
\end{enumerate}
\end{lemma}
\begin{proof} The case of the morphism $f_3$ is clear.  The morphism $f_2$ can be written as the following composition of quasi-isomorphisms
  \begin{align*}
   &  \R\tau_{*} (W^*{\otimes^{\LL_{\Box}}_{\Q_{p}}}{\mathbb D}(\iota_{\infty,*}C,\so(s,s)))
  \simeq  \R\tau_{*} (W^*{\otimes^{\LL_{\Box}}_{\Q_{p}}}\iota_{\infty,*}C(s-1)[-1])\\
&\simeq  W^*\otimes^{\LL_{\Box}}_{\Q_{p}}{\mathbb G}_a(s-1)[-1]\simeq  W^*\otimes^{\LL_{\Box}}_{\Q_{p}}\R\tau_*{\mathbb D}(\iota_{\infty,*}C,\so(s,s)).
\end{align*}

 For the morphism $f_1$, 
  we can pass to the category $\sd(\B^{\FF}_{C^{\flat},{\rm an}})^{\phi}$. Set 
$\B^{\prime}:=\B^{[u,v]}_{C^{\flat},{\rm an}}$. By Section \ref{monoidal},  it  suffices  to show that the canonical map 
  $$
    W^*{\otimes^{\LL}_{\Q_{p,\Box}}}\R\uHom_{\B^{\prime}}(C,\B^{\prime})\to   \R\uHom_{\B^{\prime}}(W{\otimes^{\LL}_{\Q_{p,\Box}}}C,\B^{\prime}) 
$$
is a quasi-isomorphism in $\sd(\B^{\prime})$.
Or, since 
\begin{equation}\label{kolo242}
\R\uHom_{\B^{\prime}}(C,\B^{\prime})\stackrel{\sim}{\leftarrow} C(-1)[-1],
\end{equation}
 that the composition in $\sd(\B^{\prime})$ 
\begin{equation}\label{kolo240}
W^*{\otimes^{\LL}_{\Q_{p,\Box}}}C(-1)[-1]\to   \R\uHom_{\B^{\prime}}(W{\otimes^{\LL}_{\Q_{p,\Box}}}C,\B^{\prime})
\end{equation}
is a quasi-isomorphism.
  
   For that, we write the map \eqref{kolo240} as a composition of quasi-isomorphisms in $\sd(\B^{\prime})$
\begin{align*}
 \R\uHom_{\B^{\prime}}(W{\otimes^{\LL}_{\Q_{p,\Box}}}C,\B^{\prime}) & \simeq  \R\uHom_{\B^{\prime}}(W{\otimes^{\LL}_{\Q_{p,\Box}}}\B^{\prime},\R\uHom_{\B^{\prime}}(C,\B^{\prime}) )\\ &  \stackrel{\sim}{\leftarrow} \R\uHom_{\B^{\prime}}(W{\otimes^{\LL}_{\Q_{p,\Box}}}\B^{\prime},C)(-1)[-1] 
   \simeq (W^*{\otimes^{\LL}_{\Q_{p,\Box}}}C)(-1)[-1].
\end{align*}
 Here, the first quasi-isomorphism is the internal tensor-hom adjunction; the second quasi-isomorphism follows from \eqref{kolo242}.  The last quasi-isomorphism is induced by   the following commutative diagram in $\sd(\Q_{p,\Box})$
\begin{equation}\label{kolo243}
\xymatrix@R4mm{
\R\uHom_{\B^{\prime}}(W{\otimes^{\LL}_{\Q_{p,\Box}}}\B^{\prime},C)\ar[r] & \uHom_{\B^{\prime}}(W{\otimes^{\LL}_{\Q_{p,\Box}}}\B^{\prime},C)\\
\R\uHom_{\Q_{p,\Box}}(W,C)\ar[u]^{\wr}\ar[r]^{\sim} & \uHom_{\Q_{p,\Box}}(W,C)\ar[u]^{\wr},
}
\end{equation}
where  we have used  that  $W$ is an internal  projective object in solid $\Q_p$-modules,  and the fact that ${\uHom}_{\Q_{p,\Box}}(W,V)\stackrel{\sim}{\leftarrow}W^*{\otimes^{\LL}_{\Q_{p,\Box}}}V$, for a Fr\'echet space $V$ over $\Q_p$ (see \cite[Th. 3.40]{RR1}). We note here that the  arrows in diagram \eqref{kolo243} are $\B^{\prime}$-linear. This finishes the proof of the first claim of the lemma.
 \end{proof}

     By Lemma \ref{frank2}, to show that the morphism \eqref{kolo241a} is a quasi-isomorphism, it suffices to show that the natural morphism
  \begin{equation}
   W^*\otimes_{\Q_{p}}^{\LL_{\Box}}\R\tau_*{\mathbb D}(\iota_{\infty,*}C, \so(s,s))\to \R\Hhom_{\rm TVS}(W\otimes^{\LL_{\Box}}_{\Q_{p}}{\mathbb G}_a,\Q_p(s))
  \end{equation}
  is a quasi-isomorphism. But this morphism factors as the composition
  $$
  \xymatrix@C=-2cm{
  W^*\otimes_{\Q_{p}}^{\LL_{\Box}}\R\tau_*{\mathbb D}(\iota_{\infty,*}C, \so(s,s))\ar[rr] \ar[dr]^{\sim}& & \R\Hhom_{\rm TVS}(W\otimes^{\LL_{\Box}}_{\Q_{p}}{\mathbb G}_a,\Q_p(s))\\
  & W^*\otimes^{\LL_{\Box}}_{\Q_{p}} \R\Hhom_{\rm TVS}({\mathbb G}_a,\Q_p(s))\ar[ru]^{f_0},
  }
  $$
  where the left quasi-isomorphism follows from Proposition \ref{recall1}. 
  Hence it remains to show that the morphism $f_0$ above is a quasi-isomorphism. 
  
        Since, by \cite[Lemma 3.27]{TVS},  we have 
      $$
       \R\Hhom_{\rm TVS}(W\otimes^{\LL_{\Box}}_{\Q_{p}}{\mathbb G}_a,\Q_p) \simeq  \R\Hhom_{\rm TVS}(\underline{W},\R\Hhom_{\rm TVS}({\mathbb G}_a,\Q_p)),
      $$
      we can rewrite the morphism $f_0$ (untwisted by $s$) as the canonical morphism 
      $$
      W^*\otimes^{\LL_{\Box}}_{\Q_{p}} \R\Hhom_{\rm TVS}({\mathbb G}_a,\Q_p)\to \R\Hhom_{\rm TVS}(\underline{W},\R\Hhom_{\rm TVS}({\mathbb G}_a,\Q_p)).
      $$
      Or, since $\R\Hhom_{\rm TVS}({\mathbb G}_a,\Q_p)\simeq {\mathbb G}_a(-1)[-1]$, as (the shift of)  the canonical morphism
      $$
            W^*\otimes^{\LL_{\Box}}_{\Q_{p}} {\mathbb G}_a\to \R\Hhom_{\rm TVS}(\underline{W},{\mathbb G}_a).
 $$
 It remains thus to show that this morphism is a quasi-isomorphism. 
 
     Let $S\in {\rm sPerf}_C$. Using the enriched Yoneda Lemma 
(see \cite[Lemma\,3.39]{TVS}), we have quasi-isomorphisms
 \begin{align*} & W^*\otimes^{\LL_{\Box}}_{\Q_{p}}{\mathbb G}_a(S)\stackrel{\sim}{\to} \R\uHom^{\Box}_{\Q_p}({W},{\mathbb G}_a(S)),\\
 &\R\Hhom_{\rm TVS}(\underline{W},{\mathbb G}_a)(S)\simeq \R\uHom_{\rm TVS}(\underline{W}\otimes^{\LL_{\Box}}_{\Q_{p}}\Q_p[h^{\rm top}_S]^{\Box},{\mathbb G}_a)
  \simeq  \R\uHom_{\rm TVS}(\underline{W},\R\Hhom_{\rm TVS}(\Q_p[h^{\rm top}_S]^{\Box}, {\mathbb G}_a)).
  \end{align*}
Here  the first quasi-isomorphism follows 
  from  the fact that $W$ is a Smith space (hence an internal projective object in solid modules). This reduces us to showing that the canonical morphism
  $$
  \R\uHom^{\Box}_{\Q_p}({W},{\mathbb G}_a(S))\to \R\uHom_{\rm TVS}(\underline{W},\R\Hhom_{\rm TVS}(\Q_p[h^{\rm top}_S]^{\Box}, {\mathbb G}_a))
  $$
  is a quasi-isomorphism. But, since
  $$\R\Gamma({\rm Spa}(C)^{\rm top}, \R\Hhom_{\rm TVS}(\Q_p[h^{\rm top}_S]^{\Box}, {\mathbb G}_a))\simeq \R\uHom_{\rm TVS}(\Q_p[h^{\rm top}_S]^{\Box}, {\mathbb G}_a)\simeq {\mathbb G}_a(S),
  $$ this follows from 
  \cite[Lemma 3.27]{TVS}. 
  
   \vskip2mm 
    ($\bullet$)  For a general smooth partially proper variety $X$,
choose a split analytic hypercovering $Y_{\jcdot}\to X$ such that
each $Y_n$ is a disjoint union of Stein open subspaces whose
de Rham cohomology is of finite rank.  We compute in
$\underline{\sd}({\rm Spa}(C),\Q_{p,\Box})$
   \begin{align*}
{\mathbb R}_{\proeet}(X_C,\Q_p)& \simeq{\rm Rlim}_{n} {\mathbb R}_{\proeet}(Y_{n,C},\Q_p)\simeq{\rm Rlim}_{n}\R\Hhom_{\rm TVS}({\mathbb R}_{\proeet,c}(Y_{n,C},\Q_p(d))[2d],\Q_p)\\
& \simeq\R\Hhom_{\rm TVS}(\colim_{n}{\mathbb R}_{\proeet,c}(Y_{n,C},\Q_p(d))[2d],\Q_p)\\
 & \simeq\R\Hhom_{\rm TVS}({\mathbb R}_{\proeet,c}(X_C,\Q_p(d))[2d],\Q_p),
   \end{align*}
   as wanted. 
   Here, the second quasi-isomorphism follows componentwise from
the case already proved, using that ordinary cohomology takes
disjoint unions to products, compactly supported cohomology
takes them to direct sums, and internal derived Hom takes
direct sums in its first argument to products.
\end{proof}
\subsection{Algebraic Poincar\'e duality} The topological Poincar\'e duality from Theorem \ref{VS-duality} has an algebraic version, which we will now present. 
Let $X$ be a smooth partially proper  rigid analytic variety over $K$.
The algebraic analog of the pairing in \eqref{kicius1}  induces a duality map in 
 ${\sd}({\rm Spa}(C)_{\proeet},\Q_{p})$
\begin{equation}\label{bonn10A}
\gamma^{\rm alg}_{X_C}:\quad {\mathbb R}^{\rm alg}_{\proeet}(X_C,\Q_p)\to \R\Hhom_{\rm VS}({\mathbb R}^{\rm alg}_{\proeet,c}(X_C,\Q_p(d))[2d],{\Q}_p).
\end{equation}

\begin{corollary}{\rm (Algebraic pro-\'etale  duality)}\label{VSS-duality} Let $X$ be a smooth partially proper  rigid analytic variety over $K$ of dimension $d$.  The duality map
\eqref{bonn10A} is a quasi-isomorphism. In particular,
we have a quasi-isomorphism in $\sd(\Q_{p})$
\begin{equation}\label{deszcz111}
\gamma^{\rm alg}_{X_C}:\quad{\R}\Gamma_{\proeet}(X_C,\Q_p)\stackrel{\sim}{\to}\R\Hom_{\rm VS}({\mathbb R}^{\rm alg}_{\proeet,c}(X_C,\Q_p(d))[2d],\Q_p).
\end{equation}
\end{corollary}
\begin{proof} Apply the projection functor $\R\pi_*$ to the map \eqref{bonn10A}. We obtain the horizontal map in the following commutative diagram
$$
\xymatrix{
 \R\pi_*{\mathbb R}^{\rm alg}_{\proeet}(X_C,\Q_p)\ar[r]^-{\R\pi_*\gamma^{\rm alg}_{X_C}} \ar[rd] \ar[d]^{\wr}& \R\pi_*\R\Hhom_{\rm VS}({\mathbb R}^{\rm alg}_{\proeet,c}(X_C,\Q_p(d))[2d],{\Q}_p)\ar[d]^{\wr}\\
{\mathbb R}_{\proeet}(X_C,\Q_p)\ar[rd]^-{\gamma_{X_C}}_{\sim} & \R\Hhom_{C^{\rm top}}(\R\pi_*{\mathbb R}^{\rm alg}_{\proeet,c}(X_C,\Q_p(d))[2d],\R\pi_*{\Q}_p)\ar[d]^{\wr}\\
 & \R\Hhom_{C^{\rm top}}({\mathbb R}_{\proeet,c}(X_C,\Q_p(d))[2d],{\Q}_p).
}
$$
The top right vertical arrow is a quasi-isomorphism by Proposition \ref{recall1}; the map $\gamma_{X_C}$ is a quasi-isomorphism by Theorem \ref{VS-duality}. It follows that so is the map $\R\pi_*\gamma^{\rm alg}_{X_C}$. 

  Now, we have a \index{rtau@\rtau}functor $\R\eta_*: \sd^{\rm psh}({\rm sPerf}_C,{\rm Mod}^{\rm cond}_{\Q_p})\to  \sd^{\rm psh}({\rm sPerf}_C,\Q_p)$ from topological presheaves to algebraic presheaves (``evaluation at $*$'') such that $\R\eta_*\R\pi_*\simeq \R\iota_*$, the canonical forgetful functor from sheaves to presheaves (see \cite[Sec. 2.1.2]{TVS} for details).
 Applying $\R\eta_*$ to $\R\pi_*\gamma^{\rm alg}_{X_C}$ we get that the map $\R\iota_*\gamma^{\rm alg}_{X_C}$ is a quasi-isomorphism. Since $\R\iota_*$ is conservative, the map $\gamma^{\rm alg}_{X_C}$ is a quasi-isomorphism, as wanted. 
\end{proof}
\subsection{Verdier exact sequence} In the Stein case, the duality \eqref{deszcz111} takes a simple form. 
\begin{corollary} \label{rainy-day} Let $X$ be a smooth Stein variety over $K$ of dimension $d$. Let $i\geq 0$. 
There exists a short exact sequence in {\rm TVS}'s
$$
0\to {\mathcal Ext}^1_{\rm TVS}({\mathbb H}^{2d-i+1}_{\proeet,c}(X_C,\Q_p(d)),\Q_p)\to {\mathbb H}^i_{\proeet}(X_C,\Q_p)\to{\mathcal Hom}_{\rm TVS}({\mathbb H}^{2d-i}_{\proeet,c}(X_C,\Q_p(d)),\Q_p)\to 0
$$
In particular, there exists a short exact sequence in $\Q_{p,\Box}$
$$
0\to\underline{\Ext}^1_{\rm TVS}({\mathbb H}^{2d-i+1}_{\proeet,c}(X_C,\Q_p(d)),\Q_p)\to {H}^i_{\proeet}(X_C,\Q_p)\to\uHom_{\rm TVS}({\mathbb H}^{2d-i}_{\proeet,c}(X_C,\Q_p(d)),\Q_p)\to 0
$$
\end{corollary}
\begin{proof} The second claim follows easily from the first claim. For the first claim, 
having Theorem \ref{VS-duality} and the spectral sequence 
$$
E^{a,b}_2={\mathcal Ext}^a_{\rm TVS}({\mathbb H}^{-b}_{\proeet,c}(X_C,\Q_p),\Q_p)\Rightarrow H^{a+b}(\R\Hhom_{\rm TVS}({\mathbb R}_{\proeet,c}(X_C,\Q_p),\Q_p))
$$
it suffices to show that
 \begin{equation} 
{\mathcal Ext}^a_{\rm TVS}({\mathbb H}^{b}_{\proeet,c}(X_C,\Q_p),\Q_p)=0,\quad a\geq 2.
\end{equation}
   Let $r\geq 2d$. By Proposition \ref{china1}, we have the quasi-isomorphism
  $$
 \se_{\proeet,c}(X_C,\Q_p(r))\simeq \se_{\synt,c}(X_C,\Q_p(r)),
  $$
  which, by \eqref{kawa2}  and \eqref{kolobrzeg1a},   yields an exact sequence 
  \begin{equation}\label{presentation1}
\cdots\stackrel{\iota_{b-1}}{\to}{\mathbb  D}{\mathbb R}^{b-1}_c(X_C,r)  \stackrel{}{\to} {\mathbb H}^b_{\proeet,c}(X_C,\Q_p(r))\stackrel{}{\to} {\mathbb H}{\mathbb K}^b_c(X_C,r)\stackrel{\iota_b}{\to}{\mathbb  D}{\mathbb R}^{b}_c(X_C,r)\to\cdots
  \end{equation}
  where we set 
  \begin{align*}
  {\mathbb H}{\mathbb K}^b_c(X_C,r) & := H^b([\R\Gamma_{\hk,c}(X_C)\otimes\mathbb{B}^+_{\st}]^{N=0,\phi=p^r})\simeq (H^b_{\hk,c}(X_C)\otimes^{{\Box}}_{\breve{C}}\mathbb{B}^+_{\st})^{N=0,\phi=p^r}\\
  {\mathbb D}{\mathbb R}^b_c(X_C,r) & :=H^{b-d}(H^d_c(X,\so)\otimes^{{\Box}}_{K}(\mathbb{B}^+_{\dr}/t^{r})\to H^d_c(X,\Omega^1)\otimes^{{\Box}}_{K}(\mathbb{B}^+_{\dr}/t^{r-1})\to\cdots\to H^d_c(X,\Omega^d)\otimes^{{\Box}}_{K}(\mathbb{B}^+_{\dr}/t^{r-d})) 
  \end{align*}
  We used here the natural isomorphisms ($?=-, c$)
  $$H^b\R\tau_*\se_{\hk,?}(X_C,r)\simeq \mathbb{HK}^b_?(X_C,r), \quad H^b\R\tau_*\se_{\dr,?}(X_C,r)\simeq \mathbb{DR}^b_?(X_C,r),
  $$
  that are compatible with the Hyodo-Kato map and which 
    follow immediately from the computations in the proof of Proposition \ref{comp-elgin}.
 
     Let $i\geq 0$. We will need to understand the maps $\iota_{i}$ from \eqref{presentation1} better. We have a commutative diagram in ${\rm TVS}$, where we set $s=r-i+d-1$: 
  \begin{equation}\label{air1}
  \xymatrix@C=4mm@R=4mm{
        0\ar[r] & (H^i_{\hk,c}(X_C)\otimes^{{\Box}}_{\breve{C}}t^{s}\mathbb{B}^+_{\st})^{N=0,\phi=p^r}\ar[d]^{\iota_{\hk}} \ar[r] &   \mathbb{HK}^i_c(X_C,r)\ar[d]^{\iota_i}\ar[r]  & H^i_{\dr,c}(X)\otimes^{{\Box}}_{K}(\mathbb{B}^+_{\dr}/F^{s})\ar@{=}[d] \\
     0\ar[r] & (H^d_c(X,\Omega^{i-d})/{\rm Im}\,d)\otimes^{{\Box}}_{K}t^{s}\mathbb{G}_a \ar[d] \ar[r] & \mathbb{DR}^i_c(X_C,r)\ar[r] \ar[d]^{f_1} & H^i_{\dr,c}(X)\otimes^{{\Box}}_{K}(\mathbb{B}^+_{\dr}/F^{s})\ar[r] \ar[d]^{f_3} &0\\
         0\ar[r] &  \mathbb{E}^i_1\ar[r]^{f_2}\ar[d]  &  \mathbb{E}^i\ar[r] \ar[d] & \mathbb{E}^i_2\ar[r]  &0\\
         &  0 & 0
         }
  \end{equation}
 Here, the middle row comes  from \eqref{kolo10}; the top row from \cite[Lemma 8.1]{AGN}.  
We have defined  ${\mathbb E}^i:=\coker(\iota_i)$;
$ \mathbb{E}^i_1$  is the image of $(H^d_c(X,\Omega^{i-d})/{\rm Im}\,d)\otimes^{\Box}_{K}t^{s}\mathbb{G}_a$ under the map $f_1$,
  and $\mathbb{E}^i_2$ is the cokernel of the map $f_2$. The rows are exact;  so are the first and the second columns.      Moreover,  the map $\iota_{\hk}:  (H^i_{\hk,c}(X_C)\otimes^{{\Box}}_{\breve{C}}t^{s}\mathbb{B}^+_{\st})^{N=0,\phi=p^r}\to (H^d_c(X,\Omega^{i-d})/{\rm Im}\,d)\otimes^{{\Box}}_{K}t^s\mathbb{G}_a$ factors as 
  \begin{equation}\label{wreszcie0}
 (H^i_{\hk,c}(X_C)\otimes^{{\Box}}_{\breve{C}}t^{s}\mathbb{B}^+_{\st})^{N=0,\phi=p^r}\lomapr{\iota_{\hk}} 
H^i_{\dr,c}(X)\otimes^{{\Box}}_{K}t^s\mathbb{G}_a\lomapr{\can} (H^d_c(X,\Omega^{i-d})/{\rm Im}\,d)\otimes^{{\Box}}_{K}t^s\mathbb{G}_a,
  \end{equation}
  where the second map is an injection with quotient $(H^d_c(X,\Omega^{i-d})/{\rm Ker}\,d)\otimes^{{\Box}}_{K}t^s\mathbb{G}_a$.  It follows that we have an exact sequence
  \begin{equation}\label{wreszcie}
  0\to \mathbb{V}\to   \mathbb{E}^i_1 \to  (H^d_c(X,\Omega^{i-d})/{\rm Ker}\,d)\otimes^{\Box}_{K}t^s\mathbb{G}_a \to 0,
  \end{equation}
  where $\mathbb{V}$ is a (topological) BC when the de Rham cohomology has finite rank.

\vskip2mm
($\bullet$) {\em Finite rank case.}  Assume first that the de Rham  cohomology has finite rank. 
  Let $
  {\mathbb A}^i:=\ker(\iota_i). 
  $ 
It is a   BC: this is because the map $\iota_i$ factors through de Rham cohomology, which is a  BC (see \eqref{wreszcie0}). 
  We have the exact sequence
  \begin{equation}\label{presentation2}
  {\mathcal Ext}^a_{\rm TVS}({\mathbb A}^b,\Q_p)\to {\mathcal Ext}^a_{\rm TVS}({\mathbb H}^{b}_{\proeet,c}(X_C,\Q_p),\Q_p)\to {\mathcal Ext}^a_{\rm TVS}({\mathbb E}^{b-1},\Q_p).
  \end{equation}
 Hence  it suffices to show that, for $a\geq 2$, we have 
 \begin{equation}\label{presentation3}
    {\mathcal Ext}^a_{\rm TVS}({\mathbb A}^b,\Q_p)=0, \quad   {\mathcal Ext}^a_{\rm TVS}({\mathbb E}^{b-1},\Q_p)=0. 
\end{equation}
  
    This is clear for ${\mathbb A}^b$ by the Banach-Colmez Ext-vanishing recalled in \eqref{BC}. For ${\mathbb E}^{b-1}$, first use the exact sequence from \eqref{air1}
    $$
    0\to {\mathbb E}^{b-1}_1\to {\mathbb E}^{b-1}\to {\mathbb E}^{b-1}_2\to 0.
    $$
    ${\mathbb E}^{b-1}_2$ is a BC: it is the cokernel of the map between BC's
    $$
    \mathbb{HK}^{b-1}_c(X_C,r)\to H^{b-1}_{\dr,c}(X)\otimes^{\Box}_K(\mathbb B^+_{\dr}/F^{r-b+d}).
    $$
    Hence ${\mathcal Ext}^a_{\rm TVS}({\mathbb E}^{b-1}_2,\Q_p)=0$ for $a\geq 2$. Applying the exact sequence \eqref{wreszcie} to ${\mathbb E}^{b-1}_1$, we are therefore reduced to showing that, for $a\geq 2$, we have 
    $$
     {\mathcal Ext}^a_{\rm TVS}(\mathbb{V},\Q_p)=0,\quad  {\mathcal Ext}^a_{\rm TVS}(W\otimes^{\Box}_{K}\mathbb{G}_a,\Q_p)=0,
    $$
    where we set $W=H^d_c(X,\Omega^{b-1-d})/{\rm Ker}\,d$. 
    Since $\mathbb{V}$ is a BC, the first equality is clear. 
     For the second one,  if $K=\Q_p$, we have in $\Q_{p,\Box}$ (see \cite[Lemma 3.27]{TVS})
    $$
     {\mathcal Ext}^a_{\rm TVS}(W\otimes^{\Box}_{K}\mathbb{G}_a,\Q_p)\simeq W^*\otimes^{\Box}_{\Q_{p} }    {\mathcal Ext}^a_{\rm TVS}(\mathbb{G}_a,\Q_p)=0,
    $$
    as wanted.  We used here that $W$ is of compact type.  

      For a general $K$,  assume first that $W$ is a Smith space. Then we can write $W=W_{\Q_p}\otimes^{\LL}_{\Q_p,\Box}K$, for a Smith space $W_{\Q_p}$ over $\Q_p$.
   It follows that the canonical morphism
   $$
   W^*\otimes^{\Box}_K    {\mathcal Ext}^a_{\rm TVS}(\mathbb{G}_a,\Q_p)\to     {\mathcal Ext}^a_{\rm TVS}(W\otimes^{\Box}_{K}\mathbb{G}_a,\Q_p)
   $$
   is an isomorphism. 
   For a general $W$,  write $W\simeq \colim_n W_n$ as a compact colimit of Smith spaces $W_n$ over $K$ and note that 
   $${\mathcal Ext}^a_{\rm TVS}(W\otimes^{\Box}_{K}\mathbb{G}_a,\Q_p)\simeq \lim_n {\mathcal Ext}^a_{\rm TVS}(W_n\otimes^{\Box}_{K}\mathbb{G}_a,\Q_p).
   $$
   This is because $\R^j\lim_n{\mathcal Ext}^{a-1}_{\rm TVS}(W_n\otimes^{\Box}_{K}\mathbb{G}_a,\Q_p)=0$ for $j\geq 1$. Indeed,  $\R^j\lim_n=0$, for $j\geq 2$, on countable inverse systems in the abelian category underlying ${\rm TVS}$ because products are exact since 
   this category has projective generators (see \cite[Sec.\,3.2.1]{TVS}).  Hence it remains to show 
 that $\R^1\lim_n{\mathcal Ext}^{a-1}_{\rm TVS}(W_n\otimes^{\Box}_{K}\mathbb{G}_a,\Q_p)=0$: this follows from the Smith case for $a\geq 3$; for $a=2$, we have 
   \begin{align*}
   \R^1\lim_n{\mathcal Ext}^{1}_{\rm TVS}(W_n\otimes^{\Box}_{K}\mathbb{G}_a,\Q_p) &     \stackrel{\sim}{\leftarrow} \R^1\lim_n(W^*_n\otimes^{\Box}_K{\mathcal Ext}^{1}_{\rm TVS}(\mathbb{G}_a,\Q_p))\\
    & \simeq \R^1\lim_n(W^*_n\otimes^{\Box}_K\mathbb{G}_a)
     \simeq  (\R^1\lim_nW^*_n)\otimes^{\Box}_K\mathbb{G}_a=0.
   \end{align*}
   The last isomorphism follows from the fact that  $\{W_n^*\}$ is a pro-system of Banach spaces with dense transition maps. The penultimate isomorphism follows from the following fact\footnote{Probably well-known but we did not find a reference.}:
   \begin{lemma}
   Let $\{V_n\}$ be a set of Banach spaces over $K$ and let $V$ be a Banach space over $K$. Then the canonical map
   $$
   V\otimes^{\Box}_K\prod_nV_n\to \prod_n(V\otimes^{\Box}_KV_n)
   $$
   is an isomorphism.
   \end{lemma}
   \begin{proof}
   Write $V=\uHom(\Z[T],K)$, for a profinite set $T$. Then
   \begin{align*}
   V\otimes^{\Box}_K\prod_nV_n & \simeq \uHom(\Z[T],K)\otimes^{\Box}_K\prod_nV_n\simeq \uHom(\Z[T],\prod_nV_n)\\
    & \simeq \prod_n\uHom(\Z[T],V_n)\simeq 
   \prod_n(\uHom(\Z[T],K)\otimes^{\Box}_KV_n)\simeq \prod_n(V\otimes^{\Box}_KV_n).
   \end{align*}
   The second isomorphism follows from the fact that the product $\prod_nV_n$ is a nuclear $K$-vector space and the penultimate one from the fact that so is every $V_n$. 
   \end{proof}
   \vskip2mm
   ($\bullet$) {\em General case.}   For  a general smooth Stein variety $X$ over $K$, we cover $X$ with an exhaustive sequence $\{X_n\}$ of Stein varieties with finite dimensional de Rham cohomologies. We have
    \begin{align*}
{\mathcal Ext}^a_{\rm TVS}(\mathbb{H}^{b}_{\proeet,c}(X_C,\Q_p),\Q_p)\simeq {\mathcal Ext}^a_{\rm TVS}(\colim_n\mathbb{H}^{b}_{\proeet,c}(X_{n,C},\Q_p),\Q_p).
\end{align*}
This yields the exact sequence
\[
\begin{aligned}
0\to {} & \R^1\lim_n{\mathcal Ext}^{a-1}_{\rm TVS}(\mathbb{H}^{b}_{\proeet,c}(X_{n,C},\Q_p),\Q_p)\\
\to {} & {\mathcal Ext}^a_{\rm TVS}(\mathbb{H}^{b}_{\proeet,c}(X_C,\Q_p),\Q_p)
\to \lim_n{\mathcal Ext}^a_{\rm TVS}(\mathbb{H}^{b}_{\proeet,c}(X_{n,C},\Q_p),\Q_p)\to 0.
\end{aligned}
\]
Hence, by the above, $ {\mathcal Ext}^a_{\rm TVS}(\mathbb{H}^{b}_{\proeet,c}(X_C,\Q_p),\Q_p)=0$, for $a\geq 3$. For $a=2$, we have
$$
\R^1\lim_n{\mathcal Ext}^{1}_{\rm TVS}(\mathbb{H}^{b}_{\proeet,c}(X_{n,C},\Q_p),\Q_p)\stackrel{\sim}{\to} {\mathcal Ext}^2_{\rm TVS}(\mathbb{H}^{b}_{\proeet,c}(X_C,\Q_p),\Q_p).
$$
It suffices thus to show that
$$
\R^1\lim_n{\mathcal Ext}^{1}_{\rm TVS}(\mathbb{H}^{b}_{\proeet,c}(X_{n,C},\Q_p),\Q_p)=0. 
$$
From the exact sequence \eqref{presentation2} and isomorphisms \eqref{presentation3}, we get
the exact sequence (for $X_n$)
$$
\Hhom_{\rm TVS}(\mathbb{E}^{b-1}_n,\Q_p) \to { \mathcal Ext}^1_{\rm TVS}(\mathbb{A}^b_n,\Q_p)\to {\mathcal Ext}^1_{\rm TVS}(\mathbb{H}^{b}_{\proeet,c}(X_{n,C},\Q_p),\Q_p)\to {\mathcal Ext}^1_{\rm TVS}(\mathbb{E}^{b-1}_n,\Q_p)\to 0
$$
This yields the exact sequence 
$$
\R^1\lim_n{\mathcal Ext}^{1}_{\rm TVS}({\mathbb{A}}^b_n,\Q_p)\to \R^1\lim_n{\mathcal Ext}^{1}_{\rm TVS}(\mathbb{H}^{b}_{\proeet,c}(X_{n,C},\Q_p),\Q_p)\to \R^1\lim_n {\mathcal Ext}^1_{\rm TVS}(\mathbb{E}^{b-1}_n,\Q_p)\to 0.
$$
Since ${\mathcal Ext}^{1}_{\rm TVS}({\mathbb{A}}^b_n,\Q_p)$ is a BC we have $\R^1\lim_n{\mathcal Ext}^{1}_{\rm TVS}({\mathbb{A}}^b_n,\Q_p)=0$ as  we have Mittag-Leffler in this setting. It suffices thus to show that  
$\R^1\lim_n {\mathcal Ext}^1_{\rm TVS}(\mathbb{E}^{b-1}_n,\Q_p)=0$.

    But from diagram  \eqref{air1}, we see that  it suffices to show that
 $$
 \R^1\lim_n {\mathcal Ext}^1_{\rm TVS}(\mathbb{E}^{b-1}_{i,n},\Q_p)=0,\quad i=1,2.
 $$
 For $i=2$ this is clear because ${\mathcal Ext}^1_{\rm TVS}(\mathbb{E}^{b-1}_{2,n},\Q_p)$ is a BC (because  $\mathbb{E}^{b-1}_{2,n}$ is the cokernel of the map $\mathbb{HK}^{b-1}_c(X_C,r) \to H^{b-1}_{\dr,c}(X)\otimes^{\Box}_K({\mathbb B}^+_{\dr}/F^s)$, hence a BC). It remains to treat the case $i=1$. For that we use the exact sequence \eqref{wreszcie}. Since $\mathbb{V}_n$ is a BC, it suffices to show that
 $$
  \R^1\lim_n {\mathcal Ext}^1_{\rm TVS}(W_n\otimes^{\Box}_{K}\mathbb{G}_a,\Q_p)=0,
 $$
 where $W_n=H^d_c(X_n,\Omega^{b-1-d})/{\rm Ker}\, d$. 
Note that $W_n\hookrightarrow H^d_c(X_n,\Omega^{b-d})=:W^{\prime}_n$.  Hence we have the surjection $$
  {\mathcal Ext}^1_{\rm TVS}(W^{\prime}_n\otimes^{\Box}_{K}\mathbb{G}_a,\Q_p)\twoheadrightarrow {\mathcal Ext}^1_{\rm TVS}(W_n\otimes^{\Box}_{K}\mathbb{G}_a,\Q_p).
$$
It suffices thus to show that
\begin{equation}\label{air2}
  \R^1\lim_n {\mathcal Ext}^1_{\rm TVS}(W^{\prime}_n\otimes^{\Box}_{K}\mathbb{G}_a,\Q_p)=0.
\end{equation}

      We claim that the canonical morphism
   $$
   W^{\prime,*}_n\otimes^{\Box}_{K}{\mathcal Ext}^1_{\rm TVS}(\mathbb{G}_a,\Q_p)\to {\mathcal Ext}^1_{\rm TVS}(W^{\prime}_n\otimes^{\Box}_{K}\mathbb{G}_a,\Q_p)
   $$
   is an isomorphism. Indeed, we can write $W^{\prime}_n=\colim_mW^{\prime}_{n,m}$ as a compact colimit of Smith spaces and argue as above using the fact that 
   $${\mathcal Hom}_{\rm TVS}(W^{\prime}_{n,m}\otimes^{\Box}_{K}\mathbb{G}_a,\Q_p)\simeq W^{\prime,*}_{n,m}\otimes^{\Box}_{K}
   {\mathcal Hom}_{\rm TVS}(\mathbb{G}_a,\Q_p)=0.
   $$
   Hence we have a functorial in $n$ isomorphism
   $$
    W^{\prime,*}_n\otimes^{\Box}_{K}{\mathbb G}_a\simeq  {\mathcal Ext}^1_{\rm TVS}(W^{\prime}_n\otimes^{\Box}_{K}\mathbb{G}_a,\Q_p).
   $$
   Since we have $W^{\prime,*}_n\simeq \Omega^{2d-b}(X_n)$,
 \eqref{air2} holds because the pro-system $\{\Omega^{2d-b}(X_{n})\}$, $n\in\N$, is equivalent to a pro-system of Banach spaces $\{V_n\}$ with  dense transition maps.
\end{proof}
\subsection{Examples} 
We will discuss in some detail here dualities for Stein curves and Drinfeld spaces. 
\subsubsection{Proper varieties} 
Let $X$ be a smooth proper rigid analytic variety over $K$. Then, since pro-\'etale cohomology of $X_C$  is finite dimensional over $\Q_p$,  Theorem \ref{VS-duality} yields a duality isomorphism 
$$
{H}^i_{\proeet}(X_C,\Q_p)\stackrel{\sim}{\to}
{H}^{2d-i}_{\proeet}(X_C,\Q_p(d))^*.
$$
This result is known from the work of Zavyalov \cite{Zav} and Mann \cite{Mann}. 
\subsubsection{Stein curves}
 Let $X$ be a geometrically connected smooth Stein curve over $K$. 
 From comparison theorems (see \cite[Th. 6.14]{CN5}) we get the following isomorphism and a short exact sequence in $\Q_{p,\Box}$
 \begin{align}\label{leje1-25}
 & H^0_{\proeet}(X_C,\Q_p)\simeq \Q_p,\\
 0\to  \so(X_C)/C \to &  H^1_{\proeet}(X_C,\Q_p(1))\to
  (H^1_{\hk}(X_C)\otimes^{\Box}_{\breve{C}}\B^+_{\st})^{N=0,\phi=p}\to  0\notag
 \end{align}
 They lift to {\rm TVS}'s. Similarly, for compactly supported cohomology, by \cite[Sec.\,7.2]{AGN}, we get the following isomorphism and exact sequence in $\Q_{p,\Box}$
 \begin{align*}
 & H^1_{\proeet,c}(X_C,\Q_p(1))\simeq (H^1_{\hk,c}(X_C)\otimes^{\Box}_{\breve{C}}\B^+_{\st})^{N=0,\phi=1}(1),\\
& \to (H^1_{\hk,c}(X_C)\otimes^{\Box}_{\breve{C}}\B^+_{\st})^{N=0,\phi=p^2} \lomapr{\iota_{\hk}} H^1{\rm DR}_c(X_C,2)\to H^2_{\proeet,c}(X_C,\Q_p(2))\to \Q_p(1) \to 0,
 \end{align*}
  where we set ${\rm DR}_c(X_C,2):= \R\Gamma_{\dr,c}(X_C,\B^+_{\dr})/F^2$. We note that, if the de Rham cohomology of $X$ is of finite rank over $K$, then  $(H^1_{\hk,c}(X_C)\otimes^{\Box}_{\breve{C}}\B^+_{\st})^{N=0,\phi=1}$ is a finite rank $\Q_p$-vector space because the slopes of Frobenius
on $H^1_{\hk,c}(X_C)$ are~$\geq 0$ (see \cite[Remark 7.10]{AGN}). 
Again, everything lifts to {\rm TVS}'s. 
  
   By Corollary \ref{rainy-day} we get
   a short exact sequence in $\Q_{p,\Box}$
$$
0\to\underline{\Ext}^1_{\rm TVS}({\mathbb H}^{2d-i+1}_{\proeet,c}(X_C,\Q_p(d)),\Q_p)\to {H}^i_{\proeet}(X_C,\Q_p)\to
\uHom_{\rm TVS}({\mathbb H}^{2d-i}_{\proeet,c}(X_C,\Q_p(d)),\Q_p)\to 0
$$
Hence, using the above computations and $d=1$ since $X$ is a curve, 
we get  the following isomorphism and a short exact sequence in $\Q_{p,\Box}$
\begin{align}\label{leje2-25}
 & H^0_{\proeet}(X_C,\Q_p)\simeq \uHom_{\rm TVS}({\mathbb H}^{2}_{\proeet,c}(X_C,\Q_p(1)),\Q_p)\\
0\to\underline{\Ext}^1_{\rm TVS}({\mathbb H}^{2}_{\proeet,c}(X_C,& \Q_p(1)),\Q_p)\to {H}^1_{\proeet}(X_C,\Q_p)\to\uHom_{\rm TVS}({\mathbb H}^{1}_{\proeet,c}(X_C,\Q_p(1)),\Q_p)\to 0\notag
\end{align}
It is tempting to think that the exact sequence in \eqref{leje2-25} recovers the exact sequence in \eqref{leje1-25} but this is not the case: if the de Rham cohomology of $X$ is of finite rank over $K$, the term on the right in \eqref{leje2-25} is a finite rank $\Q_p$-vector space while the term on the right in \eqref{leje1-25} will, in general, have a nontrivial $C$-part.
\subsubsection{Drinfeld space}
 Let $K$ be a finite extension of $\Q_p$ and let $d\geq 1$. Let ${\mathcal H}^d_K$ be the Drinfeld space of dimension $d$ over $K$.
By \cite[Th. 1.3]{CDN3}, \cite[Sec.\,8.2.2, (8.13)]{AGN}, we have exact sequences in $\Q_{p,\Box}$ ($i\geq 0$)
\begin{align}\label{leje3-25}
 0\to \Omega^{i-1}({\mathcal H}^d_C)/\ker d\to  & H^i_{\proeet}({\mathcal H}^d_C,\Q_p(i))\to {\rm Sp}_i(\Q_p)^*\to 0\\
 0\to H^d_c({\mathcal H}^d_C, \Omega^{i-d-1})/\ker d\to & H^i_{\proeet,c}({\mathcal H}^d_C,\Q_p(i-d))\to {\rm Sp}_{2d-i}(\Q_p)\to 0\notag
\end{align}
Here ${\rm Sp}_i(\Q_p)$ denotes the generalized locally constant Steinberg $\Q_p$-representation of ${\rm GL}_{d+1}(K)$ (see \cite[Sec.\,5.2.1]{CDN3} for a definition). Hence the terms on the right in \eqref{leje3-25} are nuclear Fr\'echet and of compact type over $\Q_p$, respectively. It follows that
\begin{align*}
 \uHom_{\rm TVS}({\mathbb H}^{2d-i}_{\proeet,c}({\mathcal H}^d_C,\Q_p(d)),\Q_p) & \simeq \uHom_{\rm TVS}({\rm Sp}_{i}(\Q_p)(i),\Q_p)\simeq {\rm Sp}_{i}(\Q_p)^*(-i),\\
 \underline{\Ext}^1_{\rm TVS}({\mathbb H}^{2d-i+1}_{\proeet,c}({\mathcal H}^d_C,\Q_p(d)),\Q_p) & \simeq \underline{\Ext}^1_{\rm TVS}( (H^d_c({\mathcal H}^d_C, \Omega^{d-i})/\ker d)(i-1),\Q_p)\\ & \simeq (H^d_c({\mathcal H}^d_C, \Omega^{d-i})/\ker d)^*(-i)\simeq
(\Omega^{i-1}({\mathcal H}^d_C)/\ker d)(-i).
\end{align*}
The last quasi-isomorphism uses Serre's duality (see \cite[Remark 8.10, (8.11)]{AGN} for details).

  Hence, in this example, the duality sequence from Corollary \ref{rainy-day} does transfer the compact support fundamental exact sequence into the usual fundamental exact sequence. 
  \begin{remark}
The case of affine spaces and tori is similar to the case of Drinfeld space 
but simpler since the Hyodo-Kato terms on the right in the fundamental exact sequences are actually finite dimensional over $\Q_p$. 
See \cite[Prop.\,4.17]{CDN3}, \cite[Cor.\,7.2, Sec.\,8.2.1]{AGN} for the shape of these fundamental exact sequences. 
  \end{remark}

{\Small
\printindex
}

\end{document}